\documentclass{article}
\usepackage{graphicx}
\usepackage[left=3cm, top=2cm, right=3cm, bottom=3cm]{geometry}
\usepackage{listings}
\usepackage{float}
\usepackage{amssymb,amsmath,mathtools}
\usepackage{enumerate}
\usepackage[numbers,square,sort]{natbib}
\usepackage{wrapfig}
\usepackage{framed}
\usepackage{tabstackengine}
\usepackage{braket}
\usepackage[nottoc]{tocbibind}\settocbibname{References}
\usepackage{mathrsfs}

\usepackage{indentfirst}
\usepackage{setspace}
\usepackage{bm}

\usepackage{hyperref}
\hypersetup{colorlinks,citecolor=black,linkcolor=black}

\usepackage{tikz}
\tikzset{every loop/.style={}}
\usetikzlibrary{arrows,shapes,decorations.markings,automata,backgrounds,petri,bending,calc,angles}
\usepackage{circuitikz}

\usepackage{fancyhdr}
\fancypagestyle{firststyle}
{
	
	\fancyfoot[l]{\footnotesize\textit{Date:} \today}
	\fancyfoot[c]{1}
	\fancyhead{}

}

\usepackage{tikz-cd} 
\tikzset{
    labl/.style={anchor=south, rotate=90, inner sep=.5mm}
}
\tikzcdset{row sep/normal=2.7em}
\tikzcdset{column sep/normal=3.6em}

\usepackage{enumitem}
\setlist[enumerate,1]{label=(\arabic*)}
\setlist[enumerate,2]{label={(\alph*)},ref={(\alph*)}}
\setlist[enumerate,3]{label={(\roman*)},ref={(\roman*)}}
\newlist{steplist}{enumerate}{1}
\setlist[steplist]{label={Step \arabic*:}, ref={Step \arabic*}}

\makeatletter
\newcommand*{\wackyenum}[1]{%
	\expandafter\@wackyenum\csname c@#1\endcsname%
}
\newcommand*{\@wackyenum}[1]{%
	$\ifcase#1\or(1')\or(2')\or(3a')\or(3b')\or(4')%
	\else\@ctrerr\fi$%
}
\AddEnumerateCounter{\wackyenum}{\@wackyenum}{(1')}
\makeatother

\usepackage[utf8]{inputenc}
\usepackage[english]{babel} 

\usepackage{amsthm}
\newtheorem{thm}{Theorem}[section]
\newtheorem{lem}[thm]{Lemma}
\newtheorem{prop}[thm]{Proposition}
\newtheorem{cor}[thm]{Corollary}
\numberwithin{equation}{section}

\newtheorem{que}[thm]{Question}

\theoremstyle{definition}
\newtheorem{defn}[thm]{Definition} 
\newtheorem{remk}[thm]{Remark}
\newtheorem{nota}[thm]{Notation}

\newcommand{\cA}{\mathcal{A}}

\newcommand{\cC}{\mathcal{C}}

\newcommand{\cG}{\mathcal{G}}

\newcommand{\cL}{\mathcal{L}}

\newcommand{\cR}{\mathcal{R}}

\newcommand{\Z}{\mathbb{Z}}

\newcommand{\G}{\Gamma}
\newcommand{\La}{\Lambda}

\newcommand{\Aut}{\operatorname{Aut}}

\newcommand{\rk}{\operatorname{rk}}
\newcommand{\Is}{\operatorname{Is}}

\usepackage{xparse}

\NewDocumentCommand{\dn}{e{^}}{%
	^{\IfValueT{#1}{#1}\vphantom{\smash[t]{\big|}}}
}

\newcommand{\uperp}{\dn^{\underline{\perp}}}
\newcommand{\jperp}{\dn^\perp}
\newcommand{\ut}{\underline{t}}

\makeatletter
\newsavebox{\@brx}
\newcommand{\llangle}[1][]{\savebox{\@brx}{\(\m@th{#1\langle}\)}%
	\mathopen{\copy\@brx\kern-0.5\wd\@brx\usebox{\@brx}}}
\newcommand{\rrangle}[1][]{\savebox{\@brx}{\(\m@th{#1\rangle}\)}%
	\mathclose{\copy\@brx\kern-0.5\wd\@brx\usebox{\@brx}}}
\makeatother

\begin{document}
\begin{center}
{\LARGE\bf
Commensurability of lattices in right-angled buildings}\\
\bigskip
\bigskip
{\large Sam Shepherd}
\end{center}
\bigskip

\begin{abstract}
	Let $\G$ be a graph product of finite groups, with finite underlying graph, and let $\Delta$ be the associated right-angled building.
	We prove that a uniform lattice $\La$ in the cubical automorphism group $\Aut(\Delta)$ is weakly commensurable to $\G$ if and only if all convex subgroups of $\La$ are separable.
	As a corollary, any two finite special cube complexes with universal cover $\Delta$ have a common finite cover.
	An important special case of our theorem is where $\G$ is a right-angled Coxeter group and $\Delta$ is the associated Davis complex.
	We also obtain an analogous result for right-angled Artin groups.
	In addition, we deduce quasi-isometric rigidity for the group $\G$ when $\Delta$ has the structure of a Fuchsian building.
\end{abstract}
\bigskip
\tableofcontents

\bigskip
\section{Introduction}

Given compact length spaces $X_1$ and $X_2$ with a common universal cover $\tilde{X}$, it is natural to ask whether $X_1$ and $X_2$ have a common finite cover.
The deck transformation groups of $\tilde{X}\to X_1, X_2$ are \emph{uniform lattices} $\G_1,\G_2$ in the isometry group $\Is(\tilde{X})$ (i.e. they act properly and cocompactly on $\tilde{X}$), and the existence of a common finite cover is equivalent to $\G_1$ and $\G_2$ being \emph{weakly commensurable in $\Is(\tilde{X})$} -- meaning that there exists $g\in\Is(\tilde{X})$ such that $g\G_1 g^{-1}$ is \emph{commensurable} to $\G_2$ (i.e. $g\G_1 g^{-1}\cap\G_2$ has finite index in both $g\G_1 g^{-1}$ and $\G_2$).
Alternatively, one could start with a locally compact cocompact length space $X$ and uniform lattices $\G_1,\G_2<\Is(X)$, and ask whether $\G_1$ and $\G_2$ are weakly commensurable.
In particular, one could ask if there is an algebraic property of the lattices that guarantees weak commensurability.
If $X=\mathbb{H}^2$ is the hyperbolic plane then it can happen that $\G_1$ and $\G_2$ are not weakly commensurable even if the lattices are isomorphic as groups (for instance using the uncountability of the moduli space of hyperbolic surfaces of a given genus).
However, if $X$ is a symmetric space associated to a semisimple Lie group with no compact factors and trivial center that is not locally isomorphic to SL$(2,\mathbb{R})$, then Mostow Rigidity \cite{Mostow73} tells us that $\G_1$ is conjugate to $\G_2$ if and only if the lattices are isomorphic; therefore $\G_1$ and $\G_2$ are weakly commensurable if and only if they are \emph{abstractly commensurable} (i.e. have isomorphic finite-index subgroups).
Nevertheless, there are still many examples in this setting where $\G_1$ and $\G_2$ are not abstractly commensurable but have similar algebraic properties -- for instance they could be a pair of uniform lattices in PU$(n,1)$ with isomorphic profinite completions \cite{Stover19}.

A very different setting is where $X$ is a locally finite cell complex, and we consider uniform lattices $\G_1,\G_2$ in its automorphism group $\Aut(X)$. If $X$ is a tree then Leighton's Theorem tells us that \emph{all} uniform lattices in $\Aut(X)$ are weakly commensurable \cite{Leighton82}.
There are various results for other cell complexes giving sufficient (and sometimes necessary) conditions for $\G_1$ and $\G_2$ to be weakly commensurable \cite{Haglund06,Huang18,Woodhouse21b,Shepherd22,BridsonShepherd22,Woodhouse23} -- we will discuss many of these later in the introduction.
For many of these results the conditions for $\G_1$ and $\G_2$ involve the existence of \emph{separable subgroups} (see Definition \ref{defn:separable}). Separable subgroups are natural to consider in this context, because if $\G_1$ and $\G_2$ have many separable subgroups then one can often replace them by finite-index sublattices with certain desired properties (usually involving the geometry of $X$).
On the other hand, lattices with a lack of separability properties provide a good source of counter-examples; for instance if $X$ is a product of trees then there exist some uniform lattices that are residually finite (e.g. a product of lattices) and others that are not \cite{Wise96,BurgerMozes97} -- residual finiteness is preserved by weak commensurability, so this gives examples where $\G_1$ and $\G_2$ are not weakly commensurable.

The setting of interest to us is as follows.
Let $\G=\G(\cG,(G_i)_{i\in I})$ be a graph product of finite groups, with finite underlying graph $\cG$, and let $\Delta=\Delta(\cG,(G_i)_{i\in I})$ be the associated right-angled building (see Section \ref{sec:graphprod} for definitions).
The cellular structure on $\Delta$ makes it a CAT(0) cube complex. Examples include products of trees and Davis complexes of right-angled Coxeter groups.
The finiteness assumption for the groups $G_i$ and graph $\cG$ ensures that $\Delta$ is locally finite.
Our main theorem is as follows.

\begin{thm}\label{thm:Delta}
	Let $\La<\Aut(\Delta)$ be a uniform lattice.
	Then $\La$ and $\G$ are weakly commensurable in $\Aut(\Delta)$ if and only if all convex subgroups of $\La$ are separable.
\end{thm}

Here $\Aut(\Delta)$ is the group of all cubical automorphisms of $\Delta$, not just the type-preserving automorphisms, so this theorem generalizes the work of Haglund \cite[Theorems 1.4 and 7.2]{Haglund06}.
By a \emph{convex subgroup} of $\La$ we mean a subgroup that stabilizes and acts cocompactly on a convex subcomplex of $\Delta$.
Some parts of the proof of Theorem \ref{thm:Delta} use similar arguments to \cite{Haglund06}, but the majority of the proof uses new ideas -- a summary of the proof strategy is given in Section \ref{subsec:strategy}.

\begin{remk}\label{remk:inthyp}
	The condition in Theorem \ref{thm:Delta} that all convex subgroups of $\La$ are separable can be replaced with the weaker condition that all finite-index subgroups of finite intersections of $\La$-stabilizers of hyperplanes are separable (see Proposition \ref{prop:separable} and the discussion preceding it).
	The same replacement can be made in Corollaries \ref{cor:Coxeter} and \ref{cor:Artin}.
\end{remk}

\subsection{Some corollaries}

An important concept in the theory of cube complexes is the notion of a group acting specially on a CAT(0) cube complex (Definition \ref{defn:specially}), due to Haglund and Wise \cite{HaglundWise08}.
The resulting theory has lead to many striking advancements in group theory and topology, including the resolution of the Virtual Haken Conjecture \cite{Agol13}.
At the heart of this theory is the fact that a group inherits various strong separability properties if it acts specially on a CAT(0) cube complex. In particular, if $X$ is a locally finite CAT(0) cube complex and $\La<\Aut(X)$ is a (virtually) special uniform lattice, then all convex subgroups of $\La$ are separable.
Special uniform lattices are thus good candidates for weak commensurability results. As a simple example, it follows from \cite{Leighton82,Wise06} that if $X$ is a product of trees then all special uniform lattices in $\Aut(X)$ are weakly commensurable.
Returning to the right-angled building $\Delta$, we show in Proposition \ref{prop:Gvspecial} that the uniform lattice $\G<\Aut(\Delta)$ is virtually special.
In particular, this implies that all convex subgroups of $\G$ are separable, so we deduce the ``only if'' direction of Theorem \ref{thm:Delta} (we note that the separability of convex subgroups of $\G$ is also proved in \cite{Haglund08}).
In addition, the above discussion implies the following corollary of Theorem \ref{thm:Delta}.

\begin{cor}\label{cor:speciallattices}
	All special uniform lattices in $\Aut(\Delta)$ are weakly commensurable.
\end{cor}

We remark that it is unknown whether all special uniform lattices in $\Aut(X)$ are weakly commensurable for $X$ an arbitrary locally finite CAT(0) cube complex -- this is an open problem of Haglund \cite[Problem 2.4]{Haglund06}.
One can also formulate Corollary \ref{cor:speciallattices} in terms of covering spaces as follows.

\begin{cor}
	Let $X_1$ and $X_2$ be finite special cube complexes with universal cover $\Delta$. Then $X_1$ and $X_2$ have a common finite cover.
\end{cor}

If the groups $G_i$ are all copies of $\Z/2\Z$ then the graph product $\G$ is called a \emph{right-angled Coxeter group} and the right-angled building $\Delta$ is called the \emph{Davis complex} associated to $\G$ \cite{Davis83,Moussong88}. We thus get another corollary of Theorem \ref{thm:Delta}, which in particular answers \cite[Problem 2.2]{Haglund06} in the affirmative.

\begin{cor}\label{cor:Coxeter}
	Let $W$ be a right-angled Coxeter group with associated Davis complex $X$, and let $\La<\Aut(X)$ be a uniform lattice.
	Then $\La$ and $W$ are weakly commensurable in $\Aut(X)$ if and only if all convex subgroups of $\La$ are separable.
\end{cor}

\begin{remk}\label{remk:otherDavis}
The Davis complex $X$ is sometimes defined as the CAT(0) cube complex whose 1-skeleton is the undirected Cayley graph of $W$ with respect to the standard generating set \cite[Proposition 7.3.4]{Davis08}; but taking the cubical subdivision recovers the definition of the Davis complex given above, so Corollary \ref{cor:Coxeter} holds for both versions of the Davis complex.
\end{remk}

Corollary \ref{cor:Coxeter} generalizes work of Woodhouse \cite{Woodhouse23}, who considers the case where the defining graph of $W$ is the 1-skeleton of a certain kind of Kneser complex.
Haglund also has a result similar to Corollary \ref{cor:Coxeter} \cite[Theorem 1.7]{Haglund06}, which applies to certain hyperbolic but non-right-angled Coxeter groups with two-dimensional Davis complexes.

Similar to the class of right-angled Coxeter groups is the class of right-angled Artin groups. A \emph{right-angled Artin group} $A$ is defined as the graph product of finitely many copies of $\Z$, and it is the fundamental group of a certain finite non-positively curved cube complex called the Salvetti complex \cite{Charney07}; in particular, the universal cover $X$ of the Salvetti complex is a CAT(0) cube complex admitting a proper cocompact action of $A$.
The cube complex $X$ is not the right-angled building associated to the graph product structure on $A$ (the latter is not even locally finite) but it does coincide with the Davis complex of a certain right-angled Coxeter group $W$ \cite{DavisJanuszkiewicz00} (this uses the notion of Davis complex from Remark \ref{remk:otherDavis}), and $A$ is commensurable to $W$ as a lattice in $\Aut(X)$. Therefore we obtain the following corollary of Corollary \ref{cor:Coxeter}.

\begin{cor}\label{cor:Artin}
	Let $A$ be a right-angled Artin group, let $X$ be the universal cover of the associated Salvetti complex, and let $\La<\Aut(X)$ be a uniform lattice.
	Then $\La$ and $A$ are weakly commensurable in $\Aut(X)$ if and only if all convex subgroups of $\La$ are separable.
\end{cor}

This generalizes a theorem of Huang \cite[Theorem 1.5]{Huang18}, who considers the case where the defining graph of $A$ is star-rigid and has no induced 4-cycles. However, Huang's result is stronger in one sense because it shows that \emph{all} uniform lattices in $\Aut(X)$ are weakly commensurable. It is natural to ask when this holds in general.

\begin{que}\label{que:allwc}
	For which right-angled buildings $\Delta$ is it true that all uniform lattices in $\Aut(\Delta)$ are weakly commensurable?
\end{que}

We do not have a complete solution, but we can answer this question for certain cases.
It follows from \cite{Meier96} that $\Delta$ is hyperbolic if and only if the defining graph $\cG$ contains no induced 4-cycle, and hyperbolic cubulated groups are virtually special by \cite{Agol13}, so Corollary \ref{cor:speciallattices} yields the following result.

\begin{cor}\label{cor:hyp}
	If the defining graph $\cG$ contains no induced 4-cycle, then all uniform lattices in $\Aut(\Delta)$ are weakly commensurable.
\end{cor}

We can also use Corollary \ref{cor:speciallattices} to get a positive answer to Question \ref{que:allwc} in some relatively hyperbolic cases.
Indeed, Caprace gives conditions for a Coxeter group to be hyperbolic relative to a collection of parabolic subgroups \cite{Caprace15}, and Genevois gives conditions for an arbitrary graph product of groups to be relatively hyperbolic \cite{Genevois17}. 
If the graph product $\G$ is hyperbolic relative to virtually abelian subgroups, then all uniform lattices in $\Aut(\Delta)$ will also be hyperbolic relative to virtually abelian subgroups by \cite[Theorem 1.6 and Remark 5.23]{DrutuSapir05}, so these lattices will be virtually special by \cite[Corollary 1.3]{Oregonreyes23}.
Putting this together, we get the following corollary.

\begin{cor}
	Let $W$ be a right-angled Coxeter group with defining graph $\cG$ and associated Davis complex $X$.
	Let $\mathfrak{J}$ denote the set of all induced 4-cycles in $\cG$.
	Suppose
	\begin{enumerate}
		\item for each $J\in \mathfrak{J}$, there is no vertex in $\cG-J$ adjacent to all the vertices in $J$, and
		\item for distinct $J_1,J_2\in\mathfrak{J}$, the intersection $J_1\cap J_2$ is either empty, a single vertex or a single edge.
	\end{enumerate}
Then $W$ is hyperbolic relative to the parabolic subgroups induced by $\mathfrak{J}$ (which are virtually $\Z^2$ subgroups), and all uniform lattices in $\Aut(X)$ are weakly commensurable.
\end{cor}

Additionally, we get a positive answer to Question \ref{que:allwc} if $\G$ has finite index in $\Aut(\Delta)$ as then all uniform lattices in $\Aut(\Delta)$ are commensurable, this happens for example if $\G$ is a right-angled Coxeter group with star-rigid defining graph (Remark \ref{remk:Autdiscrete}).
On the other hand, there are some cases where we get a negative answer to Question \ref{que:allwc} because there is a non-residually-finite uniform lattice in $\Aut(\Delta)$, for example if $\Delta$ is a product of trees or the cube complex associated to a right-angled Artin group whose defining graph contains an induced 4-cycle \cite[Theorem 1.8]{Huang18}. 

\subsection{Quasi-isometric rigidity}

We prove that the graph product $\G$ is quasi-isometrically rigid in certain cases with the following theorem. This requires the defining graph $\cG$ to be a \emph{generalized $m$-gon}, meaning that it is connected, bipartite, and has diameter $m$ and girth $2m$.

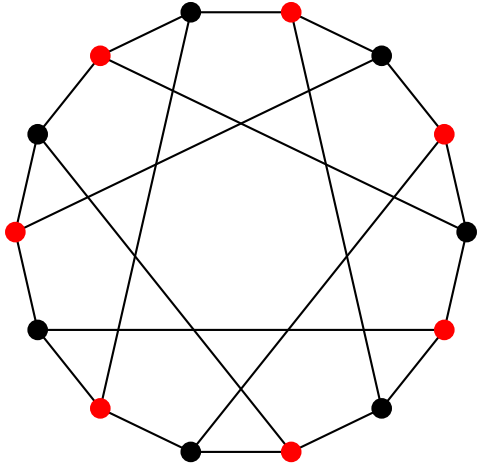
\begin{figure}[H]
	\centering
	\scalebox{1}{
		\begin{tikzpicture}[auto,node distance=2cm,
			thick,every node/.style={circle,draw,fill,inner sep=0pt,minimum size=7pt},
			every loop/.style={min distance=2cm},
			hull/.style={draw=none},
			]
			\tikzstyle{label}=[draw=none,fill=none]
			\tikzstyle{a}=[isosceles triangle,sloped,allow upside down,shift={(0,-.05)},minimum size=3pt]

\foreach \i in {0,1,...,6}
{	\draw ({3*cos(\i*360/7)},{3*sin(\i*360/7)})--({3*cos(\i*360/7+360/14)},{3*sin(\i*360/7+360/14)})--({3*cos(\i*360/7+360/7)},{3*sin(\i*360/7+360/7)});
	\draw ({3*cos(\i*360/7)},{3*sin(\i*360/7)})--({3*cos(\i*360/7+360*5/14)},{3*sin(\i*360/7+360*5/14)});
}

\foreach \i in {0,1,...,6}
{\node at ({3*cos(\i*360/7)},{3*sin(\i*360/7)}){};	
	\node[red] at ({3*cos(\i*360/7+360/14)},{3*sin(\i*360/7+360/14)}){};
}
			
		\end{tikzpicture}
	}
	\caption{An example of a 3-regular generalized 3-gon, known as the Heawood graph. This is a valid choice for $\cG$ in parts \ref{item:i} and \ref{item:ii} of Theorem \ref{thm:QI}.}\label{fig:Heawood}
\end{figure}

\begin{thm}\label{thm:QI}
	Let $\G=\G(\cG,(G_i)_{i\in I})$ be a graph product.
	Suppose that $\cG$ is a finite generalized $m$-gon, with $m\geq3$, which is bipartite with respect to the partition $I=I_1\sqcup I_2$.
	Suppose that $d_1,d_2,p_1,p_2\geq2$ are integers such that every $i\in I_k$ ($k=1,2$) has degree $d_k$ and $|G_i|=p_k$.
	Then in each of the following cases
	\begin{enumerate}[label=(\roman*)]
		\item\label{item:i} $2< d_1,d_2,p_1,p_2$,
		\item\label{item:ii} $p_1=p_2=2<d_1,d_2$,
		\item\label{item:iii} $d_1=d_2=2<p_1,p_2$,
	\end{enumerate}
any finitely generated group quasi-isometric to $\G$ is abstractly commensurable with $\G$.
\end{thm}

The source of this rigidity comes from the fact that the associated building $\Delta$ has the structure of a Fuchsian building in these cases. The result then follows by combining the quasi-isometric rigidity of Fuchsian buildings \cite{Xie06} with Corollary \ref{cor:hyp} (or with \cite{Haglund06} and \cite{Agol13} in cases \ref{item:ii} and \ref{item:iii}). The details are in Section \ref{sec:QI}.
Theorem \ref{thm:QI} provides examples of quasi-isometrically rigid hyperbolic groups with Menger curve boundary, answering a question of Dru\c{t}u--Kapovich \cite[Problem 25.20]{DrutuKapovich18}.

A key ingredient in showing that $\Delta$ has the structure of a Fuchsian building in Theorem \ref{thm:QI} is the following equivalent definition of a generalized $m$-gon: a connected, bipartite, simplicial graph with the following two properties:
\begin{enumerate}
	\item Given any pair of edges there is a circuit of length $2m$ containing both.
	\item For two circuits $A_1,A_2$ of length $2m$ with non-empty intersection, there is an isomorphism $f:A_1\to A_2$ that pointwise fixes $A_1\cap A_2$.
\end{enumerate}
See \cite[Theorem 1.1]{VanMaldeghem02} for the equivalence between the two definitions of generalized $m$-gon.
If every vertex has degree at least 3 then we say that the generalized polygon is \emph{thick}.
The above characterization of generalized polygons implies that a thick generalized polygon is simply a 1-dimensional spherical building.
The classification of generalized polygons reduces to the thick case because every non-thick generalized $m$-gon is either obtained from a thick generalized $m/k$-gon by subdividing each edge into $k$ edges, or it consists of two vertices joined by a collection of (at least two) embedded paths of length $m$ such that any pair of these paths intersect only at their endpoints \cite[Theorem 3.1]{VanMaldeghem02}. Finite thick generalized $m$-gons exist only for $m\in\{2,3,4,6,8\}$ \cite{FeitHigman64}, but there are infinitely many for each $m$. For example, a generalized 2-gon is just a complete bipartite graph (with at least two vertices in each set) and a thick generalized 3-gon is precisely the incidence graph of an abstract projective plane (such as in Figure \ref{fig:Heawood}). For any thick generalized polygon with vertex sets $I_1,I_2$, there are integers $d_1,d_2\geq3$ such that every vertex in $I_k$ has degree $d_k$ \cite[p29]{Ronan89}.

\subsection{Proof strategy}\label{subsec:strategy}

We utilize a number of well known structures on the right-angled building $\Delta$, many of which originate from the classical theory of Tits buildings.
For example, $\Delta$ is divided into a number of finite convex subcomplexes called \emph{chambers} (Definition \ref{defn:building}), and $\G$ acts simply transitively on the set of chambers.
Each vertex of $\Delta$ also has an associated \emph{rank} (Definition \ref{defn:rank}).
The cubical automorphism group $\Aut(\Delta)$ might not preserve chambers or rank, but we show that the subgroup  $\Aut_{\rk}(\Delta)$ of rank-preserving automorphisms does preserve the chamber structure, along with a number of other structures (Proposition \ref{prop:preserved}), and we show that $\Aut_{\rk}(\Delta)$ has finite index in $\Aut(\Delta)$ (Proposition \ref{prop:finiteindex}).
Most of our arguments involve these structures, so we always work with $\Aut_{\rk}(\Delta)$ instead of $\Aut(\Delta)$.

Another key definition in our argument is the notion of \emph{typed atlas} (Definition \ref{defn:typedatlas}), which can be thought of as a way of decorating $\Delta$ to make it more rigid. We show that $\Aut_{\rk}(\Delta)$ acts transitively on the set of typed atlases, and that the stabilizer of each typed atlas is a uniform lattice in $\Aut_{\rk}(\Delta)$ (Proposition \ref{prop:typedatlas}). 
Haglund has a similar argument involving atlases rather than typed atlases \cite[Proposition 6.5]{Haglund06}.
The ``if'' direction of Theorem \ref{thm:Delta} then reduces to finding a pair of typed atlases that are (virtually) preserved by $\La$ and $\G$ respectively. 
Finding such (typed) atlases is almost immediate if one works with lattices in the type-preserving automorphism group of $\Delta$, which is the setting considered by Haglund \cite{Haglund06}, but it is much harder in our setting.

Constructing a typed atlas that is preserved by $\G$ is relatively easy, but finding one that is virtually preserved by $\La$ requires considerable work. We first build a groupoid that consists of one cubical isomorphism between each pair of chambers in $\Delta$. This groupoid must satisfy several other properties, including being invariant under the action of some finite-index subgroup $\La'<\La$. We then construct a typed atlas by specifying some decorations on a single chamber $C$ and transporting these decorations to all other chambers using the groupoid.
This typed atlas may not be preserved by $\La$ or even $\La'$, but it will be preserved by the kernel of a certain holonomy map $\Upsilon:\La'\to\Aut(C)$, where $\Upsilon(\lambda)$ is defined by mapping $C$ onto another chamber using $\lambda$ and then mapping back to $C$ using the groupoid (Lemma \ref{lem:Latypedatlas}).

The core of the paper is devoted to building the $\La'$-invariant groupoid.
The strategy is to build a hierarchy of groupoids, each defined on a certain subset of chambers called a \emph{chamber-residue} (Definition \ref{defn:chamberres}). Roughly, the chamber-residues correspond to subgroups of $\G$ obtained by restricting to a subgraph of the defining graph (and also the cosets of such subgroups).
The hierarchy starts with finite chamber-residues corresponding to spherical subgroups of $\G$, and we define groupoids on these using the actions of the spherical subgroups. The groupoid we need, defined on the set of all chambers, is obtained at the final step of the hierarchy.
The groupoids built at each step of the hierarchy must be equivariant with respect to some finite-index subgroup of $\La$, which necessitates using holonomy maps for chamber-residue stabilizers (these are similar to the holonomy map described in the previous paragraph but with $\Aut(C)$ replaced by the symmetry group of a certain finite set of groupoids).
This is where we use the separability of convex subgroups of $\La$.
These arguments have parallels in the work of Woodhouse \cite{Woodhouse23}, who proved a result similar to Theorem \ref{thm:Delta} for a certain class of right-angled Coxeter groups, but one big difference is that Woodhouse's hierarchy only has two levels, one level corresponding to hyperplanes and another corresponding to the whole of $\Delta$, whereas our hierarchy is more complex.

\subsection{Structure of the paper}

Sections \ref{sec:graphprod} and \ref{sec:galleries} review some standard definitions and lemmas regarding graph products and right-angled buildings, including the notions of chambers, galleries and chamber-residues.
Section \ref{sec:rank} introduces the rank and poset structures on $\Delta$, as well as a new notion called \emph{level-equivalence}, which we use to prove propositions about the rank-preserving automorphism group $\Aut_{\rk}(\Delta)$.
Section \ref{sec:hyperplanes} studies the hyperplanes in $\Delta$, and relates them to the level-equivalence classes. We also prove that $\G$ is a virtually special lattice in $\Aut(\Delta)$, from which we deduce the ``only if'' direction in Theorem \ref{thm:Delta}.
Section \ref{sec:resgroup} introduces the notion of \emph{residue-groupoids}, which are the groupoids we need for the hierarchy discussed in Section \ref{subsec:strategy}.
Sections \ref{sec:hierclasses}--\ref{sec:atlas} prove the ``if'' direction in Theorem \ref{thm:Delta}.
Section \ref{sec:hierclasses} lays the groundwork for the hierarchy of residue-groupoids by defining a hierarchy of level-equivalence classes.
Section \ref{sec:hierresgroup} constructs the hierarchy of residue-groupoids.
Section \ref{sec:atlas} defines the notion of typed atlases, and uses them to complete the proof that $\La$ and $\G$ are weakly commensurable.
Finally, Section \ref{sec:QI} proves Theorem \ref{thm:QI}, identifying several cases where $\Delta$ has the structure of a Fuchsian building, and deducing that $\G$ is quasi-isometrically rigid in these cases.

\textbf{Acknowledgements:}\,
I thank the referee for their careful reading and helpful comments.
I am also grateful for the comments and suggestions of Jingyin Huang, Martin Bridson, Eduardo Oregón-Reyes and Daniel Woodhouse.

\bigskip
\section{Graph products and right-angled buildings}\label{sec:graphprod}

In this section we establish some basic definitions and lemmas, including the definitions of graph product and the associated right-angled building. Most of this content is standard and appears in \cite{Haglund06}.
First we need the notion of cubical cones.

\begin{nota}
For $N$ a simplicial complex, we let $\bar{N}$ denote the poset of simplices of $N$, together with the empty set, ordered by inclusion. 
We regard $\bar{N}$ combinatorially, so an element of $\bar{N}$ is a subset of the vertex set of $N$ that is either empty or spans a simplex.
\end{nota}

\begin{defn}(Cubical cones)\\\label{defn:cubecone}
	The \emph{cubical cone} $C(N)$ of a simplicial complex $N$ with vertex set $I$ is the cube complex constructed as follows.
	We have a bijection $t$ from the vertex set of $C(N)$ to $\bar{N}$, called the \emph{typing map}.
	We refer to $t(v)$ as the \emph{type} of the vertex $v$, and the vertex of type $\emptyset$ is called the \emph{center} of $C(N)$.
	We have an edge in $C(N)$ joining vertices $v_1,v_2$ whenever $t(v_1)\cup\{i\}=t(v_2)$ for some $i\in I$, so the 1-skeleton of $C(N)$ corresponds to the Hasse diagram of $\bar{N}$.
	Whenever $J_1\subset J_2\in\bar{N}$, the induced subgraph in the 1-skeleton of $C(N)$ spanned by vertices $v$ with $J_1\subset t(v)\subset J_2$ is isomorphic to the 1-skeleton of a cube of dimension $|J_2-J_1|$; and each face of this cube corresponds to sets $J_1\subset J'_1\subset J'_2\subset J_2$ and vertices $v$ with $J'_1\subset t(v)\subset J'_2$.
	We can therefore define the cubes of $C(N)$ to be in correspondence with pairs of nested sets $J_1\subset J_2\in\bar{N}$.
	If $Q$ is the cube corresponding to $J_1\subset J_2$ then we define $\ut(Q):=J_1$.
	For any point $p\in C(N)$, there will be a unique cube $Q$ containing $p$ in its interior, and we define $\ut(p):=\ut(Q)$.
\end{defn}

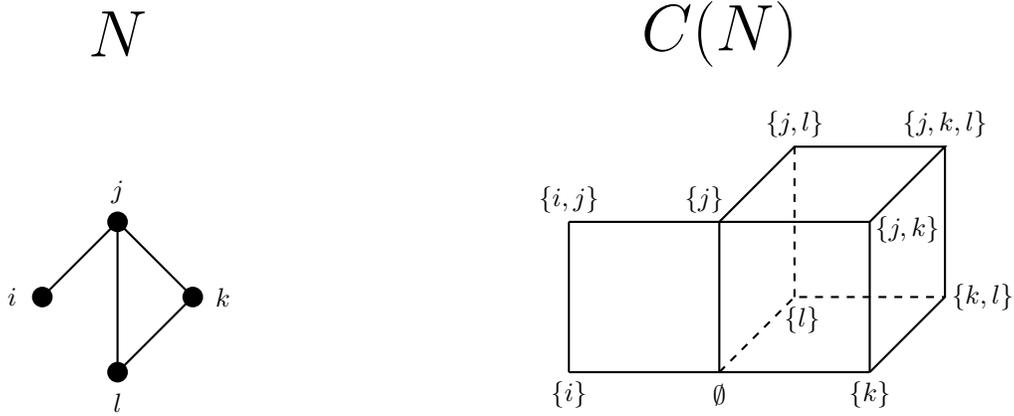
\begin{figure}[H]
	\centering
	\scalebox{1}{
		\begin{tikzpicture}[auto,node distance=2cm,
			thick,every node/.style={circle,draw,fill,inner sep=0pt,minimum size=7pt},
			every loop/.style={min distance=2cm},
			hull/.style={draw=none},
			]
			\tikzstyle{label}=[draw=none,fill=none]
			\tikzstyle{a}=[isosceles triangle,sloped,allow upside down,shift={(0,-.05)},minimum size=3pt]

			\begin{scope}[shift={(-6,1)}]
				\node at (-1,0) {};
				\node at (1,0) {};
				\node at (0,1) {};
				\node at (0,-1) {};
				\draw (-1,0)--(0,1)--(1,0)--(0,-1)--(0,1);
				\node[label] at (-1.4,0) {$i$};
				\node[label] at (0,1.4) {$j$};
				\node[label] at (1.4,0) {$k$};
				\node[label] at (0,-1.4) {$l$};		
				\node[label,font=\Huge] at (0,3.5) {$N$};		
			\end{scope}

			\begin{scope}[shift={(2,0)}]
				\draw[step=2] (-2,0) grid (2,2);
				\draw(0,2)--(1,3)--(3,3)--(2,2);
				\draw(3,3)--(3,1)--(2,0);
				\draw[dashed] (0,0)--(1,1)--(3,1);
				\draw[dashed] (1,1)--(1,3);

				\node[label] at (0,-.3) {$\emptyset$};
				\node[label] at (-2,-.3) {$\{i\}$};
				\node[label] at (-2,2.3) {$\{i,j\}$};
				\node[label] at (-.2,2.3) {$\{j\}$};
				\node[label] at (2.5,1.9) {$\{j,k\}$};
				\node[label] at (2,-.3) {$\{k\}$};
				\node[label] at (1.1,.7) {$\{l\}$};	
				\node[label] at (1,3.3) {$\{j,l\}$};
				\node[label] at (3,3.3) {$\{j,k,l\}$};
				\node[label] at (3.5,1) {$\{k,l\}$};		
				\node[label,font=\Huge] at (0,4.5) {$C(N)$};			
			\end{scope}
			
		\end{tikzpicture}
	}
	\caption{An example of a simplicial complex $N$ and corresponding cubical cone $C(N)$.}\label{fig:ccone}
\end{figure}

\begin{lem}	\cite[Lemma 3.5]{Haglund06}\\\label{lem:cubecone}
	The cubical cone $C(N)$ is a CAT(0) cube complex, and the link of its center is isomorphic to $N$.
\end{lem}

We can now define graph products and their associated right-angled buildings.

\begin{defn}(Graph products)\\
	Let $\cG$ denote a simplicial graph with vertex set $I$. Suppose that for each $i\in I$ we are given a non-trivial group $G_i$.
	We define the \emph{graph product of $(G_i)_{i\in I}$ along $\cG$} to be the group 
	$$\G=\G(\cG,(G_i)_{i\in I}):=\frac{*_{i\in I} G_i}{\llangle g_i g_j g_i^{-1} g_j^{-1}\mid i,j\in I\text{ adjacent, }g_i\in G_i,\,g_j\in G_j\rrangle},$$
	where $i,j\in I$ \emph{adjacent} means that $i,j$ are distinct vertices joined by an edge in $\cG$.
	
	Let $\G_i<\G$ be the image of the morphism $G_i\to\G$, and for $J\subset I$ we denote by $\G_J$ the subgroup of $\G$ generated by $(\G_j)_{j\in J}$, and we denote by $\cG_J$ the subgraph of $\cG$ induced by $J$.
	We say that $J\subset I$ is \emph{spherical} if any two distinct vertices of $J$ are adjacent.
	
	For $i\in I$ we let $i\jperp$ denote the set of all $j\in I$ adjacent to $i$, and we let $i\uperp:=\{i\}\cup i\jperp$.
	Similarly, for $J\subset I$ spherical we let 
	\begin{equation*}
J\uperp:=\bigcap_{j\in J} j\uperp\quad\text{and}\quad J\jperp:=J\uperp - J.
	\end{equation*}
By convention we let $\emptyset\uperp=\emptyset\jperp=I$.
	Note that we always have $J\subset J\uperp$.
\end{defn}

\begin{lem}\label{lem:GammaJ}
	The following statements hold:
	\begin{enumerate}
		\item\label{item:GiGi} For each $i\in I$ the morphism $G_i\to\G$ is injective, so we will identify the groups $G_i$ and $\G_i$.
		\item\label{item:GJ} For $J\subset I$ the natural map $*_{j\in J} G_j\to\Gamma$ induces an isomorphism $\G(\cG_J,(G_j)_{j\in J})\to\G_J$.
		\item\label{item:GJ1J2} $\G_{J_1}\cap\G_{J_2}=\G_{J_1\cap J_2}$ for $J_1,J_2\subset I$.
	\end{enumerate}	
\end{lem}
\begin{proof}
	For $J\subset I$ we have a morphism from $*_{i\in I}G_i\to *_{j\in J} G_j$ that is the identity on each $G_j$ for $j\in J$ and kills each $G_i$ for $i\notin J$. This descends to a morphism $\rho_J:\G\to\G(\cG_J,(G_j)_{j\in J})$.
	The natural map $\G(\cG_J,(G_j)_{j\in J})\to\G$ postcomposed with $\rho_J$ is the identity, so this proves \ref{item:GJ}. \ref{item:GiGi} is a special case of \ref{item:GJ}.
	We may now consider $\rho_J$ as a group retraction $\rho_J:\G\to\G_J$.
	
	For $J_1,J_2\subset I$, observe that $\rho_{J_1}\circ\rho_{J_2}$ has image in $\G_{J_1\cap J_2}$ (for example by considering the image of a product of elements in the $G_i$). But $\G_{J_1\cap J_2}$ is a subgroup of $\G_{J_1}\cap\G_{J_2}$, and $\rho_{J_1}\circ\rho_{J_2}$ is the identity on $\G_{J_1}\cap\G_{J_2}$, hence $\G_{J_1}\cap\G_{J_2}=\G_{J_1\cap J_2}$.
\end{proof}

\begin{defn}(Right-angled building of a graph product)\\\label{defn:building}
	Let $\G=\G(\cG,(G_i)_{i\in I})$ be a graph product, let $N=N(\cG)$ be the flag completion of $\cG$, and let $C(N)$ be the cubical cone of $N$.
	Recall from Lemma \ref{lem:cubecone} that we have a map $\ut:C(N)\to\bar{N}$, and in this setting $\bar{N}$ is the set of spherical subsets of $I$.
	Consider the equivalence relation on $\G\times C(N)$ defined by
	$$(\gamma,p)\sim(\gamma',p')\quad\Leftrightarrow\quad p=p'\text{ and }\gamma^{-1}\gamma'\in\G_J\text{ for }J=\ut(p).$$
	The \emph{right-angled building of $\G$} is the quotient $\Delta=\Delta(\cG,(G_i)_{i\in I}):=\G\times C(N)/\sim$.
	We denote the equivalence class of $(\gamma,p)$ by $[\gamma,p]$.
	
	For $\gamma\in\Gamma$, the image of the inclusion $\{\gamma\}\times C(N)\xhookrightarrow{}\Delta$ is called a \emph{chamber of $\Delta$}, denoted $C_\gamma$. 
	Letting $v_N$ denote the center of $C(N)$, we call $[\gamma,v_N]$ the \emph{center of $C_\gamma$}.
	We denote the base chamber $C_{1_\G}$ by $C_*$ and the set of all chambers by $\cC(\Delta)$.	
	
	We define the \emph{typing map} $t:\Delta^0\to\bar{N}$ (where $\Delta^0$ is the vertex set of $\Delta$) to be the natural extension of the typing map from Definition \ref{defn:cubecone}.
	Again we refer to $t(v)$ as the \emph{type} of the vertex $v$.
\end{defn}

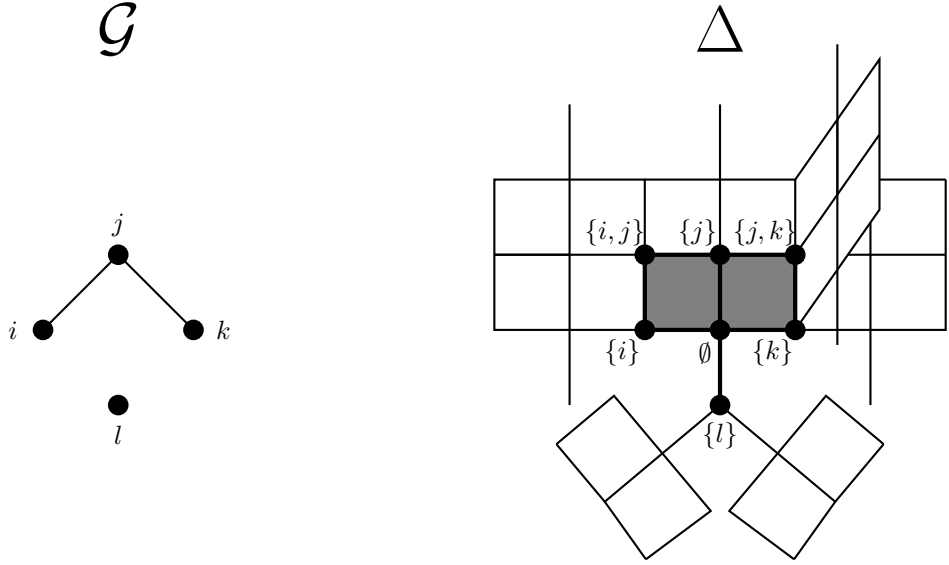
\begin{figure}[H]
	\centering
	\scalebox{1}{
		\begin{tikzpicture}[auto,node distance=2cm,
			thick,every node/.style={circle,draw,fill,inner sep=0pt,minimum size=7pt},
			every loop/.style={min distance=2cm},
			hull/.style={draw=none},
			]
			\tikzstyle{label}=[draw=none,fill=none]
			\tikzstyle{a}=[isosceles triangle,sloped,allow upside down,shift={(0,-.05)},minimum size=3pt]

			\begin{scope}[shift={(-6,0)}]
				\node at (-1,0) {};
				\node at (1,0) {};
				\node at (0,1) {};
				\node at (0,-1) {};
				\draw (-1,0)--(0,1)--(1,0);
				\node[label] at (-1.4,0) {$i$};
				\node[label] at (0,1.4) {$j$};
				\node[label] at (1.4,0) {$k$};
				\node[label] at (0,-1.4) {$l$};		
				\node[label,font=\Huge] at (0,4) {$\cG$};		
			\end{scope}

			\begin{scope}[shift={(2,0)}]
				\draw[fill=gray](-1,0)--(1,0)--(1,1)--(-1,1)--(-1,0);
				
				\draw(-2,0)-- (-3,0)--(-3,1);
				\draw(-2,0)-- (-2,1)--(-3,1);
				\draw(-2,2)-- (-3,2)--(-3,1);
				\draw(-2,2)-- (-2,1)--(-3,1);
				\draw(-2,2)-- (-2,3);
				\draw(-2,0)-- (-2,-1);
				\draw(-2,1)-- (-1,1);
				\draw(-2,0)-- (-1,0)--(-1,1);
				\draw(-2,2)-- (-1,2)--(-1,1);
				
				\draw(0,0)--(0,1)--(-1,1);
				\draw(0,2)--(-1,2);
				\draw(0,2)--(0,1)--(1,1);
				\draw(0,2)--(0,3);
				\draw(0,2)--(1,2)--(1,1);
				\draw(0,0)--(1,0)--(1,1);
				\draw(0,0)--(0,-1);
				\draw(1.56,.8)--(1,0);
				\draw(1.56,2.8)--(1,2);
				\draw(1.56,2.8)--(1.56,1.8)--(1,1);
				\draw(1.56,2.8)--(1.56,1.8)--(2.12,2.6);
				\draw(1.56,.8)--(1.56,1.8);
				\draw(1.56,.8)--(2.12,1.6)--(2.12,2.6);
				\draw(1.56,.8)--(1.56,-.2);
				\draw(1.56,2.8)--(2.12,3.6)--(2.12,2.6);
				\draw(1.56,2.8)--(1.56,3.8);
				
				\draw(2,0)--(1,0);
				\draw(2,0)--(2,1)--(3,1);
				\draw(2,0)--(2,-1);
				\draw(2,0)--(3,0)--(3,1);
				\draw(3,2)--(3,1);
				\draw(2.12,2)--(3,2);
				\draw(1.7,1)--(2,1)--(2,1.429);
				
				\draw(-.766,-1.643)--(0,-1);
				\draw(-.766,-1.643)--(-.123,-2.409)--(-.985,-3.052);
				\draw(-.766,-1.643)--(-1.532,-2.286)--(-.985,-3.052);
				\draw(-.766,-1.643)--(-1.409,-.877)--(-2.175,-1.52);
				\draw(-1.532,-2.286)--(-2.175,-1.52);
				
				\draw(.766,-1.643)--(0,-1);
				\draw(.766,-1.643)--(.123,-2.409)--(.985,-3.052);
				\draw(.766,-1.643)--(1.532,-2.286)--(.985,-3.052);
				\draw(.766,-1.643)--(1.409,-.877)--(2.175,-1.52);
				\draw(1.532,-2.286)--(2.175,-1.52);

				
				\draw[ultra thick](0,0)--(0,-1);
				\draw[ultra thick](-1,0) grid (1,1);

				\node at (0,0){};
				\node at (-1,0){};
				\node at (-1,1){};
				\node at (0,1){};		
				\node at (1,1){};
				\node at (1,0){};
				\node at (0,-1){};
				
				\node[label] at (-.2,-.3) {$\emptyset$};
				\node[label] at (-1.3,-.3) {$\{i\}$};
				\node[label] at (-1.4,1.3) {$\{i,j\}$};
				\node[label] at (-.3,1.3) {$\{j\}$};
				\node[label] at (.6,1.3) {$\{j,k\}$};
				\node[label] at (.7,-.3) {$\{k\}$};
				\node[label] at (0,-1.4) {$\{l\}$};	
				\node[label,font=\Huge] at (0,4) {$\Delta$};			
			\end{scope}
			
		\end{tikzpicture}
	}
	\caption{An example of the graph $\cG$ and a section of the right-angled building $\Delta$. In this example the groups $G_i,G_j,G_k,G_l$ have orders $2,2,3,3$ respectively. One of the chambers is shown in bold, and its vertices are labeled by their types.}\label{fig:building}
\end{figure}

We need the following key lemma regarding the structure of intersections of chambers.

\begin{lem}\label{lem:chamberint}
	Let $\gamma_1,\gamma_2\in\G$. Then $C_{\gamma_1}\cap C_{\gamma_2}$ is non-empty if and only if there exists $J\in\bar{N}$ with $\gamma_1^{-1}\gamma_2\in\G_J$.
	In this case, there is a unique minimal such $J$, and we have
	\begin{equation}\label{chamberintersect}
C_{\gamma_1}\cap C_{\gamma_2}=\{\gamma_1\}\times C_J=\{\gamma_2\}\times C_J, 
	\end{equation}	
	where $C_J\subset C(N)$ is the subcomplex defined by
	$$C_J:=\{p\in C(N)\mid J\subset \ut(p)\}.$$
\end{lem}
\begin{proof}
	If $J\in\bar{N}$ with $\gamma_1^{-1}\gamma_2\in\G_J$ then $[\gamma_1,v]=[\gamma_2,v]\in C_{\gamma_1}\cap C_{\gamma_2}$ for $v\in C(N)$ the vertex with $t(v)=\ut(v)=J$.
Conversely, if $[\gamma,p]\in C_{\gamma_1}\cap C_{\gamma_2}$ then $\gamma_1^{-1}\gamma_2\in\G_{J}$ for $J=\ut(p)$.
Now $\G_J$ is the direct product of the subgroups $G_j$ for $j\in J$, so we may write $\gamma_1^{-1}\gamma_2=\prod_{j\in J}g_j$ with $g_j\in G_j$.
Shrinking $J$ if necessary, we may assume that the $g_j$ are non-trivial. We claim that $J\in\bar{N}$ is unique minimal with $\gamma_1^{-1}\gamma_2\in\G_{J}$.
Indeed for any $i\in I$ we have a homomorphism $\rho_i:\G\to G_i$ by killing $G_j$ for $i\neq j\in I$, and it is clear that $\rho_i(\G_{J'})=\{1\}$ if $i\notin J'$.
But we know that $\rho_j(\gamma_1^{-1}\gamma_2)=g_j\neq 1$ for $j\in J$, so $\gamma_1^{-1}\gamma_2\in\G_{J'}$ implies $J\subset J'$ as required.

We now show (\ref{chamberintersect}). Consider $p\in C(N)$ with $\ut(p)=J'$. We have that $[\gamma_1,p]=[\gamma_2,p]$ if and only if $\gamma_1^{-1}\gamma_2\in\G_{J'}$, but by the first part of the lemma this is equivalent to $J\subset J'$ -- i.e. $p\in C_J$.
\end{proof}

\begin{cor}\label{cor:cubestructure}
The cube complex structure on $C(N)$ induces a cube complex structure on $\Delta$.
\end{cor}

\begin{remk}\label{remk:chamberneigh}
If $C$ is a chamber with center $v$, then $C$ is the cubical neighborhood of $v$ -- i.e. the union of cubes in $\Delta$ that contain $v$. This is because $C(N)$ is the cubical neighborhood of $v_N$ (see Definition \ref{defn:cubecone}), and for $\gamma\in \G$ the equivalence class $[\gamma,v_N]$ is a singleton, so all cubes in $\Delta$ containing $[\gamma,v_N]$ must come from cubes in $\{\gamma\}\times C(N)$.
\end{remk}

We will also need the following three results.

\begin{lem}\cite[Lemma 4.4]{Haglund06}\\
	The left action of $\G$ on $\G\times C(N)$ induces an action of $\G$ on $\Delta$ by cubical automorphisms. 
	Moreover, the projection $\G\times C(N)\to C(N)$ induces a projection $\rho:\Delta\to C(N)$, invariant under the action of $\G$, and two points are in the same $\G$-orbit if and only if they have the same image under $\rho$.
	The typing map $t$ factors through $\rho$, so $t$ is also invariant under the action of $\G$.
\end{lem}

\begin{lem}\cite[Corollary 4.9]{Haglund06}\\\label{lem:lfinite}
	$\Delta$ is locally finite if and only if the graph $\cG$ and the groups $G_i$ are finite.
	In this case $\G$ is a uniform lattice in $\Aut(\Delta)$, the group of cubical automorphisms of $\Delta$.
\end{lem}

Finiteness of $\cG$ and the $G_i$ is an assumption of Theorem \ref{thm:Delta}, but we will not assume it in Sections \ref{sec:galleries}--\ref{sec:resgroup} since it is not necessary for most of the results therein. 

Finally, the following is proved in \cite{Davis98} and \cite{Meier96}.

\begin{prop}
	$\Delta$ is a CAT(0) cube complex.
\end{prop}

\bigskip
\section{Galleries and chamber-residues}\label{sec:galleries}

In this section we define adjacency of chambers, galleries and chamber-residues, and prove some basic lemmas.
Again most of this material is standard and appears in \cite{Haglund06}.

\begin{defn}(Adjacent chambers)\\\label{defn:adjacent}
	We say that chambers $C_{\gamma_1},C_{\gamma_2}$ are \emph{$i$-adjacent} if $1\neq \gamma_1^{-1}\gamma_2\in G_i$. We say that $C_{\gamma_1},C_{\gamma_2}$ are \emph{adjacent} if they are $i$-adjacent for some $i\in I$.
	It follows from Definition \ref{defn:building} and Lemma \ref{lem:chamberint} that chambers $C,C'$ are $i$-adjacent if and only if $C\neq C'$ and there is a vertex $v\in C\cap C'$ of type $\{i\}$.
\end{defn}

\begin{defn}(Galleries)\\
	Let $C,C'\in\cC(\Delta)$ be chambers.
	A \emph{gallery of $\Delta$ joining $C$ and $C'$} is a sequence of chambers $(C_0,C_1,...,C_n)$ such that $C_0=C$, $C_n=C'$ and for each $1\leq k\leq n$ the chambers $C_{k-1},C_k$ are $i_k$-adjacent for some $i_k\in I$.
	For $J\subset I$ we say that $(C_0,C_1,...,C_n)$ is a \emph{$J$-gallery} if $i_k\in J$ for $1\leq k\leq n$.
	Note that an $\emptyset$-gallery consists of a single chamber.
	The product of two galleries $G\cdot G'$ is defined by concatenation, and the product of two $J$-galleries is again a $J$-gallery.
	
	The action of $\G$ on $\cC(\Delta)$ preserves $i$-adjacency for each $i\in I$, so we get a natural action of $\G$ on the set of galleries, which preserves the set of $J$-galleries for each $J\subset I$.	
\end{defn}

The following lemma follows straight from the definitions.

\begin{lem}\label{lem:alljoined}	
	The galleries joining $C_{\gamma}$ and $C_{\gamma'}$ correspond to product decompositions $\gamma^{-1}\gamma'=g_1g_2\cdots g_n$ for $g_k\in\sqcup_i G_i-\{1\}$.
	The corresponding gallery is $(C_{\gamma_0},C_{\gamma_1},...,C_{\gamma_n})$ where $\gamma_k:=\gamma g_1\cdots g_k$.
	In particular, any two chambers are joined by a gallery.
\end{lem}

We now describe moves that allow one to transform between any two galleries with the same end-chambers.

\begin{lem}\label{lem:moves}
	If two galleries have the same end-chambers then one can be transformed into the other by a sequence of the following three moves (or their inverses -- which we refer to by the same names):
	\begin{enumerate}[label={(M\arabic*)}]
		\item\label{M1} Replace $(C_0,C_1,...,C_n)$ with $(C_0,C_1,...,C_k,C',C_k,...,C_n)$, where $C_k,C'$ are adjacent.
		\item\label{M2} Replace $(C_0,C_1,...,C_n)$ with $(C_0,C_1,...,C_k,C',C_{k+1},...,C_n)$, where $C_k,C',C_{k+1}$ are pairwise $i$-adjacent for some $i\in I$.
		\item\label{M3} Replace $(C_0,C_1,...,C_n)$ with $(C_0,C_1,...,C_k,C',C_{k+2},...,C_n)$, where $C_k,C_{k+1}$ and $C',C_{k+2}$ are $i$-adjacent, and $C_k,C'$ and $C_{k+1},C_{k+2}$ are $j$-adjacent, for some adjacent $i,j\in I$.
	\end{enumerate}
\end{lem}
\begin{proof}
	If two words on $\sqcup_i G_i - \{1\}$ represent the same element of $\G$, then it follows from the presentation of $\G$ that the first word can be transformed into the second word by a sequence of the following three moves, or their inverses (see \cite{Green90}):
	\begin{enumerate}[label={(M\arabic*$'$)}]
		\item\label{i1} Replace $g_1\cdots g_n$ with $g_1\cdots g_kgg^{-1}g_{k+1}\cdots g_n$, where $g\in\sqcup_i G_i-\{1\}$.
		\item Replace $g_1\cdots g_n$ with $g_1\cdots g_k g'_1 g'_2 g_{k+2}\cdots g_n$, where $g_{k+1}=g'_1 g'_2$ and $g_{k+1},g'_1,g'_2\in G_i-\{1\}$ for some $i\in I$. 
		\item\label{i3} Replace $g_1\cdots g_n$ with $g_1\cdots g_k g_{k+2}g_{k+1}g_{k+3}\cdots g_n$, where $g_{k+1}\in G_i$ and $g_{k+2}\in G_j$ with $i,j\in I$ adjacent.
	\end{enumerate}
Given a gallery $(C_0,C_1,...,C_n)$, as in Lemma \ref{lem:alljoined} there exist $\gamma\in\G$ and $g_k\in\sqcup_i G_i-\{1\}$ such that $C_k=\gamma g_1\cdots g_k C_*$ for $0\leq k\leq n$. Moreover, the product $g_1\cdots g_n\in\G$ only depends on the end-chambers $C_0$ and $C_n$.
We conclude by the observation that transforming the word $g_1\cdots g_n$ using the moves \ref{i1}--\ref{i3} corresponds to transforming the gallery $(C_0,C_1,...,C_n)$ using the moves \ref{M1}--\ref{M3}.
\end{proof}

Galleries allow us to group chambers together into chamber-residues, defined as follows.

\begin{defn}(Chamber-residues)\\\label{defn:chamberres}
	Given a chamber $C\in\cC(\Delta)$ and $J\subset I$, the \emph{$J$-chamber-residue of $C$}, denoted $\cC(J,C)$, is the set of all chambers that appear in $J$-galleries based at $C$.
	(Note that some authors refer to $\cC(J,C)$ as just a $J$-residue, although Haglund defines $J$-residues to be slightly different yet related structures in \cite{Haglund06}.)
	Note that $\cC(\emptyset,C)=\{C\}$ and $\cC(I,C)=\cC(\Delta)$.
\end{defn}

We finish the section with a number of important results about the structure of chamber-residues.

\begin{lem}\label{lem:CJC}
	$\cC(J,C_\gamma)=\{C_{\gamma'}\mid \gamma'\in\gamma\G_J\}$
\end{lem}
\begin{proof}
	This follows from Lemma \ref{lem:alljoined}.
\end{proof}

\begin{lem}\label{lem:intchamres}
	$\cC(J_1,C)\cap\cC(J_2,C)=\cC(J_1\cap J_2,C)$ for $J_1,J_2\subset I$ and $C\in \cC(\Delta)$.
\end{lem}
\begin{proof}
	This follows from Lemmas \ref{lem:GammaJ}\ref{item:GJ1J2} and \ref{lem:CJC}.
\end{proof}

\begin{cor}
If $C_1,C_2\in\cC(J,C)$ are $i$-adjacent then $i\in J$.
\end{cor}

\begin{lem}\label{lem:cCv}
	Let $v\in\Delta^0$ be a vertex of type $J$.
	Then the set $\cC(v)$ of chambers containing $v$ is a $J$-chamber-residue.
\end{lem}
\begin{proof}
Let $C_{\gamma_1},C_{\gamma_2}\in\cC(v)$. Lemma \ref{lem:chamberint} tells us that there is $J'\in\bar{N}$ with $\gamma_1^{-1}\gamma_2\in\G_{J'}$ and $J'\subset \ut(v)=t(v)=J$. We deduce that $\gamma_1^{-1}\gamma_2\in\G_{J}$, so $C_{\gamma_2}\in \cC(J,C_{\gamma_1})$ by Lemma \ref{lem:CJC}. Conversely, for any other $C_{\gamma_3}\in \cC(J,C_{\gamma_1})$ we must have $\gamma_1^{-1}\gamma_3\in\G_J$, so $v\in C_{\gamma_3}$ by Lemma \ref{lem:chamberint}.
\end{proof}

\begin{lem}\label{lem:moveinJ}
	If $C,C'$ are two chambers in the same $J$-chamber-residue, then any two $J$-galleries joining $C$ and $C'$ differ by a sequence of moves \ref{M1}--\ref{M3} such that the intermediate galleries are also $J$-galleries.
\end{lem}
\begin{proof}
	By Lemma \ref{lem:alljoined}, a $J$-gallery joining $C$ and $C'$ corresponds to a word on $\sqcup_{j\in J} G_j-\{1\}$. Such a word is an element of $\G_J$, which is isomorphic to the graph product $\G(\cG_J,(G_j)_{j\in J})$ by Lemma \ref{lem:GammaJ}, so any two such words differ by a sequence of moves \ref{i1}--\ref{i3} such that all intermediate words also have letters in $\sqcup_{j\in J} G_j-\{1\}$. And such a sequence of moves \ref{i1}--\ref{i3} corresponds to a sequence of moves \ref{M1}--\ref{M3} such that the intermediate galleries are $J$-galleries.
\end{proof}

\begin{lem}(Product structure of $J\uperp$-chamber-residues for $J$ spherical)\\\label{lem:product}
	Let $C\in\cC(\Delta)$ and $J\subset I$ spherical.
	Then there is a bijection
	$$\beta:\cC(J,C)\times\cC(J\jperp,C)\to\cC(J\uperp,C),$$
	such that:
	\begin{enumerate}
		\item\label{item:betasections} $\beta(C_1,C)=C_1$ for all $C_1\in \cC(J,C)$ and $\beta(C,C_2)=C_2$ for all $C_2\in \cC(J\jperp,C)$.
		
		\item\label{item:iadjbeta} $\beta(C_1,C_2),\beta(C'_1,C'_2)$ are $i$-adjacent if and only if
		\begin{enumerate}
			\item\label{item:adj1} $C_1,C'_1$ are $i$-adjacent and $C_2=C'_2$, or
			\item\label{item:adj2} $C_2,C'_2$ are $i$-adjacent and $C_1=C'_1$.
		\end{enumerate}
		
		\item\label{item:betares} $\beta(\cC(J,C)\times\{C_2\})=\cC(J,C_2)$ for all $C_2\in\cC(J\jperp,C)$ and $\beta(\{C_1\}\times\cC(J\jperp,C))=\cC(J\jperp,C_1)$ for all $C_1\in\cC(J,C)$.
	\end{enumerate}
\end{lem}
\begin{proof}
	Translating by $\G$ we may assume that $C=C_*$.
	We have a product splitting $\G_{J\uperp}=\G_J\times \G_{J\jperp}$,
	so, using Lemma \ref{lem:CJC}, we may define $\beta$ by $\beta(C_{\gamma_1},C_{\gamma_2}):=C_{\gamma_1\gamma_2}$ for $\gamma_1\in\G_J$ and $\gamma_2\in\G_{J\jperp}$.
	Properties \ref{item:betasections} and \ref{item:iadjbeta} follow straight from this definition, and, again using Lemma \ref{lem:CJC}, property \ref{item:betares} follows from the equations
	\begin{equation*}
\cC(J,C_{\gamma_2})=\{C_{\gamma_1\gamma_2}\mid\gamma_1\in \G_J\}\quad\text{for }\gamma_2\in\G_{J\jperp},
	\end{equation*} 
and
\begin{equation*}	
	\cC(J,C_{\gamma_1})=\{C_{\gamma_1\gamma_2}\mid\gamma_2\in \G_{J\jperp}\}\quad\text{for }\gamma_1\in\G_J.\qedhere
\end{equation*}
\end{proof}

\bigskip
\section{Rank, poset structure and level-equivalence}\label{sec:rank}

In this section we introduce the rank and poset structure for the vertices of $\Delta$, and relate these notions to the galleries, chamber-residues and typing map of the previous sections. Key to this relationship is a certain equivalence relation on the vertices of $\Delta$, called level-equivalence.
The rank and poset structures are standard, but level-equivalence is a notion new to this paper.
We also prove two important propositions about the group of rank-preserving automorphisms of $\Delta$.

\begin{defn}(Rank and poset structure)\\\label{defn:rank}
	The \emph{rank} of a vertex $v\in\Delta^0$ is defined by $\rk(v):=|t(v)|$.
	Note that centers of chambers are precisely the rank-0 vertices of $\Delta$.
	Recall that the vertices of $C(N)$ are in bijection with the poset $\bar{N}$, so we get a poset structure on $C(N)$. We extend this poset structure $\G$-equivariantly to $\Delta^0$: explicitly $u\leq v$ if $t(u)\subset t(v)$ and there is some chamber $C$ containing both $u$ and $v$.
\end{defn}

\begin{lem}\label{lem:min}
	The following statements hold:
	\begin{enumerate}
		\item\label{item:Q} For vertices $u,v\in\Delta^0$ we have $u\leq v$ if and only if $\rk(u)\leq\rk(v)$ and there is a cube $Q$ containing both $u$ and $v$ of dimension $\rk(v)-\rk(u)$. In this case the cube $Q$ contains a vertex of type $J$ for every $t(u)\subset J\subset t(v)$.
		
		\item\label{item:wCcapC'} If the intersection of two chambers $C\cap C'$ is non-empty, then there is $J\in\bar{N}$ such that a vertex $v\in C$ is contained in $C\cap C'$ if and only if $J\subset t(v)$.
		Hence the vertex $v\in C$ of type $J$ is the unique $\leq$-minimal vertex in $C\cap C'$ -- and we denote it by $v=\wedge(C, C')$ (since $\wedge$ denotes the meet in a poset).
		
		\item\label{item:stayC} Given vertices $u\leq v$, if $u$ is in a chamber $C$ then $v$ is also in $C$.

		\item\label{item:iadjwedge} Chambers $C,C'$ are $i$-adjacent if and only if $t(\wedge(C, C'))=\{i\}$.
		
		\item\label{item:adjwedge} Chambers $C,C'$ are adjacent if and only if $\rk(\wedge(C, C'))=1$.
	\end{enumerate}
\end{lem}
\begin{proof}
	We prove each statement in turn.
	\begin{enumerate}
		\item If $u,v$ are vertices in $C(N)$ with $t(u)\subset t(v)$, then recall from Definition \ref{defn:cubecone} that there is a unique cube $Q$ of dimension $|t(v)|-|t(u)|$ containing $u$ and $v$; moreover, $Q$ contains a vertex of type $J$ for every $t(u)\subset J\subset t(v)$.
		This concludes the proof of \ref{item:Q} since the action of $\G$ provides type-preserving isomorphisms from every chamber in $\Delta$ to $C(N)$.
				
		\item This follows from Lemma \ref{lem:chamberint}.
		
		\item Since $u\leq v$, we know that $u$ and $v$ are both contained in some chamber $C'$. We have $u\in C\cap C'$, so by \ref{item:wCcapC'} there exists $J\in\bar{N}$ such that a vertex $w\in C'$ is contained in $C\cap C'$ if and only if $J\subset t(w)$ -- in particular $J\subset t(u)$. Hence $J\subset t(u)\subset t(v)$ and $v\in C$.
		
		\item We have already seen that distinct chambers $C,C'$ are $i$-adjacent if and only if there is a vertex $v\in C\cap C'$ with $t(v)=\{i\}$ (Definition \ref{defn:adjacent}). As distinct chambers have distinct type-$\emptyset$ vertices (i.e. distinct centers), we conclude from \ref{item:wCcapC'} that $C,C'$ are $i$-adjacent if and only if $t(\wedge(C, C'))=\{i\}$.
		
		\item The rank-1 vertices are the vertices of type $\{i\}$ for some $i\in I$, so we deduce from \ref{item:iadjwedge} that chambers $C,C'$ are adjacent if and only if $\rk(\wedge(C, C'))=1$.\qedhere
	\end{enumerate}
\end{proof}

We can think of rank as dividing $\Delta^0$ into a number of levels. Edges in $\Delta^0$ always go between consecutive levels, but we also want a notion of adjacency for vertices in the same level.
The equivalence classes generated by this ``level-adjacency'' will play an important role throughout the paper, and we will see later that they follow certain intersections of hyperplane carriers (Lemma \ref{lem:approxhyp}).

\begin{defn}(Level-adjacency and level-equivalence)\\\label{defn:[v]}
	We say that vertices $v_1,v_2\in\Delta^0$ are \emph{level-adjacent} if they are contained in adjacent chambers $C_1,C_2$ respectively, with $v_1,v_2\notin C_1\cap C_2$, and with some vertex $v_1,v_2\leq u\in C_1\cap C_2$ such that $\rk(v_1)=\rk(v_2)=\rk(u)-1$. Let $\approx$ be the equivalence relation on $\Delta^0$ generated by level-adjacency, and let $[v]$ denote the $\approx$-class of a vertex $v$. We call $\approx$ \emph{level-equivalence}.
\end{defn}

\begin{figure}[H]
	\centering
	\scalebox{1}{
		\begin{tikzpicture}[auto,node distance=2cm,
			thick,every node/.style={circle,draw,fill,inner sep=0pt,minimum size=7pt},
			every loop/.style={min distance=2cm},
			hull/.style={draw=none},
			]
			\tikzstyle{label}=[draw=none,fill=none]
			\tikzstyle{a}=[isosceles triangle,sloped,allow upside down,shift={(0,-.05)},minimum size=3pt]

			\begin{scope}[shift={(-6,0)}]
				\node at (-1,0) {};
				\node at (1,0) {};
				\node at (0,1) {};
				\node at (0,-1) {};
				\draw (-1,0)--(0,1)--(1,0);
				\node[label] at (-1.4,0) {$i$};
				\node[label] at (0,1.4) {$j$};
				\node[label] at (1.4,0) {$k$};
				\node[label] at (0,-1.4) {$l$};		
				\node[label,font=\Huge] at (0,4) {$\cG$};		
			\end{scope}

			\begin{scope}[shift={(2,0)}]
				\draw[fill=gray](-1,0)--(1,0)--(1,1)--(-1,1)--(-1,0);
				
				\draw(-2,0)-- node[a]{} (-3,0)--node[a]{}(-3,1);
				\draw(-2,0)-- node[a]{}(-2,1)-- node[a]{}(-3,1);
				\draw(-2,2)-- node[a]{}(-3,2)-- node[a]{}(-3,1);
				\draw(-2,2)-- node[a]{}(-2,1)-- node[a]{}(-3,1);
				\draw(-2,2)-- node[a]{}(-2,3);
				\draw(-2,0)-- node[a]{}(-2,-1);
				\draw(-2,1)-- node[a]{}(-1,1);
				\draw(-2,0)-- node[a]{}(-1,0)-- node[a]{}(-1,1);
				\draw(-2,2)-- node[a]{}(-1,2)-- node[a]{}(-1,1);
				\draw(0,0)-- node[a]{}(-1,0);
				\draw(0,0)-- node[a]{}(0,1)-- node[a]{}(-1,1);
				\draw(0,2)-- node[a]{}(-1,2);
				\draw(0,2)-- node[a]{}(0,1)-- node[a]{}(1,1);
				\draw(0,2)-- node[a]{}(0,3);
				\draw(0,2)-- node[a]{}(1,2)-- node[a]{}(1,1);
				\draw(0,0)-- node[a]{}(1,0)-- node[a]{}(1,1);
				\draw(0,0)-- node[a]{}(0,-1);
				\draw(1.56,.8)-- node[a]{}(1,0);
				\draw(1.56,2.8)-- node[a]{}(1,2);
				\draw(1.56,2.8)-- node[a]{}(1.56,1.8)-- node[a]{}(1,1);
				\draw(1.56,2.8)-- node[a]{}(1.56,1.8)-- node[a]{}(2.12,2.6);
				\draw(1.56,.8)-- node[a]{}(1.56,1.8);
				\draw(1.56,.8)-- node[a]{}(2.12,1.6)-- node[a]{}(2.12,2.6);
				\draw(1.56,.8)-- node[a]{}(1.56,-.2);
				\draw(1.56,2.8)-- node[a]{}(2.12,3.6)-- node[a]{}(2.12,2.6);
				\draw(1.56,2.8)-- node[a]{}(1.56,3.8);
				
				\draw(2,0)-- node[a]{}(1,0);
				\draw(2,0)-- node[a]{}(2,1)-- node[a]{}(3,1);
				\draw(2,0)-- node[a]{}(2,-1);
				\draw(2,0)-- node[a]{}(3,0)-- node[a]{}(3,1);
				\draw(3,2)-- node[a]{}(3,1);
				\draw(2.12,2)-- node[isosceles triangle,sloped,allow upside down,shift={(-.13,-.05)},minimum size=3pt]{}(3,2);
				\draw(1.7,1)--(2,1)--(2,1.429);
				
				\draw(-.766,-1.643)-- node[a]{}(0,-1);
				\draw(-.766,-1.643)-- node[a]{}(-.123,-2.409)-- node[a]{}(-.985,-3.052);
				\draw(-.766,-1.643)-- node[a]{}(-1.532,-2.286)-- node[a]{}(-.985,-3.052);
				\draw(-.766,-1.643)-- node[a]{}(-1.409,-.877)-- node[a]{}(-2.175,-1.52);
				\draw(-1.532,-2.286)-- node[a]{}(-2.175,-1.52);
				
				\draw(.766,-1.643)-- node[a]{}(0,-1);
				\draw(.766,-1.643)-- node[a]{}(.123,-2.409)-- node[a]{}(.985,-3.052);
				\draw(.766,-1.643)-- node[a]{}(1.532,-2.286)-- node[a]{}(.985,-3.052);
				\draw(.766,-1.643)-- node[a]{}(1.409,-.877)-- node[a]{}(2.175,-1.52);
				\draw(1.532,-2.286)-- node[a]{}(2.175,-1.52);

				
				\draw[ultra thick](0,0)--(0,-1);
				\draw[ultra thick](-1,0) grid (1,1);
				
				\node[red] at (0,0){};
				\node[red] at (-2,0){};
				\node[red] at (-2,2){};
				\node[red] at (0,2){};
				\node[red] at (2,0){};
				\node[red] at (1.56,.8){};
				\node[red] at (1.56,2.8){};
				\node[red] at (-.766,-1.643){};
				\node[red] at (.766,-1.643){};
				\node[blue] at (-1,0){};
				\node[blue] at (-1,2){};
				\node[orange] at (-1,1){};
				\node[green] at (0,1){};
				\node[green] at (-2,1){};
				\node[green] at (1.56,1.8){};	
				\node[green] at (2,1){};			
				\node at (1,1){};
				\node at (1,0){};
				\node at (0,-1){};
				
				\node[label] at (-.2,-.3) {$\emptyset$};
				\node[label] at (-1.3,-.3) {$\{i\}$};
				\node[label] at (-1.4,1.3) {$\{i,j\}$};
				\node[label] at (-.3,1.3) {$\{j\}$};
				\node[label] at (.6,1.3) {$\{j,k\}$};
				\node[label] at (.7,-.3) {$\{k\}$};
				\node[label] at (0,-1.4) {$\{l\}$};	
				\node[label,font=\Huge] at (0,4) {$\Delta$};			
			\end{scope}
			
		\end{tikzpicture}
	}
	\caption{A repeat of Figure \ref{fig:building}, but with the edges oriented to point upwards in the poset structure, so they form the Hasse diagram for $(\Delta^0,\leq)$.
		Four examples of level-equivalence classes are depicted, with colors red, orange, green and blue.}\label{fig:leveladj}
\end{figure}
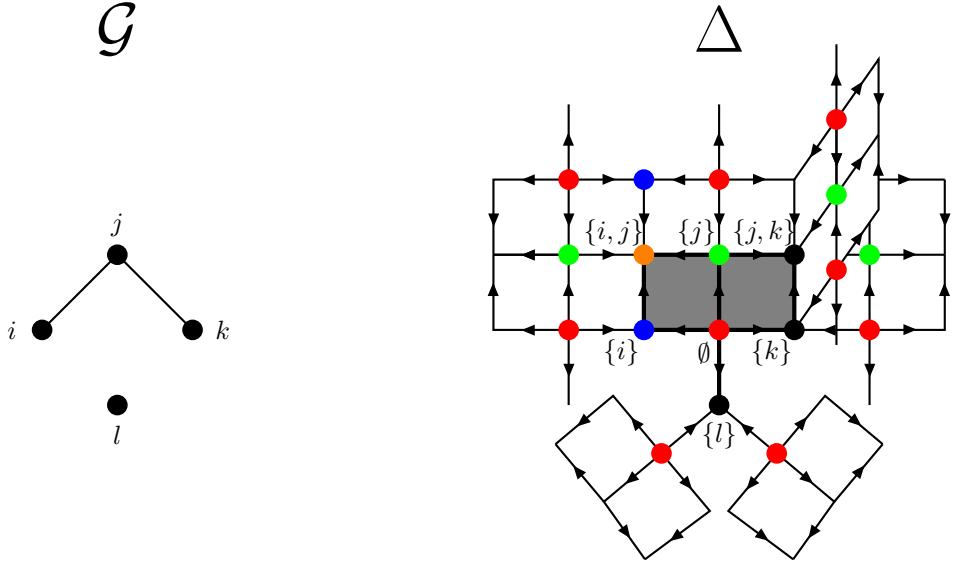

\begin{lem}\label{lem:leveladj}
	Vertices $v_1,v_2\in\Delta^0$ are level-adjacent if and only if they are of the same type and are contained in chambers $C_1,C_2$ respectively that are $i$-adjacent for some $i\in t(v_1)\jperp$.
\end{lem}
\begin{proof}
	First suppose $v_1,v_2$ are level-adjacent. Say they are contained in adjacent chambers $C_1,C_2$ respectively, with $v_1,v_2\notin C_1\cap C_2$, and with some vertex $v_1,v_2\leq u\in C_1\cap C_2$ such that $\rk(v_1)=\rk(v_2)=\rk(u)-1$.
	Let $t(u)=t(v_1)\cup\{i\}$. Note that $i\in t(v_1)\jperp$. By Lemma \ref{lem:min}\ref{item:wCcapC'}, $C_1,C_2$ must be $j$-adjacent for some $j\in t(v_1)\cup\{i\}$, and since $v_1\notin C_1\cap C_2$ we deduce that $j=i$. As $v_2\notin C_1\cap C_2$ we also have $t(v_2)=t(v_1)$.
	
	Conversely, suppose that $v_1,v_2$ are of the same type and are contained in chambers $C_1,C_2$ respectively that are $i$-adjacent for some $i\in t(v_1)\jperp$. Let $u\in C_1\cap C_2$ be the vertex of type $t(v_1)\cup\{i\}$ (Lemma \ref{lem:min}\ref{item:wCcapC'}). It follows that $v_1,v_2\leq u$ and $\rk(v_1)=\rk(v_2)=\rk(u)-1$. And applying Lemma \ref{lem:min}\ref{item:wCcapC'} again we see that $v_1,v_2\notin C_1\cap C_2$.
\end{proof}

\begin{cor}\label{cor:leveltype}
	Any two level-equivalent vertices are of the same type.
\end{cor}

\begin{remk}\label{remk:extreme}
	If $t(v)$ is a maximal spherical subset of $I$, then $v$ is $\leq$-maximal and the class $[v]=\{v\}$ is a singleton.
	At the opposite extreme, we note that two rank-0 vertices are level-adjacent if and only if their corresponding chambers are adjacent (by Lemma \ref{lem:leveladj}), so the set of all rank-0 vertices defines a single level-equivalence class.
\end{remk}

\begin{lem}\label{lem:cC[v]}
	Let $v$ be a vertex of type $J$ contained in a chamber $C$, and let $\cC([v])$ denote the set of chambers that contain a vertex in $[v]$. Then $\cC(J\uperp,C)=\cC([v])$.
	Moreover, we have a product decomposition
	$$\cC(J\uperp,C)\cong\cC(J,C)\times\cC(J\jperp,C)$$
	from Lemma \ref{lem:product} in which the sections $\cC(J,C)\times\{C_2\}$ correspond to the sets $\cC(v')$ for $v'\in[v]$ (which are $J$-chamber-residues), and the sections $\{C_1\}\times\cC(J\jperp,C)$ correspond to the $J\jperp$-chamber-residues in $\cC([v])$.
	In particular, the product decomposition would remain the same up to isomorphism if we chose a different level-equivalence class representative $v$ and chamber $C\in\cC(v)$.
\end{lem}
\begin{proof}
	We know that all vertices in $[v]$ are of type $J$ by Corollary \ref{cor:leveltype}.
	If $v_1\in[v]$ is in a chamber $C_1$ and if $C_1$ is $i$-adjacent to $C_2$ for $i\in J\jperp$, then the vertex $v_2\in C_2$ of type $J$ is level-adjacent to $v_1$ by Lemma \ref{lem:leveladj}.
	And for each $v'\in[v]$ we know that $\cC(v')$ is a $J$-chamber-residue by Lemma \ref{lem:cCv}.
	Combining these two facts, we see that $\cC(J\uperp,C)\subset\cC([v])$.
	
	Conversely, if $v_1,v_2\in[v]$ are level-adjacent, then they are contained in chambers $C_1,C_2$ respectively that are $i$-adjacent for some $i\in J\jperp$ (Lemma \ref{lem:leveladj} again). Combined with the fact that $\cC(v')$ is a $J$-chamber-residue for each $v'\in[v]$, we deduce that $\cC([v])$ is contained inside a $J\uperp$-chamber-residue, yielding the reverse inclusion $\cC([v])\subset\cC(J\uperp,C)$.
	
	The remainder of the lemma follows from Lemmas \ref{lem:cCv} and \ref{lem:product}.
\end{proof}

\begin{lem}\label{lem:approxjperp}
	If $v_1\approx v_2$ are of type $J$ and $C_1\in\cC(v_1)$, then the intersection $\cC(v_2)\cap\cC(J\jperp,C_1)$ consists of a single chamber.
\end{lem}
\begin{proof}
This follows from the fact that the sets $\cC(v_2)$ and $\cC(J\jperp,C_1)$ correspond to orthogonal sections in the product decomposition from Lemma \ref{lem:cC[v]}.
\end{proof}

We will need the following definition in later sections.

\begin{defn}(Sets $E^-(v)$ and lower degree)\\\label{defn:E-}
	For $v\in\Delta^0$, let $E^-(v)$ denote the set of edges incident to $v$ that join it to a vertex of lower rank, and let $d^-(v):=|E^-(v)|$ (which in general might be $\infty$). We call $d^-(v)$ the \emph{lower degree} of $v$. Note that $d^-(v)=0$ if and only if $\rk(v)=0$. In general, $d^-(v)$ can be computed by the formula $d^-(v)=\sum_{i\in t(v)}|G_i|$.
\end{defn}

Let $\Aut_{\rk}(\Delta)$ denote the group of cubical automorphisms of $\Delta$ that preserve ranks of vertices.
The following key proposition shows that $\Aut_{\rk}(\Delta)$ preserves many of the structures on $\Delta$ that we have defined so far.

\begin{prop}(Things preserved by $\Aut_{\rk}(\Delta)$)\\\label{prop:preserved}
	The group $\Aut_{\rk}(\Delta)$ preserves the following structures on $\Delta$:
	\begin{itemize}
		\item the poset structure on $\Delta$,
		\item lower degrees of vertices,
		\item the chambers,
		\item adjacency of chambers,
		\item galleries,
		\item level-equivalence of vertices,
		\item the families of sets $\{\cC(v)\mid v\in\Delta^0\}$ and $\{\cC([v])\mid v\in\Delta^0\}$,
		\item the product structures on the sets $\cC([v])$.
	\end{itemize} 
\end{prop}
\begin{proof}
	Lemma \ref{lem:min}\ref{item:Q} implies that $\Aut_{\rk}(\Delta)$ preserves the partial order $\leq$.
	Lower degrees of vertices are defined using the poset structure, so these are preserved as well.
	The centers of chambers are the rank-0 vertices, and each chamber is the cubical neighborhood of its center (Remark \ref{remk:chamberneigh}), so $\Aut_{\rk}(\Delta)$ also preserves the chamber structure on $\Delta$.
	Adjacency of chambers is preserved because of Lemma \ref{lem:min}\ref{item:adjwedge}, so galleries are preserved as well.
	Level-equivalence of vertices is defined using rank, chambers and the poset structures on $\Delta^0$, so this is also preserved -- along with the sets $\cC([v])$.
	The sets $\cC(v)$ are preserved because chambers are preserved.
	Moreover, we have that $g\cC(v)=\cC(gv)$ and $g\cC([v])=\cC([gv])$ for $v\in\Delta^0$ and $g\in\Aut_{\rk}(\Delta)$.
	Combining this observation with Lemma \ref{lem:cC[v]}, we deduce that an automorphism $g\in\Aut_{\rk}(\Delta)$ induces a bijection $g:\cC([v])\to\cC([gv])$ that corresponds to a bijection of products $$\cC(J,C)\times\cC(J\jperp,C)\to\cC(t(gv),gC)\times\cC(t(gv)\jperp,gC),$$
	and this bijection of products preserves the factors and their order.
\end{proof}

In the context of Theorem \ref{thm:Delta}, we would like the lattices $\G$ and $\La$ to both preserve all the structures from Proposition \ref{prop:preserved}.
This may not be true for $\La$, but it will hold for $\La\cap\Aut_{\rk}(\Delta)$, and the following proposition implies that this is a finite-index subgroup of $\La$.
As Theorem \ref{thm:Delta} is a statement about commensurability, there is no harm in replacing $\La$ by a finite-index subgroup, so in later sections we will be able to assume that $\La$ does preserve all the structures from Proposition \ref{prop:preserved}.

\begin{prop}\label{prop:finiteindex}
	If the graph $\cG$ is finite, then $\Aut_{\rk}(\Delta)$ has finite index in $\Aut(\Delta)$.
\end{prop}
\begin{proof}
	We define a second equivalence relation $\simeq$ on $\Delta^0$ generated by the relation $R$, where $uRv$ if $u,v$ are joined by an edge path of length two, but no cube contains both $u$ and $v$. (Equivalently, $u,v$ are at distance 2 in the $\ell_\infty$ metric.)
	Let $u$ be the center of a chamber $C$. 
	A vertex $w$ is adjacent to $u$ if and only if $w\in C$ and $\rk(w)=1$. A vertex $v\neq u$ is adjacent to $w$ if and only if it is the center of a chamber $C'$ adjacent to $C$ with $w=\wedge(C,C')$, or it is a rank-2 vertex in $C$. If $v$ is a rank-2 vertex in $C$ then there is a cube $Q$ in $C$ containing $u$ and $v$ (Lemma \ref{lem:min}\ref{item:Q}), so it follows from Lemma \ref{lem:min}\ref{item:adjwedge} that the vertices $v$ with $uRv$ are precisely the centers of chambers adjacent to $C$.
	We deduce that the $\simeq$-equivalence class of $u$ is the set of all rank-0 vertices.
	
	Adjacent vertices in $\Delta$ have ranks that differ by 1, so $\Aut_{\rk}(\Delta)$ is equal to the $\Aut(\Delta)$-stabilizer of the set of rank-0 vertices.
	Since $\Aut(\Delta)$ preserves the relation $R$ and the equivalence relation $\simeq$, we deduce that $\Aut_{\rk}(\Delta)$ is the group of $g\in\Aut(\Delta)$ such that $\rk(gv)=0$ for some rank-0 vertex $v$.
		
	Fix a chamber $C$ with center $v$.
	The subgroup $\G<\Aut_{\rk}(\Delta)$ acts transitively on $\cC(\Delta)$, and each chamber is the cubical neighborhood of its center, so for every $g\in\Aut(\Delta)$ there is $g'\in\G$ such that $g^{-1}v$ and $g'v$ are contained in a common cube. Hence $gg'v\in C$. So every left $\Aut_{\rk}(\Delta)$-coset contains an automorphism $g$ with $gv\in C$. The $\Aut(\Delta)$-stabilizer of $v$ is contained in $\Aut_{\rk}(\Delta)$, so automorphisms $g$ in distinct left $\Aut_{\rk}(\Delta)$-cosets have distinct images $gv$. 
	It follows that the index of $\Aut_{\rk}(\Delta)$ in $\Aut(\Delta)$ is bounded by $|C^0|$, which is finite since $\cG$ is finite.
\end{proof}

\begin{remk}\label{remk:Autdiscrete}
	If the graph $\cG$ is finite and \emph{star-rigid} -- meaning that the only automorphism of $\cG$ that pointwise fixes the star of a vertex is the identity -- and if the groups $(G_i)$ have order two, then $\G$ has finite index in $\Aut(\Delta)$. In particular $\Aut(\Delta)$ is a discrete automorphism group (i.e. it acts properly and cocompactly on $\Delta$).
	The group $\G$ acts transitively on the chambers, and $\Aut_{\rk}(\Delta)$ has finite index in $\Aut(\Delta)$ by Proposition \ref{prop:finiteindex}, so it suffices to show that $\Aut_{\rk}(\Delta)$ has finite chamber-stabilizers.
	The chambers are finite (as $\cG$ is finite), so it is enough to show that the identity is the only element of $\Aut_{\rk}(\Delta)$ that pointwise fixes a chamber.
	Let $g\in\Aut_{\rk}(\Delta)$ pointwise fix a chamber $C$.
	To show that $g$ is the identity it suffices to show that it pointwise fixes the chambers adjacent to $C$ (it follows that $g$ pointwise fixes all chambers by running the argument along galleries).
	Let $C'$ be $i$-adjacent to $C$ and let $v:=\wedge(C,C')$.
	We know that $g$ fixes $v$, and that $C'$ is the only chamber $i$-adjacent to $C$ (as $|G_i|=2$), so $g$ stabilizes $C'$ by Lemma \ref{lem:min}.
	For $j\in i\jperp$ and $u\in C'$ the vertex of type $\{j\}$, the vertex $w\in C\cap C'$ of type $\{i,j\}$ is fixed by $g$, and $u$ is the only rank-1 vertex in $C'-C$ with $u\leq w$, hence $g$ fixes $u$.
	The action of $g$ on $C'$ preserves ranks of vertices, and the typing map induces an isomorphism $C'\cong C(N(\cG))$, so the action of $g$ on $C'$ corresponds to an automorphism of $\cG$ that pointwise fixes the star $i\uperp$. But $\cG$ is star-rigid, thus this automorphism must be the identity on $\cG$, and $g$ must pointwise fix $C'$.
\end{remk}

\bigskip
\section{Hyperplanes and separable subgroups}\label{sec:hyperplanes}

In this section we recall the notion of hyperplanes in a CAT(0) cube complex, cast in terms of parallelism of edges, and we analyze the hyperplanes in $\Delta$.
We also recall the notion of a separable subgroup, and we relate a certain separability condition involving level-equivalence classes -- which we will need in later sections -- to a separability condition involving hyperplanes.
Finally, we recall the notion of a group acting specially on a CAT(0) cube complex, and prove that $\G$ is a virtually special lattice in $\Aut(\Delta)$ -- from which we deduce the ``only if'' direction in Theorem \ref{thm:Delta}.

\begin{defn}(Parallelism and hyperplanes)\\
	Two edges in a cube complex $X$ are \emph{elementary parallel} if they appear as opposite edges in some 2-cube.
	\emph{Parallelism} on edges of $X$ is the equivalence relation generated by elementary parallelism.
	For each edge $e$, we define the \emph{hyperplane} $H(e)$ to be the union of all midcubes in $X$ that intersect edges parallel to $e$. We say that $e$ is \emph{dual} to $H(e)$.
	We refer to \cite{WiseRiches} for more facts about hyperplanes. The important fact for us is that two edges are parallel if and only if they are dual to the same hyperplane.
	Also note that $H(ge)=gH(e)$ for $e$ an edge and $g\in\Aut(X)$.
	
	It will be natural in several of our arguments to consider \emph{oriented edges} (i.e. edges that come with an ordering of their vertices), and we will think of each oriented edge as pointing from an initial vertex to a terminal vertex.
	There is also a notion of parallelism for oriented edges.
	We say that two oriented edges in a cube complex $X$ are \emph{elementary parallel} if they appear as opposite edges that point in the same direction in some 2-cube.
	\emph{Parallelism} on oriented edges of $X$ is the equivalence relation generated by elementary parallelism.
\end{defn}

The following lemma enables us to label hyperplanes, and the corresponding dual edges, by elements $i\in I$.

\begin{lem}\label{lem:iedge}
	Let $e$ be an edge in $\Delta$. Then:
	\begin{enumerate}
		\item\label{item:iedge} There is $J\in\bar{N}$ and $i\in J\jperp$ such that $e$ joins vertices of type $J$ and $J\cup\{i\}$.
		We say that $e$ is an \textbf{$i$-edge}.
		\item\label{item:parisi} Any edge parallel to $e$ is also an $i$-edge.
		Hence $i$ only depends on the hyperplane $H(e)$, and we say that $H(e)$ is an \textbf{$i$-hyperplane}.
		
		\item\label{item:paror} If $e_1,e_2$ are parallel $i$-edges, then orienting them to have initial vertices of types $J_1,J_2$ and terminal vertices of types $J_1\cup\{i\},J_2\cup\{i\}$ makes them into parallel oriented edges $\overrightarrow{e_1}, \overrightarrow{e_2}$.
		
		\item\label{item:oneihyp} Each chamber in $\Delta$ intersects exactly one $i$-hyperplane (so intersects just one parallelism class of $i$-edges) for each $i\in I$.
	\end{enumerate} 	
\end{lem}
\begin{proof}
	Each elementary parallelism is supported on a 2-cube that lies inside a single chamber, and there is a type-preserving isomorphism from each chamber to $C(N)$, so it suffices to prove the lemma for $C(N)$ rather than $\Delta$.
	Statement \ref{item:iedge} holds because cubes in $C(N)$ correspond to intervals in $\bar{N}$.
	If $e_1,e_2$ are edges in $C(N)$ that are elementary parallel, and $e_1$ is an $i$-edge, then there is a 2-cube $Q$ containing $e_1,e_2$ as opposite edges, and the types of vertices of $Q$ make it isomorphic to the Hasse diagram of an interval in $\bar{N}$, so we deduce that $e_2$ is also an $i$-edge.
	Furthermore, if $e_1,e_2$ are given orientations $\overrightarrow{e_1},\overrightarrow{e_2}$ so that their terminal vertices are of types containing $i$, then $\overrightarrow{e_1},\overrightarrow{e_2}$ point in the same direction in $Q$.
	Statements \ref{item:parisi} and \ref{item:paror} follow.
	Finally, if $e$ is an $i$-edge in $C(N)$ then there is a cube $Q$ in $C(N)$ containing the center $v$ and $e$, and $e$ is parallel to the $i$-edge in $Q$ that joins $v$ to the vertex of type $\{i\}$. Hence all $i$-edges in $C(N)$ are parallel -- proving statement \ref{item:oneihyp}.
\end{proof}

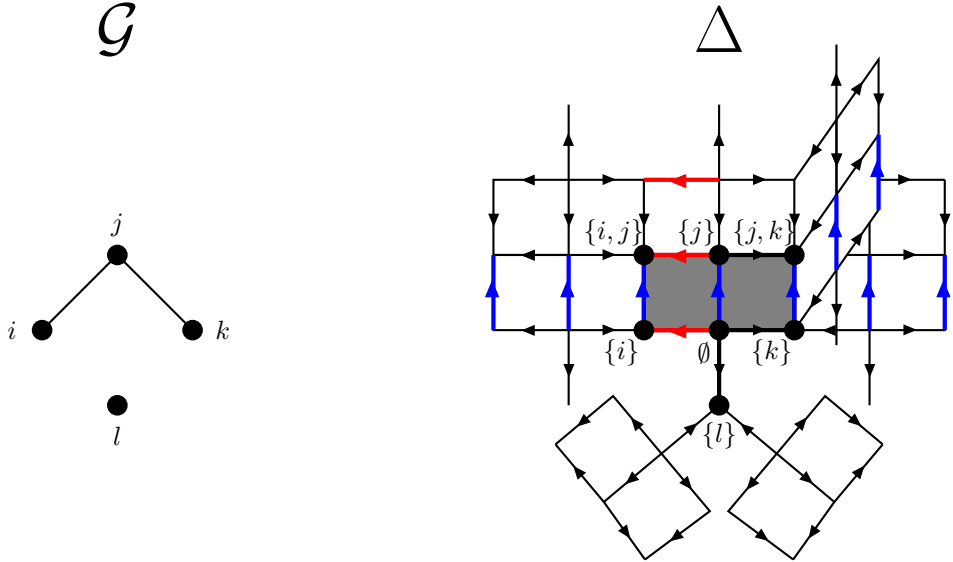
\begin{figure}[H]
	\centering
	\scalebox{1}{
		\begin{tikzpicture}[auto,node distance=2cm,
			thick,every node/.style={circle,draw,fill,inner sep=0pt,minimum size=7pt},
			every loop/.style={min distance=2cm},
			hull/.style={draw=none},
			]
			\tikzstyle{label}=[draw=none,fill=none]
			\tikzstyle{a}=[isosceles triangle,sloped,allow upside down,shift={(0,-.05)},minimum size=3pt]

			\begin{scope}[shift={(-6,0)}]
				\node at (-1,0) {};
				\node at (1,0) {};
				\node at (0,1) {};
				\node at (0,-1) {};
				\draw (-1,0)--(0,1)--(1,0);
				\node[label] at (-1.4,0) {$i$};
				\node[label] at (0,1.4) {$j$};
				\node[label] at (1.4,0) {$k$};
				\node[label] at (0,-1.4) {$l$};		
				\node[label,font=\Huge] at (0,4) {$\cG$};		
			\end{scope}

			\begin{scope}[shift={(2,0)}]
				\draw[fill=gray](-1,0)--(1,0)--(1,1)--(-1,1)--(-1,0);
				
				\draw(-2,0)-- node[a]{} (-3,0)--node[a]{}(-3,1);
				\draw(-2,0)-- node[a]{}(-2,1)-- node[a]{}(-3,1);
				\draw(-2,2)-- node[a]{}(-3,2)-- node[a]{}(-3,1);
				\draw(-2,2)-- node[a]{}(-2,1)-- node[a]{}(-3,1);
				\draw(-2,2)-- node[a]{}(-2,3);
				\draw(-2,0)-- node[a]{}(-2,-1);
				\draw(-2,1)-- node[a]{}(-1,1);
				\draw(-2,0)-- node[a]{}(-1,0)-- node[a]{}(-1,1);
				\draw(-2,2)-- node[a]{}(-1,2)-- node[a]{}(-1,1);
				
				\draw(0,0)-- node[a]{}(0,1)-- node[a]{}(-1,1);
				\draw(0,2)-- node[a]{}(-1,2);
				\draw(0,2)-- node[a]{}(0,1)-- node[a]{}(1,1);
				\draw(0,2)-- node[a]{}(0,3);
				\draw(0,2)-- node[a]{}(1,2)-- node[a]{}(1,1);
				\draw(0,0)-- node[a]{}(1,0)-- node[a]{}(1,1);
				\draw(0,0)-- node[a]{}(0,-1);
				\draw(1.56,.8)-- node[a]{}(1,0);
				\draw(1.56,2.8)-- node[a]{}(1,2);
				\draw(1.56,2.8)-- node[a]{}(1.56,1.8)-- node[a]{}(1,1);
				\draw(1.56,2.8)-- node[a]{}(1.56,1.8)-- node[a]{}(2.12,2.6);
				\draw(1.56,.8)-- node[a]{}(1.56,1.8);
				\draw(1.56,.8)-- node[a]{}(2.12,1.6)-- node[a]{}(2.12,2.6);
				\draw(1.56,.8)-- node[a]{}(1.56,-.2);
				\draw(1.56,2.8)-- node[a]{}(2.12,3.6)-- node[a]{}(2.12,2.6);
				\draw(1.56,2.8)-- node[a]{}(1.56,3.8);
				
				\draw(2,0)-- node[a]{}(1,0);
				\draw(2,0)-- node[a]{}(2,1)-- node[a]{}(3,1);
				\draw(2,0)-- node[a]{}(2,-1);
				\draw(2,0)-- node[a]{}(3,0)-- node[a]{}(3,1);
				\draw(3,2)-- node[a]{}(3,1);
				\draw(2.12,2)-- node[isosceles triangle,sloped,allow upside down,shift={(-.13,-.05)},minimum size=3pt]{}(3,2);
				\draw(1.7,1)--(2,1)--(2,1.429);
				
				\draw(-.766,-1.643)-- node[a]{}(0,-1);
				\draw(-.766,-1.643)-- node[a]{}(-.123,-2.409)-- node[a]{}(-.985,-3.052);
				\draw(-.766,-1.643)-- node[a]{}(-1.532,-2.286)-- node[a]{}(-.985,-3.052);
				\draw(-.766,-1.643)-- node[a]{}(-1.409,-.877)-- node[a]{}(-2.175,-1.52);
				\draw(-1.532,-2.286)-- node[a]{}(-2.175,-1.52);
				
				\draw(.766,-1.643)-- node[a]{}(0,-1);
				\draw(.766,-1.643)-- node[a]{}(.123,-2.409)-- node[a]{}(.985,-3.052);
				\draw(.766,-1.643)-- node[a]{}(1.532,-2.286)-- node[a]{}(.985,-3.052);
				\draw(.766,-1.643)-- node[a]{}(1.409,-.877)-- node[a]{}(2.175,-1.52);
				\draw(1.532,-2.286)-- node[a]{}(2.175,-1.52);

				\draw[ultra thick](0,0)--(0,-1);
				\draw[ultra thick](-1,0) grid (1,1);
				
				\draw[red,ultra thick](0,0)-- node[a]{}(-1,0);
				\draw[red,ultra thick](0,1)-- node[a]{}(-1,1);
				\draw[red,ultra thick](0,2)-- node[a]{}(-1,2);
				
				\draw[blue,ultra thick](0,0)-- node[a]{}(0,1);
				\draw[blue,ultra thick](-1,0)-- node[a]{}(-1,1);
				\draw[blue,ultra thick](-2,0)-- node[a]{}(-2,1);
				\draw[blue,ultra thick](-3,0)-- node[a]{}(-3,1);
				\draw[blue,ultra thick](1,0)-- node[a]{}(1,1);
				\draw[blue,ultra thick](2,0)-- node[a]{}(2,1);
				\draw[blue,ultra thick](3,0)-- node[a]{}(3,1);
				\draw[blue,ultra thick](1.56,.8)-- node[a]{}(1.56,1.8);
				\draw[blue,ultra thick](2.12,1.6)-- node[a]{}(2.12,2.6);

				\node at (0,0){};
				\node at (-1,0){};
				\node at (-1,1){};
				\node at (0,1){};		
				\node at (1,1){};
				\node at (1,0){};
				\node at (0,-1){};
				
				\node[label] at (-.2,-.3) {$\emptyset$};
				\node[label] at (-1.3,-.3) {$\{i\}$};
				\node[label] at (-1.4,1.3) {$\{i,j\}$};
				\node[label] at (-.3,1.3) {$\{j\}$};
				\node[label] at (.6,1.3) {$\{j,k\}$};
				\node[label] at (.7,-.3) {$\{k\}$};
				\node[label] at (0,-1.4) {$\{l\}$};	
				\node[label,font=\Huge] at (0,4) {$\Delta$};			
			\end{scope}
			
		\end{tikzpicture}
	}
	\caption{A repeat of Figure \ref{fig:leveladj}, but with a parallelism class of (oriented) $i$-edges shown in red and a parallelism class of (oriented) $j$-edges shown in blue.}\label{fig:iedge}
\end{figure}

Next we establish two basic lemmas regarding edge labels and parallelism.

\begin{lem}\label{lem:ijadj}
	Let $v\in\Delta^0$, let $e_1$ be an $i$-edge incident at $v$, and let $e_2$ be a $j$-edge incident at $v$.
	Then $e_1,e_2$ form the corner of a 2-cube at $v$ if and only if $i$ is adjacent to $j$.
\end{lem}
\begin{proof}
	If $e_1,e_2$ form the corner of a 2-cube $Q$ at $v$, then the types of the vertices of $Q$ correspond to the interval from $t(v)-\{i,j\}$ to $t(v)\cup\{i,j\}$ in $\bar{N}$ (Definitions \ref{defn:cubecone} and \ref{defn:building}), so $i$ is adjacent to $j$.
	
	Conversely, suppose that $i$ is adjacent to $j$.
	Let $t(v)=J$.
	It suffices to show that there is a chamber $C$ containing both $e_1$ and $e_2$, as then we can consider the 2-cube $Q\subset C$ whose vertex types correspond to the interval from $J-\{i,j\}$ to $J\cup\{i,j\}$ in $\bar{N}$, and we observe that $e_1,e_2$ form the corner of $Q$ at $v$.
	If $\{i,j\}\not\subset J$, say $i\notin J$, then the other endpoint $u$ of $e_1$ is of type $J\cup\{i\}$, so $v\leq u$, and any chamber containing $e_2$ also contains $e_1$ (Lemma \ref{lem:min}\ref{item:stayC}).
	Now suppose $\{i,j\}\subset J$.
	Consider chambers $C_{\gamma_1},C_{\gamma_2}$ containing $e_1,e_2$ respectively.
	We have $\gamma_1^{-1}\gamma_2\in\G_J$ by Lemma \ref{lem:chamberint}, so we may write $\gamma_1^{-1}\gamma_2=g_i g_j g$ where $g_i\in G_i$, $g_j\in G_j$ and $g\in \G_{J-\{i,j\}}$.
	The other endpoints of $e_1,e_2$ are of types $J-\{i\},J-\{j\}$ respectively, so Lemma \ref{lem:chamberint} implies that
	$$e_1\in C_{\gamma_1g_jg}\quad\text{and}\quad e_2\in C_{\gamma_2 g_i^{-1}}.$$
	But $g_i,g_j,g$ commute, so these two chambers are in fact the same chamber.
\end{proof}

\begin{lem}\label{lem:parallelchamber}
	Let $e,e'$ be $i$-edges in chambers $C,C'$ respectively. Then $e$ is parallel to $e'$ if and only if $C$ and $C'$ are contained in the same $i\jperp$-chamber-residue.
\end{lem}
\begin{proof}
	First suppose that $e$ is parallel to $e'$. We wish to show that $C$ and $C'$ are contained in the same $i\jperp$-chamber-residue. Since being in the same $i\jperp$-chamber-residue is an equivalence relation on chambers, it suffices to consider the case where $e$ is elementary parallel to $e'$.
	Let $Q$ be the 2-cube containing $e$ and $e'$, and let $u,u'$ be the endpoints of $e,e'$ respectively with least rank. If we label the vertices of $Q$ by their types, then $Q$ corresponds to some interval in $\bar{N}$, and one of $u,u'$ is at the bottom of this interval -- say $u$. Then $u\leq u'$, hence $u'\in C\cap C'$ by Lemma \ref{lem:min}\ref{item:stayC}. If $u'$ has type $J$ then $i\in J\jperp$ by Lemma \ref{lem:iedge}, so $J\subset i\jperp$ and $C'\in\cC(i\jperp,C)$ by Lemma \ref{lem:cCv}.
	
	Conversely, suppose that $C$ and $C'$ are contained in the same $i\jperp$-chamber-residue. We wish to show that $e$ is parallel to $e'$. As parallelism of edges is an equivalence relation, it suffices to consider the case where $C,C'$ are $j$-adjacent for some $j\in i\jperp$.
	By Lemma \ref{lem:chamberint} there is an $i$-edge $e''$ in $C\cap C'$ joining a vertex of type $\{j\}$ to a vertex of type $\{i,j\}$. But by Lemma \ref{lem:iedge}, $C$ and $C'$ each intersect just one parallelism class of $i$-edges, hence $e$ is parallel to $e'$ as required.
\end{proof}

For a vertex $v\in\Delta^0$ recall the set $E^-(v)$ from Definition \ref{defn:E-}.
These sets give us a way to characterize level-equivalence in terms of hyperplanes as follows.

\begin{lem}\label{lem:approxhyp}
	For vertices $v_1,v_2\in\Delta^0$ we have that $v_1\approx v_2$ if and only if
	\begin{equation}\label{samehyps}
		\{H(e_1)\mid e_1\in E^-(v_1)\}=\{H(e_2)\mid e_2\in E^-(v_2)\}.
	\end{equation}
\end{lem}
\begin{proof}	
	Suppose that $v_1\approx v_2$, with $t(v_1)=t(v_2)=J$ (Corollary \ref{cor:leveltype}).
	Let $e_1\in E^-(v_1)$ be a $j$-edge contained in a chamber $C_1$.
	Note that $j\in J$.
	By Lemma \ref{lem:approxjperp} there exists a chamber $C_2\in\cC(v_2)\cap\cC(J\jperp,C_1)$, so in particular $C_2\in\cC(j\jperp,C_1)$.
	Let $e_2$ be the unique $j$-edge in $C_2$ incident to $v_2$, so $e_2\in E^-(v_2)$.
	But then $H(e_1)=H(e_2)$ by Lemma \ref{lem:parallelchamber}.
	This proves the $\subset$ inclusion in (\ref{samehyps}), and the reverse inclusion holds by symmetry.
	
	Conversely, suppose (\ref{samehyps}) holds and let $t(v_1)=J$. We know that $E^-(v_1)$ contains a subset $\{e_1^j\}_{j\in J}$ where $e_1^j$ is a $j$-edge, and we know from (\ref{samehyps}) that $E^-(v_2)$ contains a corresponding subset $\{e_2^j\}_{j\in J}$ such that $e_2^j$ is parallel to $e_1^j$.
	In particular, this implies that $J\subset t(v_2)$, and by symmetry we have $t(v_2)=J$.
	If $J=\emptyset$ then $\rk(v_1)=\rk(v_2)=0$, and $v_1\approx v_2$ by Remark \ref{remk:extreme}, so suppose $J\neq\emptyset$.
	Say that each $e_1^j$ is contained in a chamber $C_1^j$ and that each $e_2^j$ is contained in a chamber $C_2^j$.
	Lemma \ref{lem:parallelchamber} implies that $C_1^j$ and $C_2^j$ are contained in the same $j\jperp$-chamber-residue.
	But we also know from Lemma \ref{lem:cCv} that the chambers $\{C_1^j\}$ are contained in the $J$-chamber-residue $\cC(v_1)$, and that the chambers $\{C_2^j\}$ are contained in the $J$-chamber-residue $\cC(v_2)$, so we deduce that
	\begin{equation}\label{intchamres}
		C_2\in\bigcap_{j\in J}\cC(J\cup j\jperp,C_1)
	\end{equation}	
	for any $C_1\in\cC(v_1)$ and $C_2\in\cC(v_2)$.
	But $\cap_{j\in J}(J\cup j\jperp)=J\uperp$, so the intersection of chamber-residues in (\ref{intchamres}) is equal to the chamber-residue $\cC(J\uperp,C_1)$ by Lemma \ref{lem:intchamres}.
	It then follows from Lemma \ref{lem:cC[v]} that $C_2\in\cC([v_1])$, and since $v_2$ is the only vertex in $C_2$ of type $J$ we deduce that $v_1\approx v_2$.
\end{proof}

We now relate hyperplane stabilizers to level-equivalence class stabilizers.

\begin{prop}\label{prop:hypstabs}
For each level-equivalence class $[v]$, the intersection of the $\Aut_{\rk}(\Delta)$-stabilizers of the hyperplanes $\{H(e)\mid e\in E^-(v)\}$ is a subgroup of the $\Aut_{\rk}(\Delta)$-stabilizer of $[v]$, and it has finite index if the groups $G_i$ are finite.
\end{prop}
\begin{proof}
	Let $A$ be the $\Aut_{\rk}(\Delta)$-stabilizer of $[v]$.
	For each level-equivalence class $[v]$, it follows from Lemma \ref{lem:approxhyp} that the intersection of the $\Aut_{\rk}(\Delta)$-stabilizers of the hyperplanes $\{H(e)\mid e\in E^-(v)\}$ is a subgroup of $A$. It also follows that $A$ acts on the set of hyperplanes $\{H(e)\mid e\in E^-(v)\}$. If the groups $G_i$ are finite then this set of hyperplanes is finite by the formula for $d^-(v)$ given in Definition \ref{defn:E-}, so there is a finite-index subgroup of $A$ that stabilizes each of these hyperplanes.
\end{proof}

Recall the definition of separable subgroup.

\begin{defn}(Separable subgroups)\\\label{defn:separable}
	A subgroup $H$ of a group $G$ is \emph{separable} (in $G$) if for any $g\in G-H$ there is a homomorphism $f:G\to\bar{G}$ to a finite group such that $f(g)\notin f(H)$.	
\end{defn}

We will need the following elementary facts about separable subgroups.

\begin{lem}\label{lem:separable}
	Let $G$ be a group. Then the following hold:
	\begin{enumerate}
		\item\label{item:H1H2} If $H_1<H_2<G$ with $H_1$ finite index in $H_2$ and separable in $G$, then $H_2$ is separable in $G$.
		\item\label{item:intH1} If $H_1<H_2<G$ with $H_1$ finite index in $H_2$ and separable in $G$, then there exists a finite-index normal subgroup $\hat{G}\triangleleft G$ with $H_2\cap\hat{G}<H_1$.
	\end{enumerate}
\end{lem}

When we prove the ``if'' direction of Theorem \ref{thm:Delta} in later sections, we will need the following subgroups of the uniform lattice $\La<\Aut(\Delta)$ to be separable (in $\La$): (1) $\La$-stabilizers of hyperplanes, and (2) finite-index subgroups of $\La$-stabilizers of level-equivalence classes.
In the following lemma we give a sufficient condition for this involving intersections of hyperplane-stabilizers; in particular, this shows that the condition from Theorem \ref{thm:Delta} -- that convex subgroups of $\La$ are separable -- is sufficient.
(Recall that a convex subgroup of $\La$ is a subgroup that stabilizes and acts cocompactly on a convex subcomplex of $\Delta$.)

\begin{prop}\label{prop:separable}
	Suppose that the graph $\cG$ and the groups $G_i$ are finite, and let $\Lambda<\Aut(\Delta)$ be a uniform lattice. If the finite-index subgroups of finite intersections of $\Lambda$-stabilizers of hyperplanes are separable in $\Lambda$, then the finite-index subgroups of the $\La$-stabilizers of level-equivalence classes are also separable in $\La$.
	In particular, this holds if all convex subgroups of $\La$ are separable.
\end{prop}
\begin{proof}
	Let $[v]$ be a level-equivalence class and let $L<\La_{[v]}$ be a finite-index subgroup.
	Let $M<\La$ be the intersection of the $\La$-stabilizers of the hyperplanes $\{H(e)\mid e\in E^-(v)\}$.
	Let $\La^{\rk}_{[v]},M^{\rk}$ be the rank-preserving subgroups of $\La_{[v]},M$ respectively.
	Proposition \ref{prop:hypstabs} implies that $M^{\rk}$ is a finite-index subgroup of $\La^{\rk}_{[v]}$.
	Proposition \ref{prop:finiteindex} implies that $\Aut_{\rk}(\Delta)$ has finite index in $\Aut(\Delta)$, so $\La^{\rk}_{[v]},M^{\rk}$ have finite index in $\La_{[v]},M$ respectively.
	Hence $L$ and $M$ are commensurable.
	As $L\cap M$ has finite index in $M$, the hypothesis of the proposition tells us that $L\cap M$ is separable in $\La$.
	And as $L\cap M$ has finite index in $L$, we deduce from Lemma \ref{lem:separable}\ref{item:H1H2} that $L$ is separable in $\La$.
\end{proof}

We recall the definition of a group acting specially on a CAT(0) cube complex \cite{HaglundWise08}, phrased in terms of parallelism of edges.

\begin{defn}(Acting specially)\\\label{defn:specially}
	Let $G$ be a group acting on a CAT(0) cube complex $X$. We say that $G$ acts \emph{specially} on $X$ if the following two properties are satisfied:
	\begin{enumerate}
		\item(Acting \emph{cleanly})\label{item:cleanly} If $\overrightarrow{e_1}, \overrightarrow{e_2}$ are distinct oriented edges with common initial vertex $v$, and $g\in G$, then $g\overrightarrow{e_1}$ is not parallel to $\overrightarrow{e_2}$.
		\item(Acting \emph{nicely}) Suppose edges $e_1,e_2$ form the corner of a 2-cube at a vertex $v$, and suppose edges $e'_1,e'_2$ are incident at a vertex $v'$. If $g\in G$ is such that $ge_1$ is parallel to $e'_1$, and $e_2$ is parallel to $e'_2$, then $e'_1,e'_2$ form the corner of a 2-cube at $v'$.
	\end{enumerate}
	We say that a uniform lattice $\G<\Aut(X)$ is \emph{special} if it acts specially on $X$, and \emph{virtually special} if it has a finite-index subgroup that acts specially on $X$.
\end{defn}

It is well known that any convex subgroup of a virtually special uniform lattice is separable (see for instance \cite[Corollary 7.9 and Lemma 9.16]{HaglundWise08}).
Thus it follows from the following proposition that all convex subgroups of $\G$ are separable.
In particular, this implies the ``only if'' direction in Theorem \ref{thm:Delta}: $\La$ and $\G$ being weakly commensurable in $\Aut(\Delta)$ implies that all convex subgroups of $\La$ are separable.

\begin{prop}\label{prop:Gvspecial}
	$\G<\Aut(\Delta)$ is virtually special.
\end{prop}
\begin{proof}
	Let $\hat{\G}$ be the kernel of the natural map $\G\to \prod_{i\in I} G_i$.
	We show that $\hat{\G}$ is a special lattice in $\Aut(\Delta)$ by verifying the two properties from Definition \ref{defn:specially}:
	\begin{enumerate}
		\item Let $\overrightarrow{e_1}, \overrightarrow{e_2}$ be distinct oriented edges with common initial vertex $v$, let $\hat{\gamma}\in \hat{\G}$, and suppose for contradiction that $\hat{\gamma}\overrightarrow{e_1}$ is parallel to $\overrightarrow{e_2}$.
		Suppose that $\overrightarrow{e_1}$ is an $i$-edge. The action of $\G$ preserves types of vertices, so $\hat{\gamma}\overrightarrow{e_1}$ is also an $i$-edge, and $\overrightarrow{e_2}$ is an $i$-edge too by Lemma \ref{lem:iedge}.
		Let $t(v)=J$.
		We must have $i\in J$, otherwise the terminal vertices $u_1,u_2$ of $\overrightarrow{e_1}, \overrightarrow{e_2}$ would be distinct vertices of the same type $J\cup\{i\}$, and any chamber containing $v$ would contain $u_1$ and $u_2$ by Lemma \ref{lem:min}\ref{item:stayC}, a contradiction.
		
		Let $\overrightarrow{e_1}$ be contained in a chamber $C_{\gamma_1}$. By Lemma \ref{lem:chamberint} there is a chamber $C_{\gamma_1 g}$ containing $\overrightarrow{e_2}$ with $g\in\G_J$.
		The edge $\hat{\gamma}\overrightarrow{e_1}$ is contained in the chamber $C_{\hat{\gamma}\gamma_1}$, so by Lemmas \ref{lem:CJC} and \ref{lem:parallelchamber} we deduce that $\hat{\gamma}\gamma_1\in \gamma_1g\G_{i\jperp}$.
		Since $\hat{\gamma}\in\hat{\G}$, we know that the projection of $\hat{\gamma}$ to $G_i$ is trivial, so we deduce that the projection of $g$ to $G_i$ is also trivial. We have $g\in\G_J$, and $\G_J$ is the product of the groups $G_j$ for $j\in J$, so $g$ must be a product of elements in the groups $G_j$ for $j\in J-\{i\}$, or equivalently $g\in\G_{J-\{i\}}$.
		The terminal vertices $u_1,u_2$ of $\overrightarrow{e_1}, \overrightarrow{e_2}$ are vertices of the same type $J-\{i\}$ in the chambers $C_{\gamma_1},C_{\gamma_1g}$ respectively, so $u_1=u_2$ by Lemma \ref{lem:chamberint} and the fact that $g\in\G_{J-\{i\}}$, contradicting the distinctness of $\overrightarrow{e_1}, \overrightarrow{e_2}$.
		
		\item Suppose edges $e_1,e_2$ form the corner of a 2-cube at a vertex $v$, and suppose edges $e'_1,e'_2$ are incident at a vertex $v'$. Suppose $\hat{\gamma}\in \hat{\G}$ is such that $\hat{\gamma}e_1$ is parallel to $e'_1$, and suppose $e_2$ is parallel to $e'_2$.
		Let $e_1$ be an $i$-edge and let $e_2$ be a $j$-edge.
		Lemma \ref{lem:ijadj} implies that $i$ is adjacent to $j$.
		The action of $\G$ preserves types of vertices, so Lemma \ref{lem:iedge} implies that $e'_1$ is an $i$-edge and $e'_2$ is a $j$-edge.
		Applying Lemma \ref{lem:ijadj} again, we deduce that $e'_1,e'_2$ form the corner of a 2-cube at $v'$, as required.\qedhere	
	\end{enumerate}
\end{proof}

\bigskip
\section{Residue-groupoids}\label{sec:resgroup}

In this section we introduce the notion of residue-groupoids, and prove several lemmas about them.

\begin{defn}(Residue groupoids)\\\label{defn:resgroup}
	Let $C\in\cC(\Delta)$ and $J\subset I$. A \emph{$\cC(J,C)$-groupoid} is a collection of maps $\phi=(\phi_{C_1,C_2})$, where $(C_1,C_2)$ ranges over all pairs of chambers in $\cC(J,C)$ and each map $\phi_{C_1,C_2}:C_1\to C_2$ is a cubical isomorphism preserving centers, such that the following properties hold for all $C_1,C_2,C_3\in\cC(J,C)$:
	\begin{enumerate}
		\item(Lower degree)\label{item:d-} $\phi_{C_1,C_2}$ preserves lower degrees of rank-1 vertices;
		\item(Identity)\label{item:id} $\phi_{C_1,C_1}$ is the identity map on $C_1$;		
		\item(Composition)\label{item:compose} $\phi_{C_2,C_3}\circ\phi_{C_1,C_2}=\phi_{C_1,C_3}$;
		\item(Intersection)\label{item:fixint} $\phi_{C_1,C_2}$ restricts to the identity map on $C_1\cap C_2$ whenever this intersection is non-empty.
	\end{enumerate}
If we do not wish to specify the chamber-residue $\cC(J,C)$ then we may refer to $\phi$ as a \emph{residue-groupoid}.
\end{defn}

\begin{remk}\label{remk:rankgroup}
	Within a given chamber, the rank of a vertex is equal to the length of a shortest edge path to the center, therefore each map $\phi_{C_1,C_2}$ within a residue-groupoid preserves ranks of vertices and induces an isomorphism between the poset structures on $C_1$ and $C_2$.
\end{remk}

In the following lemma we show that to define a residue-groupoid it is enough to define maps between adjacent chambers that satisfy a version of Definition \ref{defn:resgroup}.

\begin{lem}\label{lem:resgroup}
	Let $\cC(J,C)$ be a chamber-residue and let $\phi=(\phi_{C_1,C_2})$ be a collection of maps, where $(C_1,C_2)$ ranges over all pairs of adjacent chambers in $\cC(J,C)$ and each map $\phi_{C_1,C_2}:C_1\to C_2$ is a cubical isomorphism preserving centers. Then $\phi$ has a unique extension to a $\cC(J,C)$-groupoid $\bar{\phi}$ if the following five conditions hold:
	\begin{enumerate}[label=\wackyenum*]
		\item\label{item:d-'} Each $\phi_{C_1,C_2}$ preserves lower degrees of rank-1 vertices.
		\item\label{item:inverse} $\phi_{C_1,C_2}=\phi_{C_2,C_1}^{-1}$.		
		\item\label{item:coi} $\phi_{C_2,C_3}\circ\phi_{C_1,C_2}=\phi_{C_1,C_3}$ for all $i\in I$ and all chambers $C_1,C_2,C_3$ in the same $\{i\}$-chamber-residue in $\cC(J,C)$.
		\item\label{item:square} For all adjacent $i,j\in I$ and pairs of $i$-adjacent chambers $C_1,C_2$ and $C'_1,C'_2$, if the pairs $C_1,C'_1$ and $C_2,C'_2$ are $j$-adjacent then the following diagram commutes:
			\begin{equation}\label{square}
			\begin{tikzcd}[
				ar symbol/.style = {draw=none,"#1" description,sloped},
				isomorphic/.style = {ar symbol={\cong}},
				equals/.style = {ar symbol={=}},
				subset/.style = {ar symbol={\subset}}
				]
				C_1\ar{d}[swap]{\phi_{C_1,C'_1}}\ar{r}{\phi_{C_1,C_2}}&C_2\ar{d}{\phi_{C_2,C'_2}}\\
				C'_1\ar{r}[swap]{\phi_{C'_1,C'_2}}&C'_2.
			\end{tikzcd}
		\end{equation}
	\item\label{item:fixadj} $\phi_{C_1,C_2}$ restricts to the identity map on $C_1\cap C_2$.
	\end{enumerate}
\end{lem}
\begin{proof}
Given an arbitrary pair of chambers	$C_1,C_2\in\cC(J,C)$, we define $\bar{\phi}_{C_1,C_2}$ to be the identity map on $C_1$ if $C_1=C_2$, otherwise we consider $(C'_0,C'_1,...,C'_n)$ a $J$-gallery from $C_1$ to $C_2$ and define $\bar{\phi}_{C_1,C_2}$ to be the following composition:
\begin{equation}\label{barphidef}
\bar{\phi}_{C_1,C_2}:=\phi_{C'_{n-1},C'_n}\circ\cdots\circ\phi_{C'_1,C'_2}\circ\phi_{C'_0,C'_1}
\end{equation}
First let's show that $\bar{\phi}_{C_1,C_2}$ is well-defined.
Indeed, by Lemma \ref{lem:moveinJ}, any two $J$-galleries from $C_1$ to $C_2$ differ by a sequence of moves \ref{M1}--\ref{M3} such that the intermediate galleries are also $J$-galleries. But $\bar{\phi}_{C_1,C_2}$ is invariant under moves \ref{M1},\ref{M2},\ref{M3} precisely because of properties \ref{item:inverse},\ref{item:coi},\ref{item:square} respectively, hence $\bar{\phi}_{C_1,C_2}$ is well-defined.

It is clear that $\bar{\phi}$ satisfies properties \ref{item:d-}--\ref{item:compose} of Definition \ref{defn:resgroup}, and also that it is the unique extension of $\phi$ to satisfy properties \ref{item:id} and \ref{item:compose}, so it remains to check property \ref{item:fixint}.
Let $C_1,C_2\in\cC(J,C)$ have non-empty intersection.
To show that $\bar{\phi}_{C_1,C_2}$ restricts to the identity map on $C_1\cap C_2$, it suffices to check that it fixes every vertex in $C_1\cap C_2$. Let $v\in C_1\cap C_2$ be a vertex. If $t(v)=J'$ then it follows from Lemma \ref{lem:cCv} that $C_1,C_2$ are contained in the same $J'$-chamber-residue. And we already know that $C_1,C_2$ are in the same $J$-chamber-residue, so it follows from Lemma \ref{lem:intchamres} that they are in fact in the same $(J\cap J')$-chamber-residue. Let $(C'_0,C'_1,...,C'_n)$ be a $(J\cap J')$-gallery from $C_1$ to $C_2$.
The chambers $C'_k$ are in the same $J'$-chamber-residue as $C_1,C_2$, so Lemma \ref{lem:cCv} tells us that $v\in\cap_k C'_k$.
For each $1\leq k\leq n$ the chambers $C'_{k-1},C'_k$ are adjacent, so $\phi_{C'_{k-1},C'_k}(v)=v$ by \ref{item:fixadj}.
It then follows from (\ref{barphidef}) that $\bar{\phi}_{C_1,C_2}(v)=v$ (or from the fact that $\bar{\phi}_{C_1,C_1}$ is the identity map on $C_1$ in the case that $C_1=C_2$).
\end{proof}

\begin{lem}\label{lem:Gammaresgroup}
	For each chamber-residue $\cC(J,C)$, the group $\G$ induces a $\cC(J,C)$-groupoid $\phi$, where each map $\phi_{C_1,C_2}$ is the restriction of the unique element $\gamma\in\G$ with $\gamma C_1=C_2$.
\end{lem}
\begin{proof}
The group $\G$ preserves the rank of vertices, so also preserves the lower degrees. The identity and composition properties for $\phi$ follow from $\G$ being a group. For the intersection property, consider chambers $C_{\gamma_1},C_{\gamma_2}$ with non-empty intersection. Lemma \ref{lem:chamberint} gives us $J\in\bar{N}$ such that $\gamma_1^{-1}\gamma_2\in \G_J$ and such that every point in $C_{\gamma_1}\cap C_{\gamma_2}$ is of the form $[\gamma_1,p]$ with $J\subset \ut(p)$. The map $\phi_{C_{\gamma_1},C_{\gamma_2}}$ is given my left multiplication of $\gamma_2\gamma_1^{-1}$, and for $[\gamma_1,p]\in C_{\gamma_1}\cap C_{\gamma_2}$ we have $\gamma_2\gamma_1^{-1}\cdot [\gamma_1,p]=[\gamma_2,p]$, and this equals $[\gamma_1,p]$ by Definition \ref{defn:building} and the fact that $\gamma_1^{-1}\gamma_2\in\G_J<\G_{\ut(p)}$.
\end{proof}

\begin{lem}\label{lem:stayj}
	Let $\phi_{C_1,C_2}$ be a map within a residue-groupoid $\phi$ such that $C_1,C_2$ are $i$-adjacent.
	Then for any $j\in i\jperp$, the vertex $v_1\in C_1$ of type $\{j\}$ is level-adjacent to $\phi_{C_1,C_2}(v_1)$. In particular $\phi_{C_1,C_2}(v_1)$ is also of type $\{j\}$.
\end{lem}
\begin{proof}
	Let $u\in C_1$ be the vertex of type $\{i,j\}$. Then $u\in C_1\cap C_2$ by Lemma \ref{lem:min}\ref{item:wCcapC'} and \ref{item:iadjwedge}, and $v_1$ is the unique rank-1 vertex in $C_1-C_2$ such that $v_1\leq u$.
	Similarly, if $v_2\in C_2$ is the vertex of type $\{j\}$ then $v_2$ is the unique rank-1 vertex in $C_2-C_1$ such that $v_2\leq u$, and $v_2$ is level-adjacent to $v_1$. We know that $\phi_{C_1,C_2}$ fixes $C_1\cap C_2$ by the intersection property of residue-groupoids, so in particular it fixes $u$. Since $\phi_{C_1,C_2}$ preserves rank and poset structure (Remark \ref{remk:rankgroup}), we deduce that $\phi_{C_1,C_2}(v_1)=v_2$.
\end{proof}

\bigskip
\section{Hierarchy of level-equivalence classes}\label{sec:hierclasses}

In this section and the subsequent two we prove the ``if'' direction in Theorem \ref{thm:Delta}. So fix $\G=\G(\cG,(G_i)_{i\in I})$ a graph product of finite groups, with finite underlying graph $\cG$, let $\Delta=\Delta(\cG,(G_i)_{i\in I})$ be the associated right-angled building, and let $\La<\Aut(\Delta)$ be a uniform lattice such that all convex subgroups of $\La$ are separable.
We wish to show that $\La$ and $\G$ are weakly commensurable in $\Aut(\Delta)$.
Proposition \ref{prop:finiteindex} says that $\Aut_{\rk}(\Delta)$ has finite index in $\Aut(\Delta)$, so we may assume that $\La<\Aut_{\rk}(\Delta)$. Thus $\La$ preserves all the structures from Proposition \ref{prop:preserved}.

A key step will be to prove the following proposition.

\begin{prop}\label{prop:Deltagroupoid}
	There exists a $\La'$-invariant $\cC(\Delta)$-groupoid for some finite-index subgroup $\La'<\La$.
\end{prop}

Here, a $\cC(\Delta)$-groupoid $\phi$ is \emph{$\La'$-invariant} if the collections of maps $$\{\phi_{C_1,C_2}:C_1\to C_2\}\quad\text{and}\quad\{\lambda\circ\phi_{C_1,C_2}\circ\lambda^{-1}:\lambda C_1\to \lambda C_2\}$$
are the same for any $\lambda\in\La'$.

To prove Proposition \ref{prop:Deltagroupoid}, we will construct a hierarchy of residue-groupoids, and prove that the collection of residue-groupoids on each level of the hierarchy is $\La'$-invariant (possibly different $\Lambda'$ for different levels).
What it means for a collection of residue-groupoids to be $\La'$-invariant will be detailed in Section \ref{sec:hierresgroup}.

Before we construct the hierarchy of residue-groupoids, we must first endow the level-equivalence classes in $\Delta$ with a hierarchical structure, which is the focus of the current section. We will also define an operator on the vertices of $\Delta$ called ascent that will help us to move up this hierarchy in a controlled manner.
First we need two lemmas and a definition regarding the interaction between level-equivalence and the partial order on $\Delta^0$.

\begin{lem}\label{lem:downtype}
	If $u_1\approx v_1$, $u_2\leq u_1$ and $v_2\leq v_1$ with $t(u_2)=t(v_2)$, then $u_2\approx v_2$.
\end{lem}
\begin{proof}
	Let $J_1=t(u_1)=t(v_1)$ (Corollary \ref{cor:leveltype}) and $J_2=t(u_2)=t(v_2)$.
	We know that $J_2\subset J_1$ since $u_2\leq u_1$, and $J_1\subset J_2\uperp$ as $J_1$ is spherical.
	Hence $J_1\uperp\subset J_2\uperp$.
	Let $C_1,C_2$ be chambers containing $u_2,v_2$ respectively. We have $u_1\in C_1$ and $v_1\in C_2$ by Lemma \ref{lem:min}\ref{item:stayC}.
	Lemma \ref{lem:cC[v]} then implies that $\cC([u_1])=\cC(C_1,J_1\uperp)\subset\cC(C_1,J_2\uperp)=\cC([u_2])$.
	Now $u_1\approx v_1$, so $\cC([u_1])=\cC([v_1])$, and we deduce that $C_2\in\cC([u_2])$.
	But the only vertex in $C_2$ of type $J_2$ is $v_2$, hence $u_2\approx v_2$.
\end{proof}

\begin{defn}(1-downsets)\\\label{defn:1downset}
Write $\Delta^0_1$ for the set of rank-1 vertices in $\Delta$.
Define the \emph{1-downset} of a vertex $u\in\Delta^0$ to be the set
$${\downarrow_1}(u):=\{v\in\Delta^0_1\mid v\leq u\}.$$
\end{defn}

\begin{lem}\label{lem:downset}
	The following hold for any vertex $u\in\Delta^0$:
	\begin{enumerate}
		\item\label{item:downsettypes} $\{t(v)\mid v\in{\downarrow_1}(u)\}=\{\{i\}\mid i\in t(u)\}$
		\item\label{item:downclasses} $\{[v]\mid v\in{\downarrow_1}(u)\}$ only depends on $[u]$.
		\item\label{item:downsetdet} If ${\downarrow_1}(u)$ is non-empty then it uniquely determines the vertex $u$.
	\end{enumerate}
\end{lem}
\begin{proof}
	The inclusion $\subset$ in \ref{item:downsettypes} is immediate from the definitions. The reverse inclusion $\supset$ holds because if $C$ is a chamber containing $u$ and $i\in t(u)$ then $C$ contains a vertex $v$ of type $\{i\}$, and necessarily $v\in{\downarrow_1}(u)$.
	Statement \ref{item:downclasses} follows from \ref{item:downsettypes} and Lemma \ref{lem:downtype}.
	Finally, if ${\downarrow_1}(u)\neq\emptyset$ then we can take $v\in{\downarrow_1}(u)$, choose a chamber $C$ containing $v$, and characterize $u$ as the unique vertex in $C$ of type $t(u)$ ($u\in C$ by Lemma \ref{lem:min}\ref{item:stayC}). 
	But $t(u)$ is determined by ${\downarrow_1}(u)$ because of \ref{item:downsettypes}, hence ${\downarrow_1}(u)$ uniquely determines the vertex $u$.
	This proves \ref{item:downsetdet}.
\end{proof}

The hierarchy on level-equivalence classes will be modeled on the following total order.

\begin{lem}\label{lem:total}
	We can define a total order $\preceq$ on the power set of $\{1,...,n\}$ by setting $S_1\preceq S_2$ if $S_1=S_2$ or $\max (S_1\triangle S_2)\in S_2$.
\end{lem}
\begin{proof}
	This follows from the observation that $S_1\preceq S_2$ is equivalent to $\sum_{s\in S_1} 2^s\leq\sum_{s\in S_2} 2^s$.
\end{proof}

\begin{defn}(Hierarchy of level-equivalence classes)\\\label{den:hiertype}
	Write $\cL$ for the set of level-equivalence classes and write $q:\Delta^0_1\to\cL/\La$ for the quotient map defined by $q(v):=\La\cdot[v]$.
	For $u\in\Delta^0$ it follows from Lemma \ref{lem:downset}\ref{item:downclasses} that $q({\downarrow_1}(u))$ only depends on $\La\cdot[u]$.
	Now fix a total order $\preceq$ on $q(\Delta^0_1)$, and extend it to a partial order on $\cL/\La$ by setting $\La\cdot[u_1]\preceq\La\cdot[u_2]$ if $\La\cdot[u_1]=\La\cdot[u_2]$ or
	$$\max(q({\downarrow_1}(u_1))\triangle q({\downarrow_1}(u_2)))\in q({\downarrow_1}(u_2)).$$
	This is indeed a partial order by Lemma \ref{lem:total} (noting that $q(\Delta^0_1)$ is a finite set since $\La$ acts cocompactly on $\Delta$).
\end{defn}

The convex subgroups of $\La$ are separable by hypothesis, so in particular $\La$ has separable hyperplane stabilizers. Thus, we may apply \cite[Lemma 9.14]{HaglundWise08} and replace $\La$ by a finite-index subgroup that acts cleanly on $\Delta$ (Definition \ref{defn:specially}\ref{item:cleanly}).
As a consequence, if edges $e_1,e_2$ form the corner of a 2-cube in $\Delta$ then no $\La$-translate of $e_1$ is parallel to $e_2$.
This leads to the following lemma.

\begin{lem}\label{lem:distinctq}
	Let $v_1,v_2$ be vertices in a chamber $C$ of types $\{i\}$, $\{j\}$ respectively, with $i,j\in I$ adjacent. Then $q(v_1)\neq q(v_2)$.
\end{lem}
\begin{proof}
	Suppose not.
	Let $e_1,e_2$ be the edges that join the center of $C$ to $v_1,v_2$ respectively.
	Note that $e_1\in E^-(v_1)$ is an $i$-edge and $e_2\in E^-(v_2)$ is a $j$-edge (Definition \ref{defn:E-}).
	As $q(v_1)=q(v_2)$, there exists $\lambda\in\La$ such that $\lambda v_1\approx v_2$, and by Lemma \ref{lem:approxhyp} this implies that some edge $f_2\in E^-(v_2)$ is parallel to $\lambda e_1$ (possibly $e_2=f_2$).
	Let $C'$ be the chamber containing $f_2$, and let $f_1$ be the $i$-edge incident to the center $u'$ of $C'$.
	The chambers $C,C'$ are either equal or $j$-adjacent, and $j\in i\jperp$, so it follows from Lemma \ref{lem:parallelchamber} that $e_1$ is parallel to $f_1$.
	Hence $\lambda f_1$ is parallel to $f_2$. If we orient $f_1,f_2$ to have initial vertex $u'$, and denote these oriented edges by $\overrightarrow{f_1}, \overrightarrow{f_2}$, then it follows from Lemma \ref{lem:iedge}\ref{item:paror} that $\lambda \overrightarrow{f_1}$ is parallel to $\overrightarrow{f_2}$.
	But this contradicts the fact that $\La$ acts cleanly on $\Delta$.	
\end{proof}

\begin{remk}\label{remk:extremehier}
	A strict inequality $u_1<u_2$ in $\Delta^0$ implies that ${\downarrow_1}(u_1)\subsetneq{\downarrow_1}(u_2)$.
	It then follows from Lemma \ref{lem:distinctq} that $q({\downarrow_1}(u_1))\subsetneq q({\downarrow_1}(u_2))$, so we get a strict inequality $\Lambda\cdot[u_1]\prec\Lambda\cdot[u_2]$. Thus the $\preceq$-maximal $\Lambda\cdot[v]\in\cL/\La$ have $v$ a $\leq$-maximal vertex in $\Delta^0$, $t(v)$ a maximal spherical subset of $I$ and $[v]=\{v\}$ a singleton.
	(Although there may be some $\leq$-maximal vertices $v$ such that $\Lambda\cdot[v]$ is not $\preceq$-maximal.) On the other hand, the unique $\preceq$-smallest class $\Lambda\cdot[v]\in\cL/\La$ is the class of all rank-0 vertices (see Remark \ref{remk:extreme}).
\end{remk}

We now introduce a (partial) binary operator on $\Delta^0$ called ascent, which moves upward with respect to the orderings $\leq$ and $\preceq$, and will play a key role in the inductive construction in Section \ref{sec:hierresgroup}.

\begin{defn}(Ascent)\\
	Given a chamber $C$ and vertices $u,v\in C$ with $t(u)=\{i\}$ and $i\in t(v)\jperp$, we define the \emph{ascent of $v$ by $u$}, denoted $v\Uparrow u$, to be the unique vertex in $C$ of type
	\begin{equation}\label{ascenttype}
t(v\Uparrow u)=\cup\{t(u')\mid u'\in{\downarrow_1}(v): q(u)\prec q(u')\}\cup\{i\}.
	\end{equation}
Observe that $C$ is the only chamber containing both $u$ and $v$ by Lemma \ref{lem:min}\ref{item:wCcapC'}, so $v\Uparrow u$ only depends on $u$ and $v$.
We note that (\ref{ascenttype}) implies
\begin{equation}\label{ascenttypesandwich}
	\{i\}\subset t(v\Uparrow u)\subset t(v)\cup\{i\}.
\end{equation}
	We also note that
	\begin{equation}\label{down1Down}
		{\downarrow_1}(v\Uparrow u)=\{u'\in{\downarrow_1}(v)\mid q(u)\prec q(u')\}\cup\{u\}.
	\end{equation}
\end{defn}

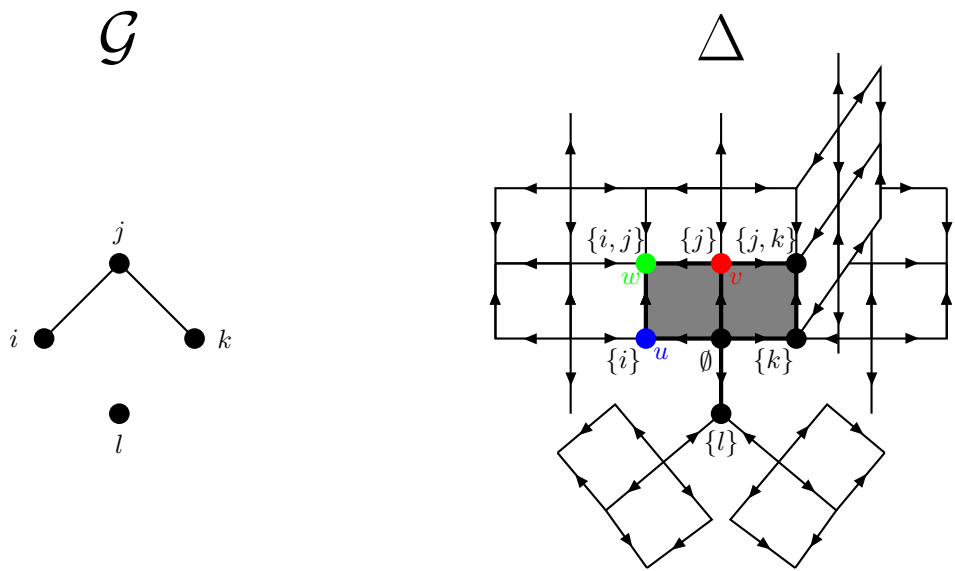
\begin{figure}[H]
	\centering
	\scalebox{1}{
		\begin{tikzpicture}[auto,node distance=2cm,
			thick,every node/.style={circle,draw,fill,inner sep=0pt,minimum size=7pt},
			every loop/.style={min distance=2cm},
			hull/.style={draw=none},
			]
			\tikzstyle{label}=[draw=none,fill=none]
			\tikzstyle{a}=[isosceles triangle,sloped,allow upside down,shift={(0,-.05)},minimum size=3pt]

			\begin{scope}[shift={(-6,0)}]
				\node at (-1,0) {};
				\node at (1,0) {};
				\node at (0,1) {};
				\node at (0,-1) {};
				\draw (-1,0)--(0,1)--(1,0);
				\node[label] at (-1.4,0) {$i$};
				\node[label] at (0,1.4) {$j$};
				\node[label] at (1.4,0) {$k$};
				\node[label] at (0,-1.4) {$l$};		
				\node[label,font=\Huge] at (0,4) {$\cG$};		
			\end{scope}

			\begin{scope}[shift={(2,0)}]
				\draw[fill=gray](-1,0)--(1,0)--(1,1)--(-1,1)--(-1,0);
				
				\draw(-2,0)-- node[a]{} (-3,0)--node[a]{}(-3,1);
				\draw(-2,0)-- node[a]{}(-2,1)-- node[a]{}(-3,1);
				\draw(-2,2)-- node[a]{}(-3,2)-- node[a]{}(-3,1);
				\draw(-2,2)-- node[a]{}(-2,1)-- node[a]{}(-3,1);
				\draw(-2,2)-- node[a]{}(-2,3);
				\draw(-2,0)-- node[a]{}(-2,-1);
				\draw(-2,1)-- node[a]{}(-1,1);
				\draw(-2,0)-- node[a]{}(-1,0)-- node[a]{}(-1,1);
				\draw(-2,2)-- node[a]{}(-1,2)-- node[a]{}(-1,1);
				
				\draw(0,0)-- node[a]{}(0,1)-- node[a]{}(-1,1);
				\draw(0,2)-- node[a]{}(-1,2);
				\draw(0,2)-- node[a]{}(0,1)-- node[a]{}(1,1);
				\draw(0,2)-- node[a]{}(0,3);
				\draw(0,2)-- node[a]{}(1,2)-- node[a]{}(1,1);
				\draw(0,0)-- node[a]{}(1,0)-- node[a]{}(1,1);
				\draw(0,0)-- node[a]{}(0,-1);
				\draw(1.56,.8)-- node[a]{}(1,0);
				\draw(1.56,2.8)-- node[a]{}(1,2);
				\draw(1.56,2.8)-- node[a]{}(1.56,1.8)-- node[a]{}(1,1);
				\draw(1.56,2.8)-- node[a]{}(1.56,1.8)-- node[a]{}(2.12,2.6);
				\draw(1.56,.8)-- node[a]{}(1.56,1.8);
				\draw(1.56,.8)-- node[a]{}(2.12,1.6)-- node[a]{}(2.12,2.6);
				\draw(1.56,.8)-- node[a]{}(1.56,-.2);
				\draw(1.56,2.8)-- node[a]{}(2.12,3.6)-- node[a]{}(2.12,2.6);
				\draw(1.56,2.8)-- node[a]{}(1.56,3.8);
				
				\draw(2,0)-- node[a]{}(1,0);
				\draw(2,0)-- node[a]{}(2,1)-- node[a]{}(3,1);
				\draw(2,0)-- node[a]{}(2,-1);
				\draw(2,0)-- node[a]{}(3,0)-- node[a]{}(3,1);
				\draw(3,2)-- node[a]{}(3,1);
				\draw(2.12,2)-- node[isosceles triangle,sloped,allow upside down,shift={(-.13,-.05)},minimum size=3pt]{}(3,2);
				\draw(1.7,1)--(2,1)--(2,1.429);
				
				\draw(-.766,-1.643)-- node[a]{}(0,-1);
				\draw(-.766,-1.643)-- node[a]{}(-.123,-2.409)-- node[a]{}(-.985,-3.052);
				\draw(-.766,-1.643)-- node[a]{}(-1.532,-2.286)-- node[a]{}(-.985,-3.052);
				\draw(-.766,-1.643)-- node[a]{}(-1.409,-.877)-- node[a]{}(-2.175,-1.52);
				\draw(-1.532,-2.286)-- node[a]{}(-2.175,-1.52);
				
				\draw(.766,-1.643)-- node[a]{}(0,-1);
				\draw(.766,-1.643)-- node[a]{}(.123,-2.409)-- node[a]{}(.985,-3.052);
				\draw(.766,-1.643)-- node[a]{}(1.532,-2.286)-- node[a]{}(.985,-3.052);
				\draw(.766,-1.643)-- node[a]{}(1.409,-.877)-- node[a]{}(2.175,-1.52);
				\draw(1.532,-2.286)-- node[a]{}(2.175,-1.52);

				\draw[ultra thick](0,0)--(0,-1);
				\draw[ultra thick](-1,0) grid (1,1);
				
				\draw(0,0)-- node[a]{}(-1,0);
				\draw(0,1)-- node[a]{}(-1,1);
				\draw(0,2)-- node[a]{}(-1,2);
				
				\draw(0,0)-- node[a]{}(0,1);
				\draw(-1,0)-- node[a]{}(-1,1);
				\draw(-2,0)-- node[a]{}(-2,1);
				\draw(-3,0)-- node[a]{}(-3,1);
				\draw(1,0)-- node[a]{}(1,1);
				\draw(2,0)-- node[a]{}(2,1);
				\draw(3,0)-- node[a]{}(3,1);
				\draw(1.56,.8)-- node[a]{}(1.56,1.8);
				\draw(2.12,1.6)-- node[a]{}(2.12,2.6);

				\node at (0,0){};
				\node[blue] at (-1,0){};
				\node[green] at (-1,1){};
				\node[red] at (0,1){};		
				\node at (1,1){};
				\node at (1,0){};
				\node at (0,-1){};
				
				\node[label,blue] at (-.8,-.2){$u$};
				\node[label,red] at (.2,.8){$v$};
				\node[label,green] at (-1.2,.8){$w$};
				
				\node[label] at (-.2,-.3) {$\emptyset$};
				\node[label] at (-1.3,-.3) {$\{i\}$};
				\node[label] at (-1.4,1.3) {$\{i,j\}$};
				\node[label] at (-.3,1.3) {$\{j\}$};
				\node[label] at (.6,1.3) {$\{j,k\}$};
				\node[label] at (.7,-.3) {$\{k\}$};
				\node[label] at (0,-1.4) {$\{l\}$};	
				\node[label,font=\Huge] at (0,4) {$\Delta$};			
			\end{scope}
			
		\end{tikzpicture}
	}
	\caption{A repeat of Figure \ref{fig:leveladj}.
	Three of the vertices are labeled $u,v,w$. We have ${\downarrow_1}(v)=\{v\}$, so $v\Uparrow u=u$ if $q(v)\prec q(u)$ and $v\Uparrow u=w$ if $q(u)\prec q(v)$.}\label{fig:ascent}
\end{figure}

In the remainder of this section we prove four lemmas about ascent.

\begin{lem}\label{lem:ascentequi}
	Ascent is $\La$-equivariant: $\lambda(v\Uparrow u)=\lambda v\Uparrow \lambda u$ for $\lambda\in\La$.
\end{lem}
\begin{proof}
Given vertices $u,v\in C$, the condition that $t(u)=\{i\}$ for some $i\in t(v)\jperp$ is equivalent to $u$ being a rank-1 vertex incomparable with $v$ but with some cube in $C$ containing both $u$ and $v$.
The notion of 1-downset is defined in terms of the poset structure on $\Delta^0$, so it follows from (\ref{down1Down}) and Lemma \ref{lem:downset}\ref{item:downsetdet} that ascent depends only on the poset structure on $\Delta^0$, the chamber structure on $\Delta$, the map $q$ and the order $\preceq$.
All of these things are preserved by $\La$, so the lemma follows.
\end{proof}

\begin{lem}\label{lem:ascentdown}
	Ascent is strictly $\preceq$-increasing: $\Lambda\cdot[v]\prec\Lambda\cdot[v\Uparrow u]$
\end{lem}
\begin{proof}
	We know from Lemma \ref{lem:distinctq} that $q(u)$ is distinct from $q(u')$ for all $u'\in{\downarrow_1}(v)$, so it follows from (\ref{down1Down}) that
	\begin{equation*}
\max(q({\downarrow_1}(v\Uparrow u))\triangle q({\downarrow_1}(v)))=q(u)\in q({\downarrow_1}(v\Uparrow u)).\qedhere
	\end{equation*}	
\end{proof}

\begin{lem}\label{lem:doubleascent}
Let $C$ be a chamber with distinct vertices $u_1,u_2,v$ such that $t(u_1)=\{i\}$, $t(u_2)=\{j\}$, $i,j\in t(v)\jperp$ are adjacent and $q(u_1)\prec q(u_2)$. Then $(v\Uparrow u_1)\Uparrow u_2=v\Uparrow u_2$.
\end{lem}
\begin{proof}
	First we note that $(v\Uparrow u_1)\Uparrow u_2$ is well-defined since $j\in (t(v)\cup\{i\})\jperp\subset t(v\Uparrow u_1)\jperp$ (using (\ref{ascenttypesandwich})).
	Now observe that any $u'\in{\downarrow_1}(v)$ with $q(u_2)\prec q(u')$ also satisfies $q(u_1)\prec q(u')$, so  $u'\in{\downarrow_1}(v\Uparrow u_1)$ by (\ref{down1Down}), and conversely we have that any $u'\in{\downarrow_1}(v\Uparrow u_1)$ with $q(u_2)\prec q(u')$ is an element of ${\downarrow_1}(v)$.
	We deduce that $(v\Uparrow u_1)\Uparrow u_2=v\Uparrow u_2$ since both vertices are in $C$ and are of the same type.
\end{proof}

\begin{lem}\label{lem:adjascent}
Let $C_1,C_2$ be $i$-adjacent chambers, let $v_1,v_2$ be vertices in $C_1,C_2$ respectively with $t(v_1)=t(v_2)=J\subset i\jperp$ (so $v_1,v_2$ are level-adjacent by Lemma \ref{lem:leveladj}). Then the following hold:
\begin{enumerate}
	\item\label{item:Downu} If $u=\wedge(C_1, C_2)$ then $v_1\Uparrow u=v_2\Uparrow u$.
	\item\label{item:Downu1u2} If $u_1,u_2$ are vertices in $C_1,C_2$ respectively of type $\{j\}$ with $j\in(J\cup\{i\})\jperp$, then $v_1\Uparrow u_1$ is level-adjacent to $v_2\Uparrow u_2$.
\end{enumerate}
\end{lem}
\begin{proof}
	It follows from Lemma \ref{lem:downset}\ref{item:downclasses} that
	\begin{equation}\label{qdownsequal}
		\{[v']\mid v'\in{\downarrow_1}(v_1)\}=\{[v']\mid v'\in{\downarrow_1}(v_2)\}.
	\end{equation}
	We now prove the two parts of the lemma.
	\begin{enumerate}
		\item Combining (\ref{ascenttype}), (\ref{qdownsequal}) and the fact that level-equivalent vertices are of the same type, we deduce that $t(v_1\Uparrow u)=t(v_2\Uparrow u)$.
		As $i\in t(v_1\Uparrow u)$, we see that $v_1\Uparrow u, v_2\Uparrow u\in C_1\cap C_2$, so they must be the same vertex.
		\item The vertices $u_1,u_2$ are level-adjacent by Lemma \ref{lem:leveladj}, so $q(u_1)=q(u_2)$.
		Combining (\ref{ascenttype}) and (\ref{qdownsequal}) again, we deduce that $t(v_1\Uparrow u_1)=t(v_2\Uparrow u_2)$.
		As $t(v_1\Uparrow u_1)\subset J\cup\{j\}\subset i\jperp$, we apply Lemma \ref{lem:leveladj} again to conclude that $v_1\Uparrow u_1$ is level-adjacent to $v_2\Uparrow u_2$.\qedhere
	\end{enumerate}	 
\end{proof}

\bigskip
\section{Hierarchy of residue-groupoids}\label{sec:hierresgroup}

The goal of this section is to prove Proposition \ref{prop:Deltagroupoid}, which we restate for the reader's convenience.

\theoremstyle{plain}
\newtheorem*{prop:Deltagroupoid}{Proposition \ref{prop:Deltagroupoid}}
\begin{prop:Deltagroupoid}
	There exists a $\La'$-invariant $\cC(\Delta)$-groupoid for some finite-index subgroup $\La'<\La$.
\end{prop:Deltagroupoid}

We will prove this by constructing a hierarchy of residue-groupoids corresponding to the hierarchy of level-equivalence classes from the previous section.
First we arrange one more property for the lattice $\La$.
This is a generalization of $\La$ \emph{having no holonomy} in the terminology of Haglund \cite[definition 5.6 and Theorem 7.2]{Haglund06}.

\begin{lem}\label{lem:trivfactor1}
	Replacing $\La$ by a finite-index subgroup if necessary, we may assume that, for each chamber-residue $\cC([v])=\cC(J\uperp,C)$, the action of $\La_{[v]}$ on $\cC(J\uperp,C)\cong\cC(J,C)\times\cC(J\jperp,C)$ is trivial on the first factor (where $\La_{[v]}$ is the $\La$-stabilizer of the class $[v]$, and the product decomposition is from Lemma \ref{lem:cC[v]}).
\end{lem}
\begin{proof}
We know that $\La_{[v]}$ preserves the product decomposition for $\cC(J\uperp,C)$ by Proposition \ref{prop:preserved}.
Thus we get a homomorphism $\La_{[v]}\to\mathfrak{S}(\cC(J,C))$, where $\mathfrak{S}(X)$ denotes the symmetric group of a set $X$.
By Lemma \ref{lem:CJC} we know that $\cC(J,C)$ bijects with a coset of the subgroup $\G_J$; but $J$ is spherical, so $\cC(J,C)$ is finite.
Hence the kernel $\hat{\La}_{[v]}$ of the homomorphism $\La_{[v]}\to\mathfrak{S}(\cC(J,C))$ has finite index in $\La_{[v]}$.
The subgroup $\hat{\La}_{[v]}$ is separable in $\La$ by Proposition \ref{prop:separable}, so by Lemma \ref{lem:separable}\ref{item:intH1} there exists a finite-index normal subgroup $\hat{\La}\triangleleft\La$ with $\La_{[v]}\cap\hat{\La}<\hat{\La}_{[v]}$.
This means that the $\hat{\La}$-stabilizer of the class $[v]$ acts trivially on the first factor of $\cC(J\uperp,C)\cong\cC(J,C)\times\cC(J\jperp,C)$. It remains to achieve this property simultaneously for all classes $[v]$.

The action of $\La$ preserves product decompositions for all chamber-residues of the form $\cC([v])$, so we have
$$\lambda \hat{\La}_{[v]}\lambda^{-1}=\hat{\La}_{[\lambda v]}$$
for all $v\in\Delta^0$ and $\lambda\in\La$.
There are only finitely many $\La$-orbits of level-equivalence classes since $\La$ acts cocompactly on $\Delta$, so applying the argument of the preceding paragraph to a set of orbit representatives yields a finite-index normal subgroup $\hat{\La}\triangleleft\La$ with $\La_{[v]}\cap\hat{\La}<\hat{\La}_{[v]}$ for all $v\in\Delta^0$, as required.
\end{proof}

We know from Proposition \ref{prop:preserved} that $\La$ preserves the families of sets $\{\cC(v)\mid v\in\Delta^0\}$ and $\{\cC([v])\mid v\in\Delta^0\}$, so it is natural to make the following definition (which generalizes the notion of a $\La'$-invariant $\cC(\Delta)$-groupoid from Section \ref{sec:hierclasses}).

\begin{defn}(Actions of $\La$ on the collections of $\cC(v)$-groupoids and $\cC([v])$-groupoids)\\\label{defn:actionresgroup}
	Let $\phi$ be a $\cC(v)$-groupoid and let $\lambda\in\La$. We define the $\cC(\lambda v)$-groupoid $\lambda\cdot\phi$ by the maps
	$$(\lambda\cdot\phi)_{\lambda C_1,\lambda C_2}:=\lambda\circ\phi_{C_1,C_2}\circ\lambda^{-1}$$
	for $C_1,C_2\in\cC(v)$.
	It is straightforward to check that $\lambda\cdot\phi$ satisfies all the properties of being a $\cC(\lambda v)$-groupoid.
	It is also straightforward to check that this defines an action of $\La$ on the collection of all $\cC(v)$-groupoids with $v$ ranging over all the vertices of $\Delta$.
	Analogously, if $\phi$ is a $\cC([v])$-groupoid and $\lambda\in\La$, we can define a $\cC([\lambda v])$-groupoid $\lambda\cdot\phi$, and this defines an action of $\La$ on the collection of all $\cC([v])$-groupoids for $v\in\Delta^0$.
\end{defn}

\begin{remk}\label{remk:lambdaadj}
	If we want to show that $\lambda\cdot\phi$ is equal to some other $\cC(\lambda v)$-groupoid $\psi$, then it suffices to check that the diagram
	\begin{equation}\label{lambdasquare}
		\begin{tikzcd}[
			ar symbol/.style = {draw=none,"#1" description,sloped},
			isomorphic/.style = {ar symbol={\cong}},
			equals/.style = {ar symbol={=}},
			subset/.style = {ar symbol={\subset}}
			]
			C_1\ar{d}[swap]{\lambda}\ar{r}{\phi_{C_1,C_2}}&C_2\ar{d}{\lambda}\\
			\lambda C_1\ar{r}[swap]{\psi_{\lambda C_1,\lambda C_2}}&\lambda C_2
		\end{tikzcd}
	\end{equation}
	commutes for all pairs of adjacent chambers $C_1,C_2$ in $\cC(v)$.
	The commutative square for an arbitrary pair of chambers $C_1,C_2\in\cC(v)$ can be obtained by composing a sequence of commutative squares for pairs of adjacent chambers corresponding to a gallery in $\cC(v)$ that joins $C_1$ and $C_2$.
	Of course the same is true for $\cC([v])$-groupoids.
\end{remk}

We now define what it means to have a hierarchy of residue-groupoids.

\begin{defn}(Hierarchy of residue-groupoids)\\\label{defn:hierarchy}
	Let $\cL'\subset\cL$ be a union of $\La$-orbits that is upward closed under $\preceq$ -- i.e. $\cL'/\La\ni\Lambda\cdot[u_1]\preceq\Lambda\cdot[u_2]$ implies $[u_2]\in\cL'$.
	Let $\La'<\La$ be a finite-index subgroup.
	A \emph{$\La'$-hierarchy of residue-groupoids on $\cL'$} is a collection of residue-groupoids $(\phi^{[v]})_{[v]\in\cL'}$, where $\phi^{[v]}$ is a $\cC([v])$-groupoid, such that:
	\begin{enumerate}
	\item (Equivariance) $\lambda\cdot\phi^{[v]}=\phi^{[\lambda v]}$ for all $[v]\in\cL'$ and $\lambda\in\La'$.
	
	\item (Restriction) Let $C$ be a chamber with vertices $u,v\in C$ such that $t(u)=\{i\}$, $i\in t(v)\jperp$ and $[v]\in\cL'$. If $C'$ is another chamber that is $i$-adjacent to $C$, then
	$$\phi^{[v]}_{C,C'}=\phi^{[v\Uparrow u]}_{C,C'}.$$
	\end{enumerate}
\end{defn}

\begin{remk}\label{remk:restriction}
	It is not hard to see that the restriction property is well-defined.
	Indeed $i\in t(v\Uparrow u)$ by (\ref{ascenttype}), so $v\Uparrow u\in C\cap C'$ and $C,C'\in\cC([v\Uparrow u])$.
	Also, $\Lambda\cdot[v]\prec\Lambda\cdot[v\Uparrow u]$ by Lemma \ref{lem:ascentdown}, so $[v\Uparrow u]\in\cL'$.
\end{remk}

\begin{remk}\label{remk:La'La''}
	Note that a $\La'$-hierarchy of residue-groupoids on $\cL'$ is also a $\La''$-hierarchy of residue-groupoids on $\cL'$ for any finite-index $\La''<\La'$.
\end{remk}

The rest of this section will be spent proving the following proposition.
We observe that Proposition \ref{prop:Deltagroupoid} follows since $\cC([v])=\cC(\Delta)$ for any rank-0 vertex $v$ (note that such $[v]$ is at the bottom of the hierarchy by Remark \ref{remk:extremehier}).

\begin{prop}\label{prop:fullhierarchy}
	There exists a $\La'$-hierarchy of residue-groupoids on the whole of $\cL$ for some finite-index $\La'< \La$.
\end{prop}

We prove this by working down the hierarchy of level-equivalence classes, noting that $\cL/\La$ is finite because $\La$ acts cocompactly on $\Delta$. As a base case, there is vacuously a $\La$-hierarchy of residue-groupoids on $\emptyset$.
It remains to prove the inductive step, so assume that we are given a $\La'$-hierarchy of residue-groupoids on $\cL'$, denoted by $(\phi^{[v]})_{[v]\in\cL'}$, and pick a class $[v]\in\cL-\cL'$ such that $\Lambda\cdot[v]$ is $\preceq$-maximal. Our task is to extend this to a $\La''$-hierarchy of residue-groupoids on $\cL'\cup \La\cdot[v]$ for some finite-index $\La''<\La'$.
We will do this in three parts:
\begin{enumerate}
	\item\label{item:C[v]groupoid} (Lemma \ref{lem:barpsi}) Construct a $\cC([v])$-groupoid that satisfies the restriction property from Definition \ref{defn:hierarchy}.
	This will be constructed by piecing together a $\cC(v)$-groupoid $\psi$ with the existing residue-groupoids in the $\La'$-hierarchy.
	\item (Lemma \ref{lem:La''}) Pass to a finite-index subgroup $\La''<\La'$ to ensure that the residue-groupoid from \ref{item:C[v]groupoid} is $\La''_{[v]}$-invariant. This involves subgroup separability and a certain ``groupoid holonomy'' of $\La'_{[v]}$.
	\item (Lemma \ref{lem:extendhierarchy}) Construct residue-groupoids for the other residues in $\La\cdot[v]$ that satisfy both the equivariance and restriction properties from Definition \ref{defn:hierarchy} with respect to $\La''$. This involves taking appropriate $\La'$-translates of the residue-groupoid from \ref{item:C[v]groupoid}.
\end{enumerate}  

Suppose $t(v)=J$ and let $\cC([v])=\cC(J\uperp,C)$ (Lemma \ref{lem:cC[v]}).
Once again we will make use of the product decomposition from Lemma \ref{lem:cC[v]}:
\begin{equation}\label{product}
\cC([v])=\cC(J\uperp,C)\cong\cC(J,C)\times\cC(J\jperp,C)
\end{equation}
Note that $\La'_{[v]}$ acts trivially on the first factor by Lemma \ref{lem:trivfactor1}.
To construct a $\cC([v])$-groupoid $\phi$, we first define $\phi$ on $J\jperp$-chamber-residues within $\cC([v])$ with the following lemma -- note that these $J\jperp$-chamber-residues correspond to sections $\{C_1\}\times\cC(J\jperp,C)$ in the product decomposition (\ref{product}), so each is stabilized by $\La'_{[v]}$.
(We remark that this step is vacuous if $J$ is a maximal spherical subset of $I$, as then $J\jperp=\emptyset$.)

\begin{lem}\label{lem:phiJjperp}
For a given chamber-residue $\cC(J\jperp,C')\subset\cC([v])$, there exists a unique $\cC(J\jperp,C')$-groupoid $\phi$ such that:
\begin{enumerate}
	\item\label{item:equi} (Equivariance) $\lambda\cdot\phi=\phi$ for all $\lambda\in\La'_{[v]}$.
	\item\label{item:rest} (Restriction) If $C_1,C_2\in\cC(J\jperp,C')$ are $i$-adjacent and $u,v_1\in C_1$ are vertices of types $\{i\}$ and $J$ respectively, then
	$$\phi_{C_1,C_2}=\phi^{[v_1\Uparrow u]}_{C_1,C_2}.$$
\end{enumerate}	
\end{lem}
\begin{proof}
	The maps $\phi_{C_1,C_2}$ are defined for adjacent chambers by \ref{item:rest} -- noting that $[v_1\Uparrow u]\in\cL'$ by Lemma \ref{lem:ascentdown}, so $\phi^{[v_1\Uparrow u]}$ is defined.
By Lemma \ref{lem:resgroup}, the collection of maps $(\phi_{C_1,C_2})$ extends to a $\cC(J\jperp,C')$-groupoid if properties \ref{item:d-'}--\ref{item:fixadj} are satisfied (and this $\cC(J\jperp,C')$-groupoid is uniquely determined), so we now check \ref{item:d-'}--\ref{item:fixadj}:
\begin{enumerate}[label=\wackyenum*]
	\item Each $\phi^{[v_1\Uparrow u]}_{C_1,C_2}$ preserves lower degrees of rank-1 vertices, so the maps $\phi_{C_1,C_2}$ do too.
	
	\item Let $C_1,C_2\in\cC(J\jperp,C')$ be $i$-adjacent, let $u=\wedge(C_1, C_2)$, and let $v_1,v_2\in[v]$ be in $C_1,C_2$ respectively.
	Let $w:=v_1\Uparrow u=v_2\Uparrow u\in C_1\cap C_2$ (Lemma \ref{lem:adjascent}\ref{item:Downu}).
	By construction we have
	$$\phi_{C_1,C_2}:=\phi^{[w]}_{C_1,C_2}\quad\text{and}\quad\phi_{C_2,C_1}:=\phi^{[w]}_{C_2,C_1},$$
	and these maps are inverse to each other since $\phi^{[w]}$ is a residue-groupoid.
	
	\item If $C_1,C_2,C_3$ are all in the same $\{i\}$-chamber-residue for $i\in J\jperp$, then the maps $\phi_{C_1,C_2},\phi_{C_2,C_3}$ and $\phi_{C_1,C_3}$ are all defined using $\phi^{[w]}$, where $w:=v_1\Uparrow u=v_2\Uparrow u=v_3\Uparrow u$, $u$ is the unique vertex in $C_1\cap C_2\cap C_3$ of type $\{i\}$, and $v_k\in C_k\cap[v]$ for $1\leq k\leq3$.
	We deduce that $\phi_{C_2,C_3}\circ\phi_{C_1,C_2}=\phi_{C_1,C_3}$ because the corresponding equation holds for $\phi^{[w]}$.
	
	\item\label{item:squarescase} Consider adjacent $i,j\in J\jperp$ and pairs of $i$-adjacent chambers $C_1,C_2$ and $C'_1,C'_2$, and suppose that the pairs $C_1,C'_1$ and $C_2,C'_2$ are $j$-adjacent.
	Let $v_1,v_2,v'_1,v'_2\in[v]$ be vertices in $C_1,C_2,C'_1,C'_2$ respectively.
	Put $u:=\wedge(C_1, C_2)$, $u':=\wedge(C'_1, C'_2)$, $x_1:=\wedge(C_1, C'_1)$ and $x_2:=\wedge(C_2, C'_2)$.
	Using Lemma \ref{lem:adjascent}\ref{item:Downu}, put
	\begin{align*}
w&:=v_1\Uparrow u=v_2\Uparrow u,\\
w'&:=v'_1\Uparrow u'=v'_2\Uparrow u',\\
y_1&:=v_1\Uparrow x_1=v'_1\Uparrow x_1,\\
y_2&:=v_2\Uparrow x_2=v'_2\Uparrow x_2.
	\end{align*}
An example is shown in Figure \ref{fig:lem}.
	Lemma \ref{lem:adjascent}\ref{item:Downu1u2} implies that $u,w,x_1,y_1$ are level-adjacent to $u',w',x_2,y_2$ respectively, so
	\begin{align*}
		q(u)&=q(u'),\\
		q(w)&=q(w'),\\
		q(x_1)&=q(x_2),\\
		q(y_1)&=q(y_2).
	\end{align*}
	By Lemma \ref{lem:distinctq} we have $q(u)\neq q(x_1)$. Say $q(x_1)\prec q(u)$ (the opposite case follows a symmetric argument).
	We can then apply Lemma \ref{lem:doubleascent} to deduce that $y_1\Uparrow u=w=y_2\Uparrow u$ and $y_1\Uparrow u'=w'=y_2\Uparrow u'$.
	Applying the restriction property of hierarchies to $\phi^{[y_1]}=\phi^{[y_2]}$ then yields
	\begin{equation*}
		\phi^{[y_1]}_{C_1,C_2}=\phi^{[w]}_{C_1,C_2}\quad\text{and}\quad \phi^{[y_1]}_{C'_1,C'_2}=\phi^{[w']}_{C'_1,C'_2}.
	\end{equation*}
	It follows that the diagrams	
	\begin{equation*}\label{phisquare}
		\begin{tikzcd}[
			ar symbol/.style = {draw=none,"#1" description,sloped},
			isomorphic/.style = {ar symbol={\cong}},
			equals/.style = {ar symbol={=}},
			subset/.style = {ar symbol={\subset}}
			]
			C_1\ar{d}[swap]{\phi_{C_1,C'_1}}\ar{r}{\phi_{C_1,C_2}}&C_2\ar{d}{\phi_{C_2,C'_2}}\\
			C'_1\ar{r}[swap]{\phi_{C'_1,C'_2}}&C'_2
		\end{tikzcd}
	\quad\text{and}\quad
	\begin{tikzcd}[
		ar symbol/.style = {draw=none,"#1" description,sloped},
		isomorphic/.style = {ar symbol={\cong}},
		equals/.style = {ar symbol={=}},
		subset/.style = {ar symbol={\subset}}
		]
		C_1\ar{d}[swap]{\phi^{[y_1]}_{C_1,C'_1}}\ar{r}{\phi^{[y_1]}_{C_1,C_2}}&C_2\ar{d}{\phi^{[y_1]}_{C_2,C'_2}}\\
		C'_1\ar{r}[swap]{\phi^{[y_1]}_{C'_1,C'_2}}&C'_2
	\end{tikzcd}
	\end{equation*}
are identical. Our goal is to prove that the left-hand diagram commutes, so this reduces to showing that the right-hand diagram commutes, but this follows since $\phi^{[y_1]}$ is a residue-groupoid.

\item Finally, each $\phi^{[v_1\Uparrow u]}_{C_1,C_2}$ fixes the intersection $C_1\cap C_2$ pointwise, so the same is true of the maps $\phi_{C_1,C_2}$.
\end{enumerate}

We now verify properties \ref{item:equi} and \ref{item:rest} from the lemma.
To show that $\lambda\cdot\phi=\phi$ for $\lambda\in\La'_{[v]}$ it suffices to check that diagram (\ref{lambdasquare}) from Remark \ref{remk:lambdaadj} commutes (with $\psi=\phi$).
But we defined the maps $\phi_{C_1,C_2}$ for pairs of adjacent chambers using \ref{item:rest}, and these will satisfy the commutative diagram because our existing hierarchy of residue-groupoids satisfies equivariance, and because ascent is $\La'$-equivariant (Lemma \ref{lem:ascentequi}).
Lastly, property \ref{item:rest} holds by construction.
\end{proof}

\begin{figure}[H]
	\centering
	\scalebox{1}{
		\begin{tikzpicture}[auto,node distance=2cm,
			thick,every node/.style={circle,draw,fill,inner sep=0pt,minimum size=7pt},
			every loop/.style={min distance=2cm},
			hull/.style={draw=none},
			]
			\tikzstyle{label}=[draw=none,fill=none]
			\tikzstyle{a}=[isosceles triangle,sloped,allow upside down,shift={(0,-.05)},minimum size=3pt]
			
				\draw(0,0)--(-2,1)--(0,2)--(2,1)--(0,0);
				\draw(0,0)--(0,-1.2)--(2,-.2)--(2,1);
				\draw(0,-1.2)--(-2,-.2)--(-2,1);
				\draw(-1,-.7)--(-1,.5)--(1,1.5);
				\draw(1,-.7)--(1,.5)--(-1,1.5);
				
				\node at (0,0){};
				\node at (-1,.5){};
				\node at (1,.5){};
				\node at (-2,1){};
				\node at (2,1){};
				\node at (-1,1.5){};
				\node at (1,1.5){};
				\node at (0,2){};
				\node at (-1,-.7){};
				\node at (1,-.7){};
				
				\node[red] at (-2,-.2){};
				\node[red] at (0,-1.2){};
				\node[red] at (2,-.2){};
				
				\node[label] at (.3,-.2){$v'_1$};
				\node[label] at (-1.3,.3){$y_1$};
				\node[label] at (1.3,.3){$w'$};
				\node[label] at (-2.3,.8){$v_1$};
				\node[label] at (2.3,.8){$v'_2$};
				\node[label] at (-1.3,1.7){$w$};
				\node[label] at (1.3,1.7){$y_2$};
				\node[label] at (.3,2.2){$v_2$};
				\node[label] at (-1.3,-.9){$x_1$};
				\node[label] at (1.3,-.9){$u'$};
			
		\end{tikzpicture}
	}
	\caption{The region of $\Delta$ relevant to part \ref{item:squarescase} in the proof of Lemma \ref{lem:phiJjperp}, shown for the case where $q(x_1)\prec q(u)\prec q(v_1)$ and $\rk(v_1)=1$. The picture shows four 3-cubes glued together, one in each of the chambers $C_1,C_2,C'_1,C'_2$. The centers of these chambers are shown in red.}\label{fig:lem}
\end{figure}
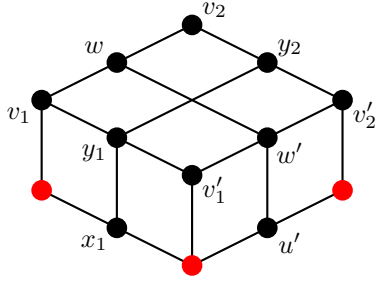

We apply Lemma \ref{lem:phiJjperp} to all $J\jperp$-chamber-residues in $\cC([v])$ and write $(\phi_{C_1,C_2})$ for the collection of all maps obtained.

If $J=\emptyset$ then $J\jperp=I$ and $\cC([v])=\cC(\Delta)$, so the maps $(\phi_{C_1,C_2})$ would define a $\cC(\Delta)$-groupoid $\phi$ that is invariant under $\La'$. In this case $\Lambda\cdot[v]$ would be the unique $\preceq$-smallest class in $\cL/\La$ (Remark \ref{remk:extremehier}), so the proof of Proposition \ref{prop:fullhierarchy} would be finished.

If $J\neq\emptyset$ however, then we are not yet finished with the inductive step in Proposition \ref{prop:fullhierarchy}, so assume $J\neq\emptyset$ for the rest of this section.
Next, we describe a procedure that extends the maps $(\phi_{C_1,C_2})$ to a $\cC([v])$-groupoid.
The idea here is to take a $\cC(v)$-groupoid $\psi$ -- noting that $\cC(v)$ corresponds to a section $\cC(J,C)\times\{C_2\}$ in the product decomposition (\ref{product}) -- and extend it using the maps $(\phi_{C_1,C_2})$.
Remember that $\cC(v)$-groupoids do exist by Lemma \ref{lem:Gammaresgroup}.

\begin{lem}\label{lem:barpsi}
	Given a $\cC(v)$-groupoid $\psi$, there is a unique extension to a $\cC([v])$-groupoid $\bar{\psi}$, such that $\bar{\psi}_{C_1,C_2}=\phi_{C_1,C_2}$ whenever $C_1,C_2$ are in the same $J\jperp$-chamber-residue.
	
	Moreover, it follows from Lemma \ref{lem:phiJjperp} that $\bar{\psi}$ satisfies the restriction property from Definition \ref{defn:hierarchy}: if $C_1,C_2\in\cC([v])$ are $i$-adjacent and $u,v_1\in C_1$ are vertices of types $\{i\}$ and $J$ respectively, then
	$$\bar{\psi}_{C_1,C_2}=\phi^{[v_1\Uparrow u]}_{C_1,C_2}.$$
\end{lem}
\begin{proof}
	Let $\psi$ be a $\cC(v)$-groupoid.
	Let's first consider what the maps between adjacent chambers in $\bar{\psi}$ must be.
	Let $C_1,C_2\in\cC([v])$ be $i$-adjacent.
	If $i\in J\jperp$, then the assumption of the lemma forces
	\begin{equation}\label{barpsiJp}
		\bar{\psi}_{C_1,C_2}:=\phi_{C_1,C_2}.
	\end{equation}
	Meanwhile, if $i\in J$ then the composition property for residue-groupoids determines $\bar{\psi}_{C_1,C_2}$.
	Specifically, if $C'_1\in\cC(v)\cap\cC(J\jperp,C_1)$ and $C'_2\in\cC(v)\cap\cC(J\jperp,C_2)$ (these chambers exist and are unique by Lemma \ref{lem:approxjperp}), then we have the following commutative diagram:
	\begin{equation}\label{barpsiJ}
		\begin{tikzcd}[
			ar symbol/.style = {draw=none,"#1" description,sloped},
			isomorphic/.style = {ar symbol={\cong}},
			equals/.style = {ar symbol={=}},
			subset/.style = {ar symbol={\subset}}
			]
			C_1\ar{d}[swap]{\phi_{C_1,C'_1}}\ar{r}{\bar{\psi}_{C_1,C_2}}&C_2\ar{d}{\phi_{C_2,C'_2}}\\
			C'_1\ar{r}[swap]{\psi_{C'_1,C'_2}}&C'_2
		\end{tikzcd}
	\end{equation}
	To prove that the collection of maps $\bar{\psi}_{C_1,C_2}$ from (\ref{barpsiJp}) and (\ref{barpsiJ}) can be extended (uniquely) to a $\cC([v])$-groupoid $\bar{\psi}$, it suffices to check that these maps satisfy properties \ref{item:d-'}--\ref{item:fixadj} from Lemma \ref{lem:resgroup}:
\begin{enumerate}[label=\wackyenum*]
	\item The maps $\bar{\psi}_{C_1,C_2}$ preserve lower degrees of rank-1 vertices because all of the maps in (\ref{barpsiJ}) and (\ref{barpsiJp}) do.
	\item The equation $\bar{\psi}_{C_1,C_2}=\bar{\psi}_{C_2,C_2}^{-1}$ follows because all of the maps in (\ref{barpsiJ}) and (\ref{barpsiJp}) come from residue-groupoids, and residue-groupoids satisfy \ref{item:inverse}.
	\item Let $C_1,C_2,C_3$ be chambers in an $\{i\}$-chamber-residue in $\cC([v])$. We want $\bar{\psi}_{C_2,C_3}\circ\bar{\psi}_{C_1,C_2}=\bar{\psi}_{C_1,C_3}$.
	If $i\in J$ then this follows from (\ref{barpsiJ}) and the fact that \ref{item:coi} holds for $\psi$.
	If $i\in J\jperp$ then this follows from (\ref{barpsiJp}) and the fact that \ref{item:coi} holds for the maps $(\phi_{C_1,C_2})$.
	\item Let $i,j\in J\uperp$ be adjacent and let $C_1,C_2,C'_1,C'_2\in\cC([v])$ be such that $C_1,C_2$ and $C'_1,C'_2$ are $i$-adjacent and $C_1,C'_1$ and $C_2, C'_2$ are $j$-adjacent. We want the following diagram to commute:
		\begin{equation}\label{barsquare}
		\begin{tikzcd}[
			ar symbol/.style = {draw=none,"#1" description,sloped},
			isomorphic/.style = {ar symbol={\cong}},
			equals/.style = {ar symbol={=}},
			subset/.style = {ar symbol={\subset}}
			]
			C_1\ar{d}[swap]{\bar{\psi}_{C_1,C'_1}}\ar{r}{\bar{\psi}_{C_1,C_2}}&C_2\ar{d}{\bar{\psi}_{C_2,C'_2}}\\
			C'_1\ar{r}[swap]{\bar{\psi}_{C'_1,C'_2}}&C'_2
		\end{tikzcd}
	\end{equation}
If $i,j\in J$ then this follows from (\ref{barpsiJ}) and the fact that \ref{item:square} holds for $\psi$.
If $i,j\in J\jperp$ then this follows from (\ref{barpsiJp}) and the fact that \ref{item:square} holds for the maps $(\phi_{C_1,C_2})$.
If $i\in J$ and $j\in J\jperp$ then (\ref{barsquare}) is identical to (\ref{barpsiJ}) after making substitutions. Lastly, the case $i\in J\jperp$ and $j\in J$ follows from the previous case by symmetry.

\item Let $C_1,C_2\in\cC([v])$ be $i$-adjacent.
We want $\bar{\psi}_{C_1,C_2}$ to fix $C_1\cap C_2$ pointwise.
If $i\in J\jperp$ then this follows from (\ref{barpsiJp}) and the fact that $\phi_{C_1,C_2}$ fixes $C_1\cap C_2$ pointwise.

Now suppose $i\in J$.
We induct by the length of a gallery joining $C_1$ to $\cC(v)$.
For the base case, if $C_1\in\cC(v)$ then (\ref{barpsiJ}) implies that $\bar{\psi}_{C_1,C_2}=\psi_{C_1,C_2}$, and we are done because $\psi$ satisfies the intersection property of residue-groupoids.
For the general case, take a pair of $i$-adjacent chambers $C'_1,C'_2$ such that $C_1,C'_1$ and $C_2,C'_2$ are $j$-adjacent for some $j\in J\jperp$, and such that $C'_1$ is joined to $\cC(v)$ by a shorter gallery than $C_1$ (such chambers exist by the product structure on $\cC([v])$).
The map $\bar{\psi}_{C'_1,C'_2}$ fixes $C'_1\cap C'_2$ pointwise by our induction hypothesis.

Let $w=\wedge(C_1, C_2)$, which is a vertex of type $\{i\}$ (Lemma \ref{lem:min}\ref{item:iadjwedge}). We know from Lemma \ref{lem:stayj} that $w':=\phi_{C_1,C'_1}(w)$ is also of type $\{i\}$, hence $w'\in C'_1\cap C'_2$. A second application of Lemma \ref{lem:stayj} implies that $\phi_{C_2,C'_2}(w)$ is of type $\{i\}$ as well, and also lies in $C'_1\cap C'_2$, so $w'=\phi_{C_2,C'_2}(w)=\wedge(C'_1,C'_2)$. The maps $\phi_{C_1,C'_1},\phi_{C_2,C'_2}$ preserve poset structure (Remark \ref{remk:rankgroup}), so by Lemma \ref{lem:min}\ref{item:wCcapC'} they both map $C_1\cap C_2$ to $C'_1\cap C'_2$.

Now consider the vertices $u_1:=\wedge(C_1,C'_1)$ and $u_2:=\wedge(C_2, C'_2)$, which are of type $\{j\}$, and the vertex $x\in C_1\cap C_2$ of type $J$.
Lemma \ref{lem:adjascent}\ref{item:Downu1u2} tells us that $y_1:=x\Uparrow u_1$ is level-adjacent to $y_2:= x\Uparrow u_2$, so $[y_1]=[y_2]$. Then Lemma \ref{lem:phiJjperp}\ref{item:rest} and (\ref{barpsiJp}) imply that
\begin{equation}\label{barpsiphix}
\bar{\psi}_{C_1,C'_1}=\phi^{[y_1]}_{C_1,C'_1}\quad\text{and}\quad\bar{\psi}_{C_2,C'_2}=\phi^{[y_1]}_{C_2,C'_2}.
\end{equation}
We have $\{j\}\subset t(y_1)\subset J\cup\{j\}$ by (\ref{ascenttypesandwich}), so $i,j\in t(y_1)\uperp$ and the residue-groupoid $\phi^{[y_1]}$ contains the following commutative square of maps:
	\begin{equation*}
	\begin{tikzcd}[
		ar symbol/.style = {draw=none,"#1" description,sloped},
		isomorphic/.style = {ar symbol={\cong}},
		equals/.style = {ar symbol={=}},
		subset/.style = {ar symbol={\subset}}
		]
		C_1\ar{d}[swap]{\phi^{[y_1]}_{C_1,C'_1}}\ar{r}{\phi^{[y_1]}_{C_1,C_2}}&C_2\ar{d}{\phi^{[y_1]}_{C_2,C'_2}}\\
		C'_1\ar{r}[swap]{\phi^{[y_1]}_{C'_1,C'_2}}&C'_2
	\end{tikzcd}
\end{equation*}
The horizontal maps $\phi^{[y_1]}_{C_1,C_2},\phi^{[y_1]}_{C'_1,C'_2}$ restrict to identity maps on $C_1\cap C_2,C'_1\cap C'_2$ respectively by the intersection property of residue-groupoids, so we deduce that the vertical maps $\phi^{[y_1]}_{C_1,C'_1}$ and $\phi^{[y_1]}_{C_2,C'_2}$ restrict to the same map $C_1\cap C_2\to C'_1\cap C'_2$.
Applying (\ref{barpsiphix}) tells us that $\bar{\psi}_{C_1,C'_1}$ and $\bar{\psi}_{C_2,C'_2}$ also restrict to the same map $C_1\cap C_2\to C'_1\cap C'_2$.
Finally, since $\bar{\psi}_{C'_1,C'_2}$ restricts to the identity on $C'_1\cap C'_2$, it follows from (\ref{barsquare})  that $\bar{\psi}_{C_1,C_2}$ restricts to the identity on $C_1\cap C_2$, as required.\qedhere
\end{enumerate}
\end{proof}

Next, we will pass to a finite-index subgroup $\La''<\La'$ to ensure that the residue-groupoids $\bar{\psi}$ from Lemma \ref{lem:barpsi} are $\La''_{[v]}$-invariant.

\begin{defn}(Groupoid holonomy at $\cC(v)$)\\
Given a $\cC([v])$-groupoid $\psi$, we let $\psi|_{\cC(v)}$ denote the $\cC(v)$ groupoid obtained by restricting to $\cC(v)$.
Let $\cR\cG(v)$ denote the collection of all $\cC(v)$-groupoids, which is finite since $\cC(v)$ is finite and each chamber is finite. Let $\mathfrak{S}(\cR\cG(v))$ denote the symmetric group on $\cR\cG(v)$.
We define the \emph{groupoid holonomy of $\La'$ at $\cC(v)$} to be the map
\begin{align*}
\Upsilon:\La'_{[v]}&\to\mathfrak{S}(\cR\cG(v))\\
\lambda&\mapsto(\psi\mapsto(\lambda\cdot\bar{\psi})|_{\cC(v)}),
\end{align*}
where $\bar{\psi}$ is the extension of $\psi$ defined by Lemma \ref{lem:barpsi}.
\end{defn}

\begin{lem}\label{lem:resbar}
For $\lambda\in\La'_{[v]}$ and $\psi\in\cR\cG(v)$, we have $\lambda\cdot\bar{\psi}=\overline{((\lambda\cdot\bar{\psi})|_{\cC(v)})}$.
\end{lem}
\begin{proof}
Lemma \ref{lem:phiJjperp}\ref{item:equi} implies that $(\lambda\cdot\bar{\psi})_{C_1,C_2}=\phi_{C_1,C_2}$ for any chambers $C_1,C_2$ in the same $J\jperp$-chamber-residue, so the result follows from the uniqueness in Lemma \ref{lem:barpsi}.
\end{proof}

\begin{lem}
	The groupoid holonomy $\Upsilon$ is a homomorphism.
\end{lem}
\begin{proof}
It is clear that $\Upsilon(1)$ is trivial.
Let $\lambda_1,\lambda_2\in\La'_{[v]}$ and $\psi\in\cR\cG(v)$.
By Lemma \ref{lem:resbar} we have
\begin{align*}
	\Upsilon(\lambda_1)\Upsilon(\lambda_2)(\psi)&=(\lambda_1\cdot(\lambda_2\cdot\bar{\psi}))|_{\cC(v)}\\
	&=((\lambda_1\lambda_2)\cdot\bar{\psi})|_{\cC(v)}\\
	&=\Upsilon(\lambda_1\lambda_2)(\psi).\qedhere
\end{align*}
\end{proof}

\begin{lem}\label{lem:La''}
	There is a finite-index subgroup $\La''<\La'$ (normal in $\La$) such that, for all $\psi\in\cR\cG(v)$, the residue-groupoid $\bar{\psi}$ is $\La''_{[v]}$-invariant.
\end{lem}
\begin{proof}
We know that $\ker(\Upsilon)$ has finite index in $\La'_{[v]}$, which in turn has finite index in $\La_{[v]}$, so $\ker(\Upsilon)$ is separable in $\La$ by Proposition \ref{prop:separable}.
By Lemma \ref{lem:separable}\ref{item:intH1} there exists a finite-index normal subgroup $\La''\triangleleft\La$ such that $\La''_{[v]}<\ker(\Upsilon)$. Intersecting with $\La'$ if necessary, we may assume that $\La''<\La'$.
As $\La''_{[v]}<\ker(\Upsilon)$, it follows from Lemma \ref{lem:resbar} that $\bar{\psi}$ is $\La''_{[v]}$-invariant for any $\psi\in\cR\cG(v)$.
\end{proof}

We complete the inductive step of Proposition \ref{prop:fullhierarchy} with the following lemma.

\begin{lem}\label{lem:extendhierarchy}
	The $\La'$-hierarchy of residue-groupoids $(\phi^{[u]})_{[u]\in\cL'}$ extends to a $\La''$-hierarchy of residue-groupoids on $\cL'\cup \La\cdot[v]$ for some finite-index $\La''<\La'$.
\end{lem}
\begin{proof}	
Put $\phi=\bar{\psi}$ for some $\psi\in\cR\cG(v)$ and let $\La''<\La'$ be the subgroup from Lemma \ref{lem:La''}.
It suffices to define residue-groupoids for the orbit $\La'\cdot[v]$ that satisfy the equivariance and restriction properties for a $\La''$-hierarchy of residue groupoids. Indeed one could run the same argument for each $\La'$-orbit in $\La\cdot[v]$ and then take the intersection of the corresponding groups $\La''$ (a finite intersection, so the resulting group would still have finite index in $\La$).
Let $(\lambda_k)$ be a set of left coset representatives for $\La''\La'_{[v]}$ in $\La'$ (note that $\La''\La'_{[v]}$ is a subgroup since $\La''$ is normal in $\La$).
For each $\lambda_k$ and each $\lambda''\in\La''$ define the residue-groupoid
\begin{equation}\label{philamlam}
	\phi^{[\lambda_k\lambda'' v]}:=(\lambda_k\lambda'')\cdot\phi.
\end{equation}
We must check that (\ref{philamlam}) is consistent, so suppose $[\lambda_k\lambda''_1v]=[\lambda_l\lambda''_2v]$ for two classes as above.
Then
$$\lambda_k\lambda''_1\in\lambda_l\lambda''_2\La'_{[v]},$$
so $\lambda_k=\lambda_l$ and $\lambda''_1\in\lambda''_2\La''_{[v]}$.
We have $\lambda''_1\cdot\phi=\lambda''_2\cdot\phi=\phi$ by Lemma \ref{lem:La''}, so $(\lambda_k\lambda''_1)\cdot\phi=(\lambda_l\lambda''_2)\cdot\phi$ as required.

It remains to check that these residue-groupoids satisfy the equivariance and restriction properties for a $\La''$-hierarchy of residue groupoids:
\begin{enumerate}
	\item (Equivariance) Given $\lambda\in\La''$ and a class $[\lambda_k\lambda''v]$ as in (\ref{philamlam}), we have $\lambda\lambda_k\lambda''=\lambda_k(\lambda_k^{-1}\lambda\lambda_k)\lambda''$ with $(\lambda_k^{-1}\lambda\lambda_k)\lambda''\in\La''$ since $\La''$ is normal in $\La'$.
	Therefore (\ref{philamlam}) tells us that
	\begin{align*}
		\phi^{[\lambda\lambda_k\lambda'' v]}&=(\lambda\lambda_k\lambda'')\cdot\phi\\
		&=\lambda\cdot((\lambda_k\lambda'')\cdot\phi)\\
		&=\lambda\cdot\phi^{[\lambda_k\lambda''v]}.
	\end{align*}
\item (Restriction) The restriction property holds for $\phi^{[v]}=\phi$ by Lemma \ref{lem:barpsi}.
Ascent is $\La$-equivariant (Lemma \ref{lem:ascentequi}), and the residue groupoids $\phi^{[v_1\Uparrow u]}$ are from higher up the hierarchy (Lemma \ref{lem:ascentdown}) so satisfy $\La'$-equivariance, thus we deduce that the restriction property holds for all residue-groupoids from (\ref{philamlam}).\qedhere
\end{enumerate}
\end{proof}

\bigskip
\section{Typed atlases}\label{sec:atlas}

In this section we complete the proof of Theorem \ref{thm:Delta}, using the $\La'$-invariant $\cC(\Delta)$-groupoid from Proposition \ref{prop:Deltagroupoid} to show that $\La$ and $\G$ are weakly commensurable.
Key to this argument will be the notion of typed atlases, which in turn depends on a more general notion of typing map; roughly this will be a map $\Delta^0\to\bar{N}$ that behaves locally in the same way as the standard typing map $t$ from Definition \ref{defn:building}. From now on we will denote this standard typing map by $t_\G$ (as it was defined using the action of $\G$ on $\Delta$) and a general typing map by $t$, which is defined precisely as follows.
Other notions that were defined earlier -- such as rank, lower degree and chamber-residues -- remain the same, so these are defined in terms of the standard typing map $t_\G$. 

\begin{defn}(Typing map)\\\label{defn:typingmap}
	A \emph{typing map} is a map $t:\Delta^0\to\bar{N}$ such that
	\begin{enumerate}
		\item\label{item:toC*} For each chamber $C$ there is a cubical isomorphism $\phi:C\to C_*$ to the base chamber $C_*$ that preserves centers and satisfies $t(v)=t_\G(\phi v)$ for all vertices $v\in C$.
		\item\label{item:d-prod} For each $i\in I$ and each vertex $v\in\Delta^0$ with $t(v)=\{i\}$, the lower degree $d^-(v)$ is equal to $|G_i|$.
	\end{enumerate}
\end{defn}

\begin{lem}\label{lem:tGamma}
	The standard typing map $t_\G$ satisfies Definition \ref{defn:typingmap}.
\end{lem}
\begin{proof}
	Each chamber is of the form $C_\gamma$ for some $\gamma\in\G$, and $\gamma:C_*\to C_\gamma$ is a cubical isomorphism preserving centers and $t_\G$.
	For a vertex $v$ with $t_\G(v)=\{i\}$ we know that $\rk(v)=1$, so $E^-(v)$ consists of the edges that join $v$ to a rank-0 vertex. Hence $d^-(v)=|E^-(v)|$ is equal to the number of chambers containing $v$, and this equals $|G_i|$ by Lemmas \ref{lem:CJC} and \ref{lem:cCv}.
\end{proof}

\begin{lem}\label{lem:fttot'}
	Let $t$ and $t'$ be typing maps and let $C,C'\in\cC(\Delta)$.
	Then there is a unique cubical isomorphism $f:C\to C'$ such that $t(v)=t'(fv)$ for all vertices $v\in C$.
	Moreover, $f$ preserves centers, rank and poset structure.
\end{lem}
\begin{proof}
	Such a map $f$ can be obtained by composing the maps $C,C\to C_*$ from Definition \ref{defn:typingmap}\ref{item:toC*}. Uniqueness follows since $t,t'$ are injective on the vertex sets of $C,C'$ respectively.
	The maps $C,C\to C_*$ preserve centers by definition, so $f$ does as well.
	The edges in $C$ correspond exactly to the pairs of vertices $\{u,v\}$ such that $t_\G(u)=t_\G(v)\cup\{i\}$ for some $i\in I$, and the center is of type $\emptyset$, so the rank $\rk(v):=|t_\G(v)|$ of a vertex $v\in C$ is equal to the length of a shortest edge path joining it to the center of $C$. Once can also deduce that $u\leq v$ for $u\in C$ if and only if $u$ belongs to one of the shortest edge paths joining $v$ to the center of $C$.
	As $f$ preserves centers, it follows that it also preserves rank and poset structure.
\end{proof}

\begin{cor}\label{cor:ranktype}
	$|t(v)|=\rk(v)$ for any $v\in\Delta^0$ and any typing map $t$.
	In particular, Definition \ref{defn:typingmap}\ref{item:d-prod} is a statement about rank-1 vertices.
\end{cor}
\begin{proof}
	Apply Lemma \ref{lem:fttot'} with $t'=t_\G$.
\end{proof}

\begin{lem}\label{lem:typinglevel}
	If $t$ is a typing map and $v_1\approx v_2$ then $t(v_1)=t(v_2)$. 
\end{lem}
\begin{proof}
	It suffices to consider the case where $v_1,v_2$ are level-adjacent. Suppose $v_1,v_2$ are contained in adjacent chambers $C_1,C_2$ respectively, with $v_1,v_2\notin C_1\cap C_2$, and with some vertex $v_1,v_2\leq u\in C_1\cap C_2$ such that $\rk(v_1)=\rk(v_2)=\rk(u)-1$.
	We know from Lemma \ref{lem:leveladj} that $t_\G(v_1)=t_\G(v_2)=J$ say, and that $t_\G(u)=J\cup\{i\}$ for some $i\in J\jperp$. Since $v_1\notin C_1\cap C_2$, we deduce from Lemma \ref{lem:min}\ref{item:wCcapC'} and \ref{item:iadjwedge} that $C_1,C_2$ are $i$-adjacent.	
	Let
	$$\Delta^-(u):=\{v\in \Delta^0\mid v\leq u\text{ and }\rk(v)=\rk(u)-1\}.$$	
	We know that $\Delta^-(u)\cap C_1$ consists of one vertex $v$ with $t_\G(v)=t_\G(u)-\{j\}$ for each $j\in t_\G(u)$. Furthermore, by Lemma \ref{lem:min}\ref{item:wCcapC'} and \ref{item:iadjwedge}, $v_1$ is the only element of $\Delta^-(u)\cap C_1$ that is not in $C_1\cap C_2$. 
	By property \ref{item:toC*} of typing maps, we also know that $\Delta^-(u)\cap C_1$ consists of one vertex $v$ with $t(v)=t(u)-\{j\}$ for each $j\in t(u)$, hence
	$$t(v_1)=t(u)-\bigcap_{v\in\Delta^-(u)\cap C_1\cap C_2}t(v).$$
	The same is true for $t(v_2)$, so $t(v_1)=t(v_2)$ as required.
\end{proof}

\begin{lem}\label{lem:ijadjacent}
	Let $t$ be a typing map and let $u,v$ be rank-1 vertices in a chamber $C$ with $t(u)=\{i\}$ and $t(v)=\{j\}$.
	Then the property of $i,j$ being adjacent is independent of the choice of $t$.
\end{lem}
\begin{proof}
	By property \ref{item:toC*} of typing maps, we see that $i,j$ are adjacent if and only if there is a rank-2 vertex $w\in C$ with $u,v\leq w$ -- and in this case $t(w)=\{i,j\}$.
\end{proof}

We now define a notion of typed atlas, which generalizes Haglund's notion of atlas \cite[Definition 6.1]{Haglund06}.
(In fact Haglund's atlases coincide with typed atlases for which the typing map is the standard one.)
Roughly speaking, a typed atlas builds on the data given by a typing map by prescribing a family of local actions of the groups $G_i$ on $\Delta$, and it does this in a way that mimics the actions of the conjugates of the $G_i$ when viewed as subgroups of $\G$.

\begin{defn}(Typed atlas)\\\label{defn:typedatlas}
	A \emph{typed atlas} is a pair $(t,\cA)$ where $t$ is a typing map and $\cA$ is a collection of homomorphisms $(\cA_v)$, where $v$ ranges over rank-1 vertices in $\Delta$, such that:
	\begin{itemize}
		\item $\cA_v$ is a homomorphism
		$$\cA_v:G_i\to\mathfrak{S}(\cC(v))$$
		for $v\in\Delta^0$ with $t(v)=\{i\}$, such that $G_i$ acts simply transitively on $\cC(v)$. 
		\item If $v\approx v'$, then $\cA_v$ coincides with $\cA_{v'}$ under the natural identification $\cC(v)\cong\cC(v')$ (coming from Lemma \ref{lem:cC[v]}). 
		Note that $t(v)=t(v')=\{i\}$ in this case (Lemma \ref{lem:typinglevel}), so $\cA_v$ and $\cA_{v'}$ are both actions of $G_i$.
	\end{itemize}	
	We call $\cA$ an \emph{atlas} for $t$.	
\end{defn}

\begin{remk}\label{remk:=Gi}
	The residue $\cC(v)$ in Definition \ref{defn:typedatlas} has size $|G_i|$ because $|G_i|=d^-(v)$ (applying the second property of typing maps to $t$) and $d^-(v)=\cC(v)$ (immediate from the definition of $d^-$).
	Thus it is possible for $G_i$ to act simply transitively on $\cC(v)$.
\end{remk}

\begin{remk}
	In the case where $\G$ is a right-angled Coxeter group and $\Delta$ the associated Davis complex, we know that the groups $G_i$ all have order two, so in this case there is only one atlas for each typing map.
\end{remk}

The reason for defining typed atlases is so that we can establish an action of $\Aut_{\rk}(\Delta)$ on the set of typed atlases. We will see later that this action is closely related to the conjugation action of $\Aut_{\rk}(\Delta)$ on the set of its uniform lattices -- which is in turn connected to weak commensurability of uniform lattices.

\begin{defn}(Action on the set of typed atlases)\\
The structures used in the definitions of typing map and typed atlas depend on the standard typing map $t_\G$, however these structure are also preserved by $\Aut_{\rk}(\Delta)$ (Proposition \ref{prop:preserved}), so we get an action of $\Aut_{\rk}(\Delta)$ on the set of typed atlases. More precisely this action is defined as follows: given $f\in\Aut_{\rk}(\Delta)$ and a typed atlas $(t,\cA)$, we define $f_*(t,\cA)=(t',\cA')$ by $t'(fv):=t(v)$ for $v\in\Delta^0$ and $\cA'_{fv}(g)(fC):=f\cA_{v}(g)(C)$ for $v\in\Delta^0$ with $t(v)=\{i\}$, $g\in G_i$ and $C\in\cC(v)$.
One can easily check that $f_*(t,\cA)$ is a typed atlas using Proposition \ref{prop:preserved}.
\end{defn}

The next step is to construct typed atlases that are invariant under (finite-index subgroups of) $\La$ and $\G$.
This is where we make use of the $\La'$-invariant $\cC(\Delta)$-groupoid from Proposition \ref{prop:Deltagroupoid}.

\begin{defn}(Standard typed atlas)\\
Define the \emph{standard typed atlas} to be the typed atlas $(t_\G,\cA_\G)$, with $t_\G$ as already defined and $\cA_\G$ defined as follows.
Given a rank-1 vertex $v\in\Delta^0$ with $t_\G(v)=\{i\}$, a chamber $C_\gamma\in\cC(v)$ (as in Definition \ref{defn:building}), and an element $g\in G_i<\G$, we define
$$\cA_{\G,v}(g)(C_\gamma):=C_{\gamma g^{-1}}.$$
\end{defn}

\begin{lem}\label{lem:Gammatypedatlas}
	$(t_\G,\cA_\G)$ is a well-defined typed atlas, and it is preserved by $\G$.
\end{lem}
\begin{proof}
	The standard typing map $t_\G$ is preserved by $\G$ by construction (Definition \ref{defn:building}).
	
	Given a rank-1 vertex $v\in\Delta^0$ with $t_\G(v)=\{i\}$, a chamber $C_\gamma\in\cC(v)$, and an element $g\in G_i<\G$, we have
	$$\cA_{\G,v}(g)(C_\gamma):=C_{\gamma g^{-1}}.$$
	This is in the same $\{i\}$-chamber-residue as $C_\gamma$ by Lemma \ref{lem:CJC}, and $\cA_{\G,v}$ clearly defines a simply transitive (left) action of $G_i$ on $\cC(v)$.
	
	Now suppose $v\approx v'$.
	Fix a chamber $C_\gamma\in\cC(v)$.
	Then Lemmas \ref{lem:CJC} and \ref{lem:cC[v]} tell us that there is $\gamma_1\in \G_{i\jperp}$ such that $\cC(v')$ consists precisely of the chambers $C_{\gamma \gamma_1\gamma_2}$ where $\gamma_2\in G_i$.
	Moreover, the identification $\cC(v)\cong\cC(v')$ is given by $C_{\gamma\gamma_2}\mapsto C_{\gamma \gamma_1\gamma_2}$, so it follows that $\cA_{\G,v}$ and $\cA_{\G,v'}$ coincide under this identification.
	
	The collection of homomorphisms $(\cA_{\G,v})$ thus defines an atlas $\cA_\G$ for $t_\G$. Each homomorphism $\cA_{\G,v}$ is defined using right-multiplication, while the action of $\G$ on $\cC(\Delta)$ is defined using left multiplication ($\gamma'C_\gamma=C_{\gamma'\gamma}$), these operations commute with each other so one easily verifies that $\G$ preserves the atlas $\cA_\G$.
\end{proof}

\begin{lem}\label{lem:Latypedatlas}
	There is a typed atlas $(t,\cA)$ preserved by a finite-index subgroup $\hat{\La}<\La$.
\end{lem}
\begin{proof}
Proposition \ref{prop:Deltagroupoid} gives us a finite-index subgroup $\La'<\La$ and a $\La'$-invariant $\cC(\Delta)$-groupoid $\phi$.
We define $t:\Delta\to\bar{N}$ so that
\begin{equation}\label{t}
t(v)=t_\G(\phi_{C,C_*}(v))
\end{equation}  
for each chamber $C$ and all vertices $v\in C$.
By the composition and intersection properties of residue-groupoids (Definition \ref{defn:resgroup}) we know that $\phi_{C_1,C_*}(v)=\phi_{C_2,C_*}(v)$ if $C_1,C_2\in\cC(v)$, so (\ref{t}) provides a consistent definition for a map $t:\Delta\to\bar{N}$.
The map $t$ satisfies property \ref{item:toC*} of typing maps because of (\ref{t}), and it satisfies property \ref{item:d-prod} because residue-groupoids preserve lower degrees of rank-1 vertices. Thus $t$ is a valid typing map.

The group $\La'$ may not preserve $t$, so we will pass to a finite-index subgroup of $\La'$ that does by using the following holonomy map:
\begin{align*}
	\Upsilon:\La'&\to\Aut(C_*)\\
	\lambda&\mapsto\phi_{\lambda C_*,C_*}\circ\lambda
\end{align*}
Here $\Aut(C_*)$ is the group of cubical automorphisms of the chamber $C_*$.
Let's check that $\Upsilon$ is a homomorphism.
Recall that, since $\phi$ is $\La'$-invariant, we have
\begin{equation}\label{La'invariant}
	\phi_{\lambda C_1,\lambda C_2}=\lambda\circ\phi_{C_1,C_2}\circ\lambda^{-1}
\end{equation}
for all chambers $C_1,C_2$ and $\lambda\in\La'$.
Suppose $\lambda_1,\lambda_2\in\La'$.
We then have
\begin{align*}
	\Upsilon(\lambda_1)\circ\Upsilon(\lambda_2)&=\phi_{\lambda_1C_*,C_*}\circ\lambda_1\circ\phi_{\lambda_2C_*,C_*}\circ\lambda_2,\\
	&=\phi_{\lambda_1C_*,C_*}\circ\phi_{\lambda_1\lambda_2 C_*,\lambda_1 C_*}\circ\lambda_1\circ\lambda_2,&\text{by (\ref{La'invariant}),}\\
	&=\phi_{\lambda_1\lambda_2 C_*, C_*}\circ (\lambda_1\lambda_2),\\
	&=\Upsilon(\lambda_1\lambda_2).
\end{align*}

Let $\hat{\La}=\ker\Upsilon$, which has finite index in $\La'$ since $C_*$ is finite.
Given a vertex $v$ in a chamber $C$ and $\lambda\in\hat{\La}$, we can then compute
\begin{align*}
	\phi_{\lambda C,C_*}(\lambda v)&=\phi_{\lambda C_*,C_*}\circ\phi_{\lambda C,\lambda C_*}(\lambda v),\\
	&=\phi_{\lambda C_*,C_*}\circ\lambda\circ\phi_{C,C_*}\circ\lambda^{-1}(\lambda v),&\text{by (\ref{La'invariant}),}\\
	&=\phi_{C,C_*}(v), &\text{as }\lambda\in\ker\Upsilon.
\end{align*}
It follows from (\ref{t}) that $t(v)=t(\lambda v)$, hence $t$ is $\hat{\La}$-invariant.

We now turn to constructing a $\hat{\La}$-invariant atlas $\cA$ for $t$.
Let $v\in\Delta^0$ be a rank-1 vertex with $t(v)=\{i\}$.
We know from Remark \ref{remk:=Gi} that $\cC(v)$ has size $|G_i|$.
Construct $\cA_{v}$ by arbitrarily choosing a simply transitive action of $G_i$ on $\cC(v)$.
In turn, this determines $\cA_{v'}$ for all $v'\approx v$ by the identifications $\cC(v)\cong\cC(v')$.
For $\lambda\in\hat{\La}$ we define $\cA_{\lambda v}$ by
\begin{equation}\label{Alav}
\cA_{\lambda v}(\lambda C)=\lambda \cA_{v}(C),
\end{equation}
for $C\in\cC(v)$.
Any $\lambda\in\hat{\La}_{[v]}$ acts trivially on the first factor of the product decomposition $\cC([v])\cong\cC(v)\times\cC(j^\perp,C)$ by Lemma \ref{lem:trivfactor1} (where $t_\G(v)=\{j\}$ and $C\in\cC(v)$ is some chamber), so $\cA_{\lambda v}$ and $\cA_{v}$ coincide under the identification $\cC(v)\cong\cC(\lambda v)$.
The $\hat{\La}$-invariant atlas $\cA$ is constructed by repeating this argument for each $\hat{\La}$-orbit of level-equivalence classes of rank-1 vertices.
\end{proof}

Given typed atlases $(t,\cA)$ and $(t',\cA')$, and an automorphism $f\in\Aut_{\rk}(\Delta)$ with $f_*(t,\cA)=(t',\cA')$, we would like to interpret $f$ in terms of these atlases.
This motivates the following definition and subsequent three lemmas.

\begin{defn}($(t,\cA)$-word of a gallery)\\\label{defn:tAword}
	Let $(t,\cA)$ be a typed atlas.
	We describe a way of associating a letter $s$ in the alphabet 
	$$\mathfrak{A}:=\sqcup_{i\in I}(\{i\}\times (G_i-\{1\}))$$
	 to each ordered pair of adjacent chambers.
	 Indeed, if $C_1,C_2$ are adjacent chambers, and $v=\wedge(C_1, C_2)$ with $t(v)=\{i\}$, and $g\in G_i$ with $\cA_{v}(g)(C_1)=C_2$, then we associate the letter $(i,g)\in\mathfrak{A}$ to $(C_1,C_2)$.
	 Note that the element $g\in G_i$ exists and is unique because $\cA_{v}$ is a simply transitive action.
	 The \emph{$(t,\cA)$-word} of a gallery $(C_0,C_1,...,C_n)$ is the word $(s_1,...,s_n)$ on $\mathfrak{A}$ such that each $s_k$ is the letter associated to the pair $(C_{k-1},C_k)$.
\end{defn}

\begin{lem}\label{lem:wordtogallery}
	Let $(s_1,...,s_n)$ be a word on $\mathfrak{A}$, let $C\in\cC(\Delta)$ and let $(t,\cA)$ be a typed atlas.
	Then there exists a unique gallery $(C_0,C_1,...,C_n)$ with $C_0=C$ whose $(t,\cA)$-word is $(s_1,...,s_n)$.
\end{lem}
\begin{proof}
	If $s_k=(i_k,g_k)$, then the gallery is determined recursively by $C_0:=C$ and $C_k:=\cA_{v_k}(g_k)(C_{k-1})$, where $v_k$ is the unique vertex in $C_{k-1}$ with $t(v_k)=\{i_k\}$.
\end{proof}

\begin{lem}\label{lem:gallerytransfer}
Let $(t,\cA)$ and $(t',\cA')$ be typed atlases.
Let $G=(C_0,C_1,...,C_n)$ and $G'=(C'_0,C'_1,...,C'_n)$ be galleries, and suppose that the $(t,\cA)$-word of $G$ is the same as the $(t',\cA')$-word of $G'$.
Then the end-chamber $C'_n$ depends only on the end-chambers $C_0,C_n,C'_0$, not on the choice of galleries $G,G'$.
\end{lem}
\begin{proof}
The moves \ref{M1}--\ref{M3} from Lemma \ref{lem:moves} do not change the end-chambers of galleries, and any two galleries with the same end-chambers differ by a sequence of such moves.
So it suffices to show that if we modify the gallery $G$ by one of the moves \ref{M1}--\ref{M3}, then we can modify the gallery $G'$ by a corresponding move such that the $(t,\cA)$-word of $G$ remains the same as the $(t',\cA')$-word of $G'$.
We split into three cases according to which of the moves \ref{M1}--\ref{M3} we apply to $G$:
\begin{enumerate}[label={(M\arabic*)}]
	\item This move replaces a subgallery $(C_k)$ by $(C_k,\hat{C},C_k)$, where $C_k,\hat{C}$ are adjacent.
	It follows readily from Definition \ref{defn:tAword} that the $(t,\cA)$-word of $(C_k,\hat{C},C_k)$ is of the form $((i,g),(i,g^{-1}))$, and that there is a chamber $\hat{C}'$ such that the $(t',\cA')$-word of $(C'_k,\hat{C}',C'_k)$ is also $((i,g),(i,g^{-1}))$.
	Hence replacing $(C'_k)$ by $(C'_k,\hat{C}',C'_k)$ in $G'$ is an \ref{M1} move which ensures that the $(t,\cA)$-word of $G$ remains the same as the $(t',\cA')$-word of $G'$.
	
	\item This move replaces a subgallery $(C_k,C_{k+1})$ by $(C_k,\hat{C},C_{k+1})$, where $C_k,\hat{C},C_{k+1}$ are pairwise $j$-adjacent for some $j\in I$.
	In this case there exists a unique rank-1 vertex $v\in C_k\cap \hat{C}\cap C_{k+1}$, and $t_\G(v)=\{j\}$.
	Observe that $C_k,\hat{C},C_{k+1}\in\cC(v)$.
	If $t(v)=\{i\}$ then $s_{k+1}=(i,g)$ for some $g\in G_i$ with
	\begin{equation}\label{AgCk}
		\cA_{v}(g)(C_k)=C_{k+1}.
	\end{equation} 
	Similarly, the $(t,\cA)$-word of $(C_k,\hat{C},C_{k+1})$ takes the form $((i,g_1),(i,g_2))$ for $g_1,g_2\in G_i$ with
	\begin{equation}\label{Ag1Ck}
		\cA_{v}(g_1)(C_k)=\hat{C}\quad\text{and}\quad\cA_{v}(g_2)(\hat{C})=C_{k+1}.
	\end{equation}
	Let $v'\in C'_k$ be the vertex with $t'(v')=\{i\}$ and let $\hat{C}':=\cA'_{v'}(g_1)(C'_k)$.
	The action of $G_i$ on $\cC(v)$ is simply transitive, so it follows from (\ref{AgCk}) and (\ref{Ag1Ck}) that $g=g_2g_1$, hence
	\begin{align*}
		\cA'_{v'}(g_2)(\hat{C}')=\cA'_{v'}(g)(C'_k)=C'_{k+1}.
	\end{align*}
	Thus $(C'_k,\hat{C}',C'_{k+1})$ has $(t',\cA')$-word $((i,g_1),(i,g_2))$, and replacing $(C'_k,C'_{k+1})$ by $(C'_k,\hat{C}',C'_{k+1})$ in $G'$ is an \ref{M2} move which ensures that the $(t,\cA)$-word of $G$ remains the same as the $(t',\cA')$-word of $G'$.
	
	\item This move replaces a subgallery $(C_k,C_{k+1},C_{k+2})$ by $(C_k,\hat{C},C_{k+2})$, where $C_k,C_{k+1}$ and $\hat{C},C_{k+2}$ are $j_1$-adjacent, and $C_k,\hat{C}$ and $C_{k+1},C_{k+2}$ are $j_2$-adjacent, for adjacent $j_1,j_2\in I$.
	Let $v_1:=\wedge(C_k, C_{k+1})$ and $v_2:=\wedge(C_{k+1}, C_{k+2})$.
	So $t_\G(v_1)=\{j_1\}$ and $t_\G(v_2)=\{j_2\}$.
	Let $t(v_1)=\{i_1\}$ and $t(v_2)=\{i_2\}$.
	Lemma \ref{lem:ijadjacent} implies that $i_1,i_2$ are adjacent.
	The $(t,\cA)$-word for $(C_k,C_{k+1},C_{k+2})$ takes the form $((i_1,g_1),(i_2,g_2))$ for some $g_1\in G_{i_1}$ and $g_2\in G_{i_2}$.
	
	Now consider the corresponding subgallery $(C'_k,C'_{k+1},C'_{k+2})$ in $G'$, which has $((i_1,g_1),(i_2,g_2))$ as its $(t',\cA')$-word.
	This means that $t'(v'_1)=\{i_1\}$ and $t'(v'_1)=\{i_2\}$ for $v'_1:=\wedge(C'_k, C'_{k+1})$ and $v'_2:=\wedge(C'_{k+1}, C'_{k+2})$.
	Let $t_\G(v'_1)=\{j'_1\}$ and $t_\G(v'_2)=\{j'_2\}$.
	Lemma \ref{lem:ijadjacent} implies that $j'_1,j'_2$ are adjacent.
	By the product structure of $\{j'_1,j'_2\}$-chamber-residues (Lemma \ref{lem:product}), we deduce that there is a chamber $\hat{C}'$ that is $j'_2$-adjacent to $C'_k$ and $j'_1$-adjacent to $C'_{k+1}$. So the setup of the chambers $C'_k,C'_{k+1},C'_{k+2},\hat{C}'$ mimics the setup of the chambers $C_k,C_{k+1},C_{k+2},\hat{C}$, but with respect to $j'_1,j'_2$ instead of $j_1,j_2$.
	
	Putting $u_1:=\wedge(\hat{C},C_{k+2})$, we have $t_\G(v_1)=t_\G(u_1)=\{j_1\}$ and $v_1\approx u_1$. The actions $\cA_{v_1}$ and $\cA_{u_1}$ coincide under the identification $\cC(v_1)\cong\cC(u_1)$, and $C_k,C_{k+1}\mapsto \hat{C},C_{k+2}$ under this identification, so $\cA_{u_1}(g_1)(\hat{C})=C_{k+2}$.
	Therefore $(\hat{C},C_{k+2})$ also has $(t,\cA)$-word $((i_1,g_1))$.
	Running the same argument with the roles of $i_1$ and $i_2$ reversed, we deduce that $(C_1,\hat{C})$ has $(t,\cA)$-word $((i_2,g_2))$, so $(C_k,\hat{C},C_{k+2})$ has $(t,\cA)$-word $((i_2,g_2),(i_1,g_1))$.
	
	Running the same argument for the chambers $C'_k,C'_{k+1},C'_{k+2},\hat{C}'$, we discover that the $(t',\cA')$-word for $(C'_k,\hat{C}',C'_{k+2})$ is also $((i_2,g_2),(i_1,g_1))$.
	So replacing the subgallery $(C'_k,C'_{k+1},C'_{k+2})$ with $(C'_k,\hat{C}',C'_{k+2})$ in $G'$ is an \ref{M3} move which ensures that the $(t,\cA)$-word of $G$ remains the same as the $(t',\cA')$-word of $G'$.\qedhere
\end{enumerate}
\end{proof}

\begin{lem}\label{lem:fzips}
	Let $(t,\cA)$ and $(t',\cA')$ be typed atlases, and let $f\in\Aut_{\rk}(\Delta)$ with $f_*(t,\cA)=(t',\cA')$.
	Let $G=(C_0,C_1,...,C_n)$ and $G'=(C'_0,C'_1,...,C'_n)$ be galleries, and suppose that the $(t,\cA)$-word of $G$ is the same as the $(t',\cA')$-word of $G'$.
	If $fC_0=C'_0$ then $fC_k=C'_k$ for all $0\leq k\leq n$.
\end{lem}
\begin{proof}
	Let $(s_1,...,s_n)$ be the $(t,\cA)$-word of $G$. We prove the lemma by induction on $k$.
	The case $k=0$ is immediate, so suppose $k\geq 1$, and assume by the inductive hypothesis that $fC_{k-1}=C'_{k-1}$.
	Let $s_k=(i,g)$, $v:=\wedge(C_{k-1}, C_k)$ and $v':=\wedge(C'_{k-1}, C'_k)$.
	Then $t(v)=t'(v')=\{i\}$ and
	\begin{equation}\label{AA'}
		\cA_{v}(g)(C_{k-1})=C_k\quad\text{and}\quad\cA'_{v'}(g)(C'_{k-1})=C'_k.
	\end{equation}
	We know that $\{i\}=t(v)=t'(fv)$ since $f_*(t,\cA)=(t',\cA')$, and $fv\in C'_{k-1}$ because $fC_{k-1}=C'_{k-1}$, so we deduce that $fv=v'$. Hence it follows from $f_*(t,\cA)=(t',\cA')$ and (\ref{AA'}) that
	\begin{align*}
		fC_k&=f\cA_{v}(g)(C_{k-1})\\
		&=\cA'_{fv}(g)(fC_{k-1})\\
		&=\cA'_{v'}(g)(C'_{k-1})\\
		&=C'_k.\qedhere
	\end{align*}
\end{proof}

The final ingredient in the proof of Theorem \ref{thm:Delta} is the following proposition. This proposition, along with its proof, are analogous to \cite[Proposition 6.5]{Haglund06}.

\begin{prop}\label{prop:typedatlas}
	Let $(t,\cA)$ and $(t',\cA')$ be typed atlases.
	Let $C$ and $C'$ be chambers in $\Delta$ and let $f:C\to C'$ be the unique isomorphism such that $t'(fv)=t(v)$ for all vertices $v\in C$. Then $f$ admits a unique extension $\bar{f}\in\Aut_{\rk}(\Delta)$ such that $\bar{f}_*(t,\cA)=(t',\cA')$.
\end{prop}
\begin{proof}
Given a gallery $G=(C_0,C_1,...,C_n)$ with origin $C_0=C$, let $G'=(C'_0,C'_1,...,C'_n)$ be the gallery with origin $C'_0=C'$ whose $(t',\cA')$-word is the same as the $(t,\cA)$-word of $G$ -- such a gallery exists and is unique by Lemma \ref{lem:wordtogallery}.
For each $k$, the $(t,\cA)$-word of the subgallery $(C_0,C_1,...,C_k)$ is equal to the $(t',\cA')$-word of the subgallery $(C'_0,C'_1,...,C'_k)$, so it follows from Lemma \ref{lem:gallerytransfer} that the chamber $C'_k$ depends only on the chamber $C_k$, not on the choice of gallery $G$ containing $C_k$.
Every chamber in $\Delta$ is joined to $C$ by some gallery, hence we get a well-defined map $\varepsilon:\cC(\Delta)\to\cC(\Delta)$ by sending $C_k\mapsto C'_k$ for each pair of galleries $G,G'$ as above.
By Lemma \ref{lem:fttot'}, for each chamber $C_k$ there is a unique cubical isomorphism $h_k:C_k\to C'_k$ such that
\begin{equation}\label{thk}
t(v)=t'(h_k v)
\end{equation} 
for all vertices $v\in C_k$.
Note that $h_0=f$.
We want these maps to fit together to define a cubical map $\bar{f}:\Delta\to\Delta$, so we must show that these maps agree on intersections of chambers.
Let $v$ be a vertex in the intersection of chambers $\hat{C}_1$ and $\hat{C}_2$.
Let $G$ be a gallery as above with $\hat{C}_1=C_{m_1}$ and $\hat{C}_2=C_{m_2}$ for some $0\leq m_1\leq m_2\leq n$.
Our task is to show that $h_{m_1}v=h_{m_2} v$.
By Lemma \ref{lem:cCv}, we may choose $G$ so that $v\in C_k$ for $m_1\leq k\leq m_2$.
Therefore, it is enough to show that $h_{k-1}v=h_kv$ for each $m_1< k\leq m_2$.
Let $u:=\wedge(C_{k-1}, C_k)$ and $u':=\wedge(C'_{k-1}, C'_k)$, and let $s_k=(i,g)$.
So $t(u)=\{i\}=t'(u')$, and $h_{k-1}u=u'=h_ku$ by (\ref{thk}).
By Lemma \ref{lem:min}\ref{item:wCcapC'} we know that $u\leq v$, and $h_{k-1},h_k$ both preserve poset structure by Lemma \ref{lem:fttot'} so $u'\leq h_{k-1}v,h_kv\in C'_{k-1}\cap C'_k$. As $t'(h_{k-1}v)=t(v)=t'(h_kv)$, we conclude that $h_{k-1}v=h_kv$.

We now have a cubical map $\bar{f}:\Delta\to\Delta$, and it preserves rank because each of the maps $h_k$ do (Lemma \ref{lem:fttot'}).
We note that the inverse maps $h^{-1}_k:C'_k\to C_k$ for pairs of galleries $G,G'$ as above define an inverse to $\bar{f}$ (which is well-defined by the same argument we used for $\bar{f}$), so $\bar{f}\in\Aut_{\rk}(\Delta)$.
Furthermore, $\bar{f}_*(t,\cA)=(t',\cA')$ follows from the way we constructed $\bar{f}$ with the pairs of galleries $G,G'$.

It remains to prove that such $\bar{f}$ is unique.
Indeed if $\bar{f}\in\Aut_{\rk}(\Delta)$ is an extension of $f$ with $\bar{f}_*(t,\cA)=(t',\cA')$, and if $G,G'$ are galleries as above, then it follows from Lemma \ref{lem:fzips} that $\bar{f}C_k=C'_k$ for all $0\leq k\leq n$. 
Hence the action of $\bar{f}$ on $\cC(\Delta)$ induces the map $\varepsilon:\cC(\Delta)\to\cC(\Delta)$ defined above.
Finally, the restriction of $\bar{f}$ to a chamber $C_k$ (again for some choice of $G$ as above) must agree with $h_k$ because of (\ref{thk}) and the fact that $\bar{f}_*(t,\cA)=(t',\cA')$, so $\bar{f}$ is uniquely determined.
\end{proof}

We can now prove Theorem \ref{thm:Delta}.

\begin{proof}[Proof of Theorem \ref{thm:Delta}]
	The group $\G$ acts transitively on the chambers of $\Delta$ and preserves the standard typed atlas $(t_\G,\cA_\G)$ by Lemma \ref{lem:Gammatypedatlas}, so it follows from Proposition \ref{prop:typedatlas} that the $\Aut_{\rk}(\Delta)$-stabilizer of $(t_\G,\cA_\G)$ is $\G$.
	By Lemma \ref{lem:Latypedatlas} there is a typed atlas $(t,\cA)$ preserved by a finite-index subgroup $\hat{\La}<\La$.
	Applying Proposition \ref{prop:typedatlas} again produces an automorphism $g\in\Aut_{\rk}(\Delta)$ with $g_*(t,\cA)=(t_\G,\cA_\G)$, so $g\hat{\La}g^{-1}$ preserves the typed atlas $(t_\G,\cA_\G)$.
	It follows that $g\hat{\La}g^{-1}$ is a subgroup of $\G$, and it has finite index since it acts cocompactly on $\Delta$.
	Thus $\La$ and $\G$ are weakly commensurable in $\Aut(\Delta)$.
\end{proof}

\bigskip
\section{Quasi-isometric rigidity}\label{sec:QI}

In this section we prove the following theorem about the quasi-isometric rigidity of the graph product $\G$.

\theoremstyle{plain}
\newtheorem*{thm:QI}{Theorem \ref{thm:QI}}
\begin{thm:QI}
	Let $\G=\G(\cG,(G_i)_{i\in I})$ be a graph product.
	Suppose that $\cG$ is a finite generalized $m$-gon, with $m\geq3$, which is bipartite with respect to the partition $I=I_1\sqcup I_2$.
	Suppose that $d_1,d_2,p_1,p_2\geq2$ are integers such that every $i\in I_k$ ($k=1,2$) has degree $d_k$ and $|G_i|=p_k$.
	Then in each of the following cases
	\begin{enumerate}[label=(\roman*)]
		\item $2< d_1,d_2,p_1,p_2$,
		\item $p_1=p_2=2<d_1,d_2$,
		\item $d_1=d_2=2<p_1,p_2$,
	\end{enumerate}
	any finitely generated group quasi-isometric to $\G$ is abstractly commensurable with $\G$.
\end{thm:QI}

The idea is to show that the associated right-angled building $\Delta=\Delta(\cG,(G_i)_{i\in I})$ has the structure of a Fuchsian building in these cases (Definition \ref{defn:Fuchsian}), and then apply the following quasi-isometric rigidity theorem of Xie \cite[Corollary 1.4]{Xie06} (which built on work of Bourdon--Pajot \cite{BourdonPajot00}).

\begin{thm}
	Let $X$ be a locally finite Fuchsian building and suppose that $\Aut(X)$ contains a uniform lattice. Then any finitely generated group quasi-isometric to $X$ admits a proper and cocompact action on $X$.
\end{thm}

As a consequence, any group $\G'$ quasi-isometric to $\G$ acts properly and cocompactly on $\Delta$.
Since $\Delta$ is hyperbolic (as $\cG$ has no induced 4-cycles), such $\G'$ would have a finite-index subgroup which is a special uniform lattice in $\Aut(\Delta)$ \cite{Agol13}, so $\G'$ would be abstractly commensurable with $\G$ by Theorem \ref{thm:Delta}.
(In fact, for cases \ref{item:ii} and \ref{item:iii} we could use \cite[Theorem 1.7 and Theorem 1.3]{Haglund06} respectively in combination with \cite{Agol13} instead of Theorem \ref{thm:Delta} -- noting that the group of type preserving automorphisms has finite index in $\Aut(\Delta)$ in case \ref{item:iii}.)

In the rest of this section we show how the conditions of Theorem \ref{thm:QI} endow $\Delta$ with the structure of a Fuchsian building.
We follow a similar strategy to \cite{BoundsXie20}, who did exactly this for case \ref{item:ii} above (note that in this case $\G$ is a right-angled Coxeter group and $\Delta$ the associated Davis complex).
We remark that there are some other cases not covered by Theorem \ref{thm:QI} where $\Delta$ still has the structure of a Fuchsian building (which yield more examples of quasi-isometrically rigid groups), Theorem \ref{thm:QI} just deals with three of the simpler cases.

\begin{defn}(Fuchsian buildings)\\\label{defn:Fuchsian}
	Consider a compact convex polygon $R$ in the hyperbolic plane $\mathbb{H}^2$ whose angles are of the form $\pi/m$ for $2\leq m\in\mathbb{N}$. 
	Let $W$ be the Coxeter group generated by the reflections about the edges of $R$.	
	It is well known that the images of $R$ under $W$ form a tessellation of $\mathbb{H}^2$.
	If we label the vertices of $R$ cyclically by $1,2,...,k$, then we get a $W$-invariant labeling of the tessellation of $\mathbb{H}^2$. Let $A_R$ denote the obtained labeled 2-complex.
	
	Let $X$ be a connected 2-complex with vertices labeled by $1,2,...,k$, such that each 2-cell (called a \emph{chamber}) can be identified with $R$ via a label preserving isometry. Suppose that $X$ has a family of subcomplexes (called \emph{apartments}), each of which is isomorphic to $A_R$ as labeled 2-complexes. Then $X$ is a \emph{Fuchsian building} if it satisfies the following three properties:
	\begin{enumerate}
		\item\label{item:apart} Given any two chambers there is an apartment containing both.
		\item\label{item:mapapart} For any two apartments $A_1, A_2$ that share a chamber there is an isomorphism of labeled 2-complexes $f:A_1\to A_2$ that pointwise fixes $A_1\cap A_2$.
		\item\label{item:qi} There are integers $q_i\geq2$, $i=1,2,...,k$, such that each edge of $X$ with endpoints labeled by $i,i+1$ (mod $k$) is contained in exactly $q_i+1$ chambers.
	\end{enumerate}
\end{defn}

Properties \ref{item:apart} and \ref{item:mapapart} above can be replaced with local properties by using the following theorem, which is a special case of \cite[Corollary 2.4]{GaboriauPaulin01}.

\begin{thm}\label{thm:localF}
	Let $X$ be a simply connected 2-complex with all the assumptions from Definition \ref{defn:Fuchsian} except properties \ref{item:apart} and \ref{item:mapapart}.
	Then $X$ is a Fuchsian building (hence satisfies \ref{item:apart} and \ref{item:mapapart}) if the following properties hold:
	\begin{enumerate}[label=(\alph*)]
		\item\label{item:CAT1} The link of each vertex in $X$ is CAT(1).
		\item\label{item:embedcircle} Through any two points in the link of a vertex in $X$ there passes an isometrically embedded circle.
	\end{enumerate}
\end{thm}

Recall that a \emph{generalized $m$-gon} is a connected, bipartite, simplicial graph with diameter $m$ and girth $2m$ ($m\geq2$). Equivalently, it is a connected, bipartite, simplicial graph with the following two properties \cite[Theorem 1.1]{VanMaldeghem02}:
\begin{enumerate}
	\item Given any pair of edges there is a circuit of length $2m$ containing both.
	\item For two circuits $A_1,A_2$ of length $2m$ with non-empty intersection, there is an isomorphism $f:A_1\to A_2$ that pointwise fixes $A_1\cap A_2$.
\end{enumerate}

These definitions of generalized $m$-gon connect neatly with Theorem \ref{thm:localF} as follows.
If $X$ is the 2-complex from Theorem \ref{thm:localF}, and if a vertex $x\in X$ subtends angles $\pi/m$ in each of its incident 2-cells, then the link of $x$ satisfies properties \ref{item:CAT1} and \ref{item:embedcircle} above if and only if it is a generalized $m$-gon (with one edge in the link of $x$ for each 2-cell incident to $x$).

We are now ready to show that $\Delta$ has the structure of a Fuchsian building in the cases of Theorem \ref{thm:QI} (in fact now we no longer require $\cG$ to be finite).

\begin{prop}\label{prop:Fuchsian}
Let $\cG$ be a generalized $m$-gon, with $m\geq3$, which is bipartite with respect to the partition $I=I_1\sqcup I_2$, and let $(G_i)_{i\in I}$ be finite groups. 
Suppose that $d_1,d_2,p_1,p_2\geq2$ are integers such that every $i\in I_k$ ($k=1,2$) has degree $d_k$ and $|G_i|=p_k$.
Then in each of the following cases
\begin{enumerate}[label=(\roman*)]
	\item\label{item:i'} $2< d_1,d_2,p_1,p_2$,
	\item\label{item:ii'} $p_1=p_2=2<d_1,d_2$,
	\item\label{item:iii'} $d_1=d_2=2<p_1,p_2$,
\end{enumerate}
the right-angled building $\Delta=\Delta(\cG,(G_i)_{i\in I})$ has the structure of a Fuchsian building.
\end{prop}

\begin{proof}
	To exhibit a Fuchsian building structure on $\Delta$ we must identify the chambers (in the sense of Definition \ref{defn:Fuchsian}). To avoid ambiguity, we will refer to chambers in the sense of Definition \ref{defn:Fuchsian} as \emph{Fuchsian chambers} and refer to chambers in the sense of Definition \ref{defn:building} as \emph{cubical chambers}.
	
	The general form of (candidate) Fuchsian chambers in each of cases \ref{item:i}--\ref{item:iii} is depicted in Figure \ref{fig:Fchambers}, along with vertex labels corresponding to the typing map $t:\Delta^0\to\bar{N}$.
	Each Fuchsian chamber is given a hyperbolic metric with angles as indicated and side lengths as symmetric as possible.
	One can easily verify in each case that $\Delta$ has the structure of a 2-complex in which the 2-cells are the Fuchsian chambers (using for instance Lemmas \ref{lem:intchamres} and \ref{lem:min}).
	
	Next we must consistently label the vertices of the Fuchsian chambers cyclically by numbers $1,2,...,k$ (where each Fuchsian chamber is a $k$-gon). This is easy in cases \ref{item:i} and \ref{item:iii} because the vertices of the Fuchsian chambers have distinct labels coming from the typing map.
	For case \ref{item:ii} we note that the four vertices of the Fuchsian chamber are centers of cubical chambers, and each cubical chamber is of the form $C_{\gamma}$ for some $\gamma\in\G$, so we can label the vertices according to the homomorphism
	$$\eta:\G\to \mathbb{Z}/2\mathbb{Z}\times \mathbb{Z}/2\mathbb{Z},$$
	where $G_i<\G$ maps onto the $k$-th factor ($k=1,2$) if $i\in I_k$ (note that all the $G_i$ have order 2).
	Let's analyze how this labeling might look in a single Fuchsian chamber.
	If two vertices of the Fuchsian chamber are in cubical chambers $C_{\gamma_1},C_{\gamma_2}$, and if the side joining the vertices has midpoint of type $\{i\}$ with $i\in I_1$, then $\gamma_2^{-1}\gamma_1\in G_i$, so $\eta(\gamma_1),\eta(\gamma_2)$ differ only in their first coordinates.
	Whereas a midpoint of type $\{j\}$ with $j\in I_2$ implies that $\eta(\gamma_1),\eta(\gamma_2)$ differ only in their second coordinates.
	A possible labeling by $\eta$ is shown in blue in Figure \ref{fig:Fchambers}, and the above discussion implies that the other three possible labelings are obtained by reflecting the picture in the $x$- and/or $y$-axes.
	In all cases we get cyclic labelings of the Fuchsian chambers by composing with the map $(0,0),(1,0),(1,1),(0,1)\mapsto 1,2,3,4$.
	
	In Figure \ref{fig:Fchambers}, each side of a Fuchsian chamber is labeled (in red) by the number of Fuchsian chambers that contain it. These numbers are all greater than 2, so this verifies property \ref{item:qi} in Definition \ref{defn:Fuchsian}.
	
	Finally, for properties \ref{item:apart} and \ref{item:mapapart} in Definition \ref{defn:Fuchsian} we use Theorem \ref{thm:localF}. This involves the angles for the Fuchsian chambers.
	The vertices with angle $\pi/m$ are all centers of cubical chambers, so have links isomorphic to $\cG$.
	Meanwhile, it is straightforward to check that the vertices with right-angles all have links isomorphic to complete bipartite graphs, so these are generalized 2-gons.
	It follows that $\Delta$ satisfies properties \ref{item:CAT1} and \ref{item:embedcircle} from Theorem \ref{thm:localF}.
	
\end{proof}

\begin{figure}[H]
	\centering
	\scalebox{1}{
		\begin{tikzpicture}[auto,node distance=2cm,
			thick,every node/.style={circle,draw,fill,inner sep=0pt,minimum size=7pt},
			every loop/.style={min distance=2cm},
			hull/.style={draw=none},
			]
			\tikzstyle{label}=[draw=none,fill=none]
			\tikzstyle{a}=[isosceles triangle,sloped,allow upside down,shift={(0,-.05)},minimum size=3pt]
			
			\begin{scope}[shift={(-7,1)}]
			
			\draw ({8*(-sqrt(2)*sin(112.5)+1)},0) arc (0:-22.5:8);
			\draw (0,{8*(-sqrt(2)*sin(112.5)+1)}) arc (90:112.5:8);
			\draw ({8*(-sqrt(2)*sin(112.5)+1)},0)--(0,0)--	(0,{8*(-sqrt(2)*sin(112.5)+1)});
			
			\draw [domain=20:70] plot ({-8*sqrt(2)*sin(112.5)+8*cos(22.5)+0.5*cos(\x)},{-8*sin(22.5)+0.5*sin(\x)});
			\draw ({8*(-sqrt(2)*sin(112.5)+1)},0) rectangle ({8*(-sqrt(2)*sin(112.5)+1)+0.3},-.3);
			\draw (0,{8*(-sqrt(2)*sin(112.5)+1)}) rectangle (-.3,{8*(-sqrt(2)*sin(112.5)+1)+0.3});
			\draw (0,0) rectangle (-.3,-.3);
			
			\node[label] at (-2.5,-2.5){$\frac{\pi}{m}$};
			\node[label] at (-3.2,-3.2){$\emptyset$};
			\node[label] at (-2.6,.3){$\{i\}$};
			\node[label] at (.3,-2.6){$\{j\}$};
			\node[label] at (.3,.3){$\{i,j\}$};
			\node[label,red] at (.3,-1.2){$p_1$};
			\node[label,red] at (-1.2,.3){$p_2$};
			\node[label,red] at (-2.8,-1.2){$d_1$};
			\node[label,red] at (-1.2,-2.8){$d_2$};

			\end{scope}
		
		\begin{scope}[scale=.8]
		\begin{scope}
\draw ({8*(-sqrt(2)*sin(112.5)+1)},0) arc (0:-22.5:8);
\draw (0,{8*(-sqrt(2)*sin(112.5)+1)}) arc (90:112.5:8);
\draw ({8*(-sqrt(2)*sin(112.5)+1)},0)--(0,0)--	(0,{8*(-sqrt(2)*sin(112.5)+1)});

\draw [domain=20:70] plot ({-8*sqrt(2)*sin(112.5)+8*cos(22.5)+0.5*cos(\x)},{-8*sin(22.5)+0.5*sin(\x)});

\node[label] at (-2.5,-2.5){$\frac{\pi}{m}$};
\node[label] at (-3.2,-3.2){$\emptyset$};
		\end{scope}
			\begin{scope}[rotate=90]
		\draw ({8*(-sqrt(2)*sin(112.5)+1)},0) arc (0:-22.5:8);
		\draw (0,{8*(-sqrt(2)*sin(112.5)+1)}) arc (90:112.5:8);
		\draw ({8*(-sqrt(2)*sin(112.5)+1)},0)--(0,0)--	(0,{8*(-sqrt(2)*sin(112.5)+1)});
		
		\draw [domain=20:70] plot ({-8*sqrt(2)*sin(112.5)+8*cos(22.5)+0.5*cos(\x)},{-8*sin(22.5)+0.5*sin(\x)});
		
		\node[label] at (-2.5,-2.5){$\frac{\pi}{m}$};
		\node[label] at (-3.2,-3.2){$\emptyset$};
	\end{scope}
		\begin{scope}[rotate=180]
	\draw ({8*(-sqrt(2)*sin(112.5)+1)},0) arc (0:-22.5:8);
	\draw (0,{8*(-sqrt(2)*sin(112.5)+1)}) arc (90:112.5:8);
	\draw ({8*(-sqrt(2)*sin(112.5)+1)},0)--(0,0)--	(0,{8*(-sqrt(2)*sin(112.5)+1)});
	
	\draw [domain=20:70] plot ({-8*sqrt(2)*sin(112.5)+8*cos(22.5)+0.5*cos(\x)},{-8*sin(22.5)+0.5*sin(\x)});
	
	\node[label] at (-2.5,-2.5){$\frac{\pi}{m}$};
	\node[label] at (-3.2,-3.2){$\emptyset$};
\end{scope}
		\begin{scope}[rotate=270]
	\draw ({8*(-sqrt(2)*sin(112.5)+1)},0) arc (0:-22.5:8);
	\draw (0,{8*(-sqrt(2)*sin(112.5)+1)}) arc (90:112.5:8);
	\draw ({8*(-sqrt(2)*sin(112.5)+1)},0)--(0,0)--	(0,{8*(-sqrt(2)*sin(112.5)+1)});
	
	\draw [domain=20:70] plot ({-8*sqrt(2)*sin(112.5)+8*cos(22.5)+0.5*cos(\x)},{-8*sin(22.5)+0.5*sin(\x)});
	
	\node[label] at (-2.5,-2.5){$\frac{\pi}{m}$};
	\node[label] at (-3.2,-3.2){$\emptyset$};
\end{scope}

\node[label] at (-2.8,0){$\{i\}$};
\node[label] at (2.8,0){$\{i\}$};
\node[label] at (0,-2.8){$\{j\}$};
\node[label] at (0,2.8){$\{j\}$};
\node[label] at (.5,.3){$\{i,j\}$};
\node[label,red] at (-2.8,-1.2){$d_1$};
\node[label,red] at (-2.8,1.2){$d_1$};
\node[label,red] at (2.8,-1.2){$d_1$};
\node[label,red] at (2.8,1.2){$d_1$};
\node[label,red] at (-1.2,-2.8){$d_2$};
\node[label,red] at (1.2,-2.8){$d_2$};
\node[label,red] at (-1.2,2.8){$d_2$};
\node[label,red] at (1.2,2.8){$d_2$};

\node[label,blue] at (3.2,3.7){$(1,1)$};
\node[label,blue] at (-3.2,3.7){$(1,0)$};
\node[label,blue] at (3.2,-3.7){$(0,1)$};
\node[label,blue] at (-3.2,-3.7){$(0,0)$};

\end{scope}

\begin{scope}[shift={(-7,-7)},scale=1.4]
\begin{scope}
	\draw [domain=-15:15] plot ({4*(-1/cos(15)-sqrt(2)*tan(15)+cos(\x))},{4*sin(\x)});
	\draw ({4*(1-1/cos(15)-sqrt(2)*tan(15))},0)--(0,0);
	\draw ({4*(-1/cos(15)-sqrt(2)*tan(15)+cos(12))},{4*sin(12)})--
	({4*(-1/cos(15)-sqrt(2)*tan(15)+1.055*cos(12.1))},{4.22*sin(12.1)});
	\draw ({4*(-1/cos(15)-sqrt(2)*tan(15)+cos(-12))},{4*sin(-12)})--
	({4*(-1/cos(15)-sqrt(2)*tan(15)+1.055*cos(-12.1))},{4.22*sin(-12.1)});
\end{scope}
\begin{scope}[rotate=60]
	\draw [domain=-15:15] plot ({4*(-1/cos(15)-sqrt(2)*tan(15)+cos(\x))},{4*sin(\x)});
	\draw ({4*(1-1/cos(15)-sqrt(2)*tan(15))},0)--(0,0);
	\draw ({4*(-1/cos(15)-sqrt(2)*tan(15)+cos(12))},{4*sin(12)})--
	({4*(-1/cos(15)-sqrt(2)*tan(15)+1.055*cos(12.1))},{4.22*sin(12.1)});
	\draw ({4*(-1/cos(15)-sqrt(2)*tan(15)+cos(-12))},{4*sin(-12)})--
	({4*(-1/cos(15)-sqrt(2)*tan(15)+1.055*cos(-12.1))},{4.22*sin(-12.1)});
\end{scope}
\begin{scope}[rotate=120]
	\draw [domain=-15:15] plot ({4*(-1/cos(15)-sqrt(2)*tan(15)+cos(\x))},{4*sin(\x)});
	\draw ({4*(1-1/cos(15)-sqrt(2)*tan(15))},0)--(0,0);
	\draw ({4*(-1/cos(15)-sqrt(2)*tan(15)+cos(12))},{4*sin(12)})--
	({4*(-1/cos(15)-sqrt(2)*tan(15)+1.055*cos(12.1))},{4.22*sin(12.1)});
	\draw ({4*(-1/cos(15)-sqrt(2)*tan(15)+cos(-12))},{4*sin(-12)})--
	({4*(-1/cos(15)-sqrt(2)*tan(15)+1.055*cos(-12.1))},{4.22*sin(-12.1)});
\end{scope}
\begin{scope}[rotate=180]
	\draw [domain=-15:15] plot ({4*(-1/cos(15)-sqrt(2)*tan(15)+cos(\x))},{4*sin(\x)});
	\draw ({4*(1-1/cos(15)-sqrt(2)*tan(15))},0)--(0,0);
	\draw ({4*(-1/cos(15)-sqrt(2)*tan(15)+cos(12))},{4*sin(12)})--
	({4*(-1/cos(15)-sqrt(2)*tan(15)+1.055*cos(12.1))},{4.22*sin(12.1)});
	\draw ({4*(-1/cos(15)-sqrt(2)*tan(15)+cos(-12))},{4*sin(-12)})--
	({4*(-1/cos(15)-sqrt(2)*tan(15)+1.055*cos(-12.1))},{4.22*sin(-12.1)});
\end{scope}
\begin{scope}[rotate=240]
	\draw [domain=-15:15] plot ({4*(-1/cos(15)-sqrt(2)*tan(15)+cos(\x))},{4*sin(\x)});
	\draw ({4*(1-1/cos(15)-sqrt(2)*tan(15))},0)--(0,0);
	\draw ({4*(-1/cos(15)-sqrt(2)*tan(15)+cos(12))},{4*sin(12)})--
	({4*(-1/cos(15)-sqrt(2)*tan(15)+1.055*cos(12.1))},{4.22*sin(12.1)});
	\draw ({4*(-1/cos(15)-sqrt(2)*tan(15)+cos(-12))},{4*sin(-12)})--
	({4*(-1/cos(15)-sqrt(2)*tan(15)+1.055*cos(-12.1))},{4.22*sin(-12.1)});
\end{scope}
\begin{scope}[rotate=300]
	\draw [domain=-15:15] plot ({4*(-1/cos(15)-sqrt(2)*tan(15)+cos(\x))},{4*sin(\x)});
	\draw ({4*(1-1/cos(15)-sqrt(2)*tan(15))},0)--(0,0);
	\draw ({4*(-1/cos(15)-sqrt(2)*tan(15)+cos(12))},{4*sin(12)})--
	({4*(-1/cos(15)-sqrt(2)*tan(15)+1.055*cos(12.1))},{4.22*sin(12.1)});
	\draw ({4*(-1/cos(15)-sqrt(2)*tan(15)+cos(-12))},{4*sin(-12)})--
	({4*(-1/cos(15)-sqrt(2)*tan(15)+1.055*cos(-12.1))},{4.22*sin(-12.1)});
\end{scope}

\node[label] at (-.3,.2){$\emptyset$};
\node[label] at (-1.9,0){$\{i_1\}$};
\node[label] at (1,1.6){$\{i_2\}$};
\node[label] at (1,-1.6){$\{i_3\}$};
\node[label] at (-1,1.6){$\{j_1\}$};
\node[label] at (1.9,0){$\{j_2\}$};
\node[label] at (-1,-1.6){$\{j_3\}$};
\node[label] at (-2.2,1.2){$\{i_1,j_1\}$};
\node[label] at (0,2.2){$\{i_2,j_1\}$};
\node[label] at (2.2,1.2){$\{i_2,j_2\}$};
\node[label] at (2.2,-1.2){$\{i_3,j_2\}$};
\node[label] at (0,-2.2){$\{i_3,j_3\}$};
\node[label] at (-2.2,-1.2){$\{i_1,j_3\}$};

\node[label,red] at (-1.9,.5){$p_1$};
\node[label,red] at (-1.9,-.5){$p_1$};
\node[label,red] at (.6,1.8){$p_1$};
\node[label,red] at (.6,-1.8){$p_1$};
\node[label,red] at (1.4,1.35){$p_1$};
\node[label,red] at (1.4,-1.35){$p_1$};
\node[label,red] at (1.9,.5){$p_2$};
\node[label,red] at (1.9,-.5){$p_2$};
\node[label,red] at (-.6,1.8){$p_2$};
\node[label,red] at (-.6,-1.8){$p_2$};
\node[label,red] at (-1.4,1.35){$p_2$};
\node[label,red] at (-1.4,-1.35){$p_2$};

\end{scope}

\node[label] at (-11,3){$(i)$};
\node[label] at (-3.5,3){$(ii)$};
\node[label] at (-11,-4){$(iii)$};
		\end{tikzpicture}
	}
	\caption{The general forms of Fuchsian chambers in each of cases \ref{item:i}--\ref{item:iii} from Proposition \ref{prop:Fuchsian}.
	The vertices are labeled according to the typing map $t:\Delta^0\to\bar{N}$, with $i,i_l\in I_1$ and $j,j_l\in I_2$.
	We assume that $\cG$ is a hexagon in case \ref{item:iii} (in general it will be a $2m$-gon).
	In case \ref{item:ii} the vertex of the Fuchsian chamber in the cubical chamber $C_{\gamma}$ is labeled (in blue) by $\eta(\gamma)$.
    Each side of a Fuchsian chamber is labeled (in red) by the number of Fuchsian chambers that contain it.
    Each Fuchsian chamber is given a hyperbolic metric with angles as indicated.
    }\label{fig:Fchambers}
\end{figure}
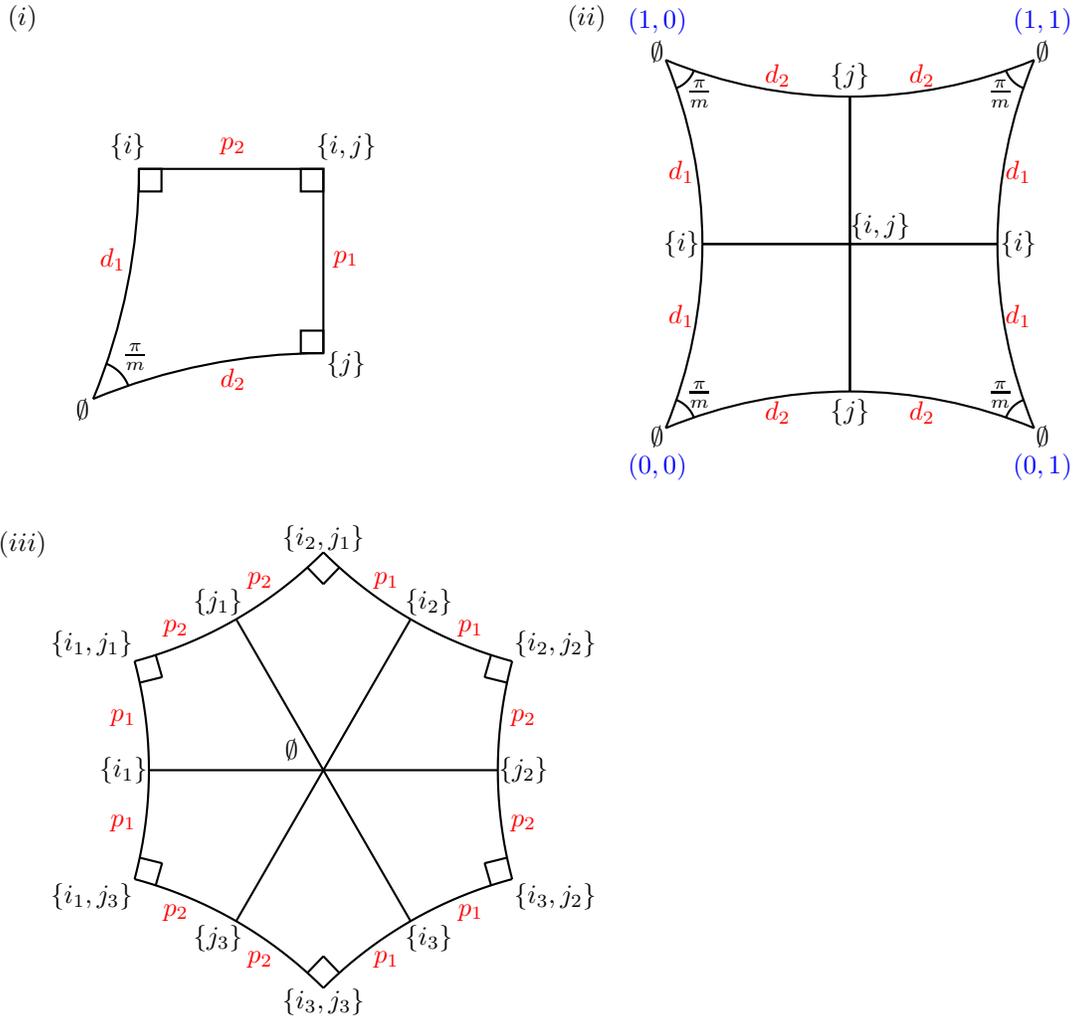

\bibliographystyle{alpha}
\bibliography{Ref}

\begin{thebibliography}{Woo23}

\bibitem[Ago13]{Agol13}
Ian Agol.
\newblock The virtual {H}aken conjecture.
\newblock {\em Doc. Math.}, 18:1045--1087, 2013.
\newblock With an appendix by Agol, Daniel Groves, and Jason Manning.

\bibitem[BM97]{BurgerMozes97}
Marc Burger and Shahar Mozes.
\newblock Finitely presented simple groups and products of trees.
\newblock {\em C. R. Acad. Sci. Paris S\'{e}r. I Math.}, 324(7):747--752, 1997.

\bibitem[BP00]{BourdonPajot00}
Marc Bourdon and Herv\'{e} Pajot.
\newblock Rigidity of quasi-isometries for some hyperbolic buildings.
\newblock {\em Comment. Math. Helv.}, 75(4):701--736, 2000.

\bibitem[BS22]{BridsonShepherd22}
Martin~R. Bridson and Sam Shepherd.
\newblock Leighton's theorem : extensions, limitations and quasitrees.
\newblock {\em Algebr. Geom. Topol.}, 22(2):881--917, 2022.

\bibitem[BX20]{BoundsXie20}
Jordan Bounds and Xiangdong Xie.
\newblock Quasi-isometric rigidity of a class of right-angled {C}oxeter groups.
\newblock {\em Proc. Amer. Math. Soc.}, 148(2):553--568, 2020.

\bibitem[Cap15]{Caprace15}
Pierre-Emmanuel Caprace.
\newblock Erratum to ``{B}uildings with isolated subspaces and relatively
  hyperbolic {C}oxeter groups'' [ {MR}2665193].
\newblock {\em Innov. Incidence Geom.}, 14:77--79, 2015.

\bibitem[Cha07]{Charney07}
Ruth Charney.
\newblock An introduction to right-angled {A}rtin groups.
\newblock {\em Geom. Dedicata}, 125:141--158, 2007.

\bibitem[Dav83]{Davis83}
Michael~W. Davis.
\newblock Groups generated by reflections and aspherical manifolds not covered
  by {E}uclidean space.
\newblock {\em Ann. of Math. (2)}, 117(2):293--324, 1983.

\bibitem[Dav98]{Davis98}
Michael~W. Davis.
\newblock Buildings are {${\rm CAT}(0)$}.
\newblock In {\em Geometry and cohomology in group theory ({D}urham, 1994)},
  volume 252 of {\em London Math. Soc. Lecture Note Ser.}, pages 108--123.
  Cambridge Univ. Press, Cambridge, 1998.

\bibitem[Dav08]{Davis08}
Michael~W. Davis.
\newblock {\em The geometry and topology of {C}oxeter groups}, volume~32 of
  {\em London Mathematical Society Monographs Series}.
\newblock Princeton University Press, Princeton, NJ, 2008.

\bibitem[DJ00]{DavisJanuszkiewicz00}
Michael~W. Davis and Tadeusz Januszkiewicz.
\newblock Right-angled {A}rtin groups are commensurable with right-angled
  {C}oxeter groups.
\newblock {\em J. Pure Appl. Algebra}, 153(3):229--235, 2000.

\bibitem[DK18]{DrutuKapovich18}
Cornelia Dru\c{t}u and Michael Kapovich.
\newblock {\em Geometric group theory}, volume~63 of {\em American Mathematical
  Society Colloquium Publications}.
\newblock American Mathematical Society, Providence, RI, 2018.
\newblock With an appendix by Bogdan Nica.

\bibitem[DS05]{DrutuSapir05}
Cornelia Dru\c{t}u and Mark Sapir.
\newblock Tree-graded spaces and asymptotic cones of groups.
\newblock {\em Topology}, 44(5):959--1058, 2005.
\newblock With an appendix by Denis Osin and Mark Sapir.

\bibitem[FH64]{FeitHigman64}
Walter Feit and Graham Higman.
\newblock The nonexistence of certain generalized polygons.
\newblock {\em J. Algebra}, 1:114--131, 1964.

\bibitem[Gen17]{Genevois17}
Anthony Genevois.
\newblock Cubical-like geometry of quasi-median graphs and applications to
  geometric group theory, 2017.

\bibitem[GP01]{GaboriauPaulin01}
Damien Gaboriau and Fr\'{e}d\'{e}ric Paulin.
\newblock Sur les immeubles hyperboliques.
\newblock {\em Geom. Dedicata}, 88(1-3):153--197, 2001.

\bibitem[Gre90]{Green90}
Elisabeth~Ruth Green.
\newblock Graph products of groups.
\newblock 1990.

\bibitem[Hag06]{Haglund06}
Fr\'{e}d\'{e}ric Haglund.
\newblock Commensurability and separability of quasiconvex subgroups.
\newblock {\em Algebr. Geom. Topol.}, 6:949--1024, 2006.

\bibitem[Hag08]{Haglund08}
Fr\'{e}d\'{e}ric Haglund.
\newblock Finite index subgroups of graph products.
\newblock {\em Geom. Dedicata}, 135:167--209, 2008.

\bibitem[Hua18]{Huang18}
Jingyin Huang.
\newblock Commensurability of groups quasi-isometric to {RAAG}s.
\newblock {\em Invent. Math.}, 213(3):1179--1247, 2018.

\bibitem[HW08]{HaglundWise08}
Fr\'{e}d\'{e}ric Haglund and Daniel~T. Wise.
\newblock Special cube complexes.
\newblock {\em Geom. Funct. Anal.}, 17(5):1551--1620, 2008.

\bibitem[Lei82]{Leighton82}
Frank~Thomson Leighton.
\newblock Finite common coverings of graphs.
\newblock {\em J. Combin. Theory Ser. B}, 33(3):231--238, 1982.

\bibitem[Mei96]{Meier96}
John Meier.
\newblock When is the graph product of hyperbolic groups hyperbolic?
\newblock {\em Geom. Dedicata}, 61(1):29--41, 1996.

\bibitem[Mos73]{Mostow73}
G.~D. Mostow.
\newblock {\em Strong rigidity of locally symmetric spaces}.
\newblock Annals of Mathematics Studies, No. 78. Princeton University Press,
  Princeton, N.J.; University of Tokyo Press, Tokyo, 1973.

\bibitem[Mou88]{Moussong88}
Gabor Moussong.
\newblock {\em Hyperbolic {C}oxeter groups}.
\newblock ProQuest LLC, Ann Arbor, MI, 1988.
\newblock Thesis (Ph.D.)--The Ohio State University.

\bibitem[Rey23]{Oregonreyes23}
Eduardo Reyes.
\newblock On cubulated relatively hyperbolic groups.
\newblock {\em Geom. Topol.}, 27(2):575--640, 2023.

\bibitem[Ron89]{Ronan89}
Mark Ronan.
\newblock {\em Lectures on buildings}, volume~7 of {\em Perspectives in
  Mathematics}.
\newblock Academic Press, Inc., Boston, MA, 1989.

\bibitem[She22]{Shepherd22}
Sam Shepherd.
\newblock Two generalisations of {L}eighton's theorem.
\newblock {\em Groups Geom. Dyn.}, 16(3):743--778, 2022.
\newblock With an appendix by Giles Gardam and Daniel J. Woodhouse.

\bibitem[Sto19]{Stover19}
Matthew Stover.
\newblock Lattices in {${\rm PU}(n,1)$} that are not profinitely rigid.
\newblock {\em Proc. Amer. Math. Soc.}, 147(12):5055--5062, 2019.

\bibitem[VM02]{VanMaldeghem02}
Hendrik Van~Maldeghem.
\newblock An introduction to generalized polygons.
\newblock In {\em Tits buildings and the model theory of groups
  ({W}\"{u}rzburg, 2000)}, volume 291 of {\em London Math. Soc. Lecture Note
  Ser.}, pages 23--57. Cambridge Univ. Press, Cambridge, 2002.

\bibitem[Wis96]{Wise96}
D.T. Wise.
\newblock {\em Non-positively curved squared complexes, aperiodic tilings, and
  non-residually finite groups}.
\newblock PhD thesis, Princeton University, 1996.

\bibitem[Wis06]{Wise06}
Daniel~T. Wise.
\newblock Subgroup separability of the figure 8 knot group.
\newblock {\em Topology}, 45(3):421--463, 2006.

\bibitem[Wis12]{WiseRiches}
Daniel~T. Wise.
\newblock {\em From riches to raags: 3-manifolds, right-angled {A}rtin groups,
  and cubical geometry}, volume 117 of {\em CBMS Regional Conference Series in
  Mathematics}.
\newblock Published for the Conference Board of the Mathematical Sciences,
  Washington, DC; by the American Mathematical Society, Providence, RI, 2012.

\bibitem[Woo21]{Woodhouse21b}
Daniel~J. Woodhouse.
\newblock Revisiting {L}eighton's theorem with the {H}aar measure.
\newblock {\em Math. Proc. Cambridge Philos. Soc.}, 170(3):615--623, 2021.

\bibitem[Woo23]{Woodhouse23}
Daniel~J. Woodhouse.
\newblock Leighton's theorem and regular cube complexes.
\newblock {\em Algebr. Geom. Topol.}, 23(7):3395--3415, 2023.

\bibitem[Xie06]{Xie06}
Xiangdong Xie.
\newblock Quasi-isometric rigidity of {F}uchsian buildings.
\newblock {\em Topology}, 45(1):101--169, 2006.

\end{thebibliography}

\end{document}